\documentclass[11pt,reqno]{amsart}

\usepackage[english]{babel}
\usepackage[T1]{fontenc}
\usepackage[utf8]{inputenc}
\usepackage{csquotes}
\usepackage{tikz,fp,ifthen}
\usetikzlibrary{matrix}
\pgfdeclarelayer{background} %%%% creating two layers for tikz pictures, also line below
\pgfsetlayers{background,main} %%
\usepackage{graphicx}
\usepackage[small,normal]{caption}
\usepackage[title,titletoc,page]{appendix}
\usepackage[a4paper,left=3cm,right=3cm,top=2cm,bottom=4cm,bindingoffset=5mm]{geometry}
\usepackage[colorlinks=true,linkcolor=blue,urlcolor=black,linktocpage=true,citecolor=black]{hyperref}
\usepackage[all]{hypcap}
\usepackage{setspace}
\usepackage{xcolor}

\usepackage{aligned-overset}
\usepackage{amsfonts}
\usepackage{amsmath}
\usepackage{amssymb}
\usepackage{amsthm}
\usepackage{dsfont}
\usepackage{float}
\usepackage{mathrsfs}
\usepackage{mathtools}
\mathtoolsset{showonlyrefs}

\usepackage{thmtools}
\usepackage[shortlabels]{enumitem}

\usepackage{graphicx}
\usepackage{mathscinet}
\usepackage[backend=biber,sorting=nty,natbib=true,isbn=false,doi=false,url=false,maxbibnames=5,style=ieee]{biblatex}
\usepackage{version}
\usepackage[obeyFinal]{todonotes}
\usepackage{cancel}
\DeclareMathOperator{\R}{\mathbb{R}}
\DeclareMathOperator{\N}{\mathbb{N}}

\DeclareMathOperator{\eps}{\varepsilon}

\numberwithin{equation}{section}

\newtheorem{theorem}{Theorem}[section]
\newtheorem{lemma}[theorem]{Lemma}
\newtheorem{corollary}[theorem]{Corollary}
\newtheorem{proposition}[theorem]{Proposition}

\theoremstyle{definition}
\newtheorem{definition}[theorem]{Definition}

\newtheorem{remark}[theorem]{Remark}
\newtheorem{assumption}[theorem]{Assumption}

\newtheorem{example}[theorem]{Example}

\usepackage{bbm}
\usepackage[normalem]{ulem}
\usepackage{cancel}

\newcommand{\cB}{\mathcal{B}}

\newcommand{\cD}{\mathcal{D}}

\newcommand{\cH}{\mathcal{H}}
\newcommand{\cI}{\mathcal{I}}

\newcommand{\cM}{\mathcal{M}}

\newcommand{\cX}{\mathcal{X}}

\newcommand{\bA}{\mathbb{A}}
\newcommand{\bB}{\mathbb{B}}

\newcommand{\bH}{\mathbb{H}}

\newcommand{\bN}{\mathbb{N}}

\newcommand{\bR}{\mathbb{R}}

\newcommand{\bONE}{\mathbbm{1}}

\newcommand{\bfz}{\mathbf{z}}

\newcommand{\dd}{ \mathrm{d}}
\def\eps{\varepsilon}

\DeclareMathOperator{\Tr}{Tr}

\newcommand{\vn}[1]{\left| \! \left| #1\right| \!\right|}

\newcommand{\ip}[2]{\left\langle #1,#2 \right\rangle}

\newcommand{\ri}[1]{{\color{teal}#1}}

\DeclareRobustCommand{\rchi}{{\mathpalette\irchi\relax}}
\newcommand{\irchi}[2]{\raisebox{\depth}{$#1\chi$}} % inner command, used by \rchi

\newcommand{\ssup}[1]{\left\lceil #1 \right\rceil}
\newcommand{\iinf}[1]{\left\lfloor #1 \right\rfloor}

\title[A comparison principle based on couplings]{A comparison principle based on couplings of partial integro-differential operators}

\author{Serena Della Corte}
\address{Delft Institute of Applied Mathematics, Delft University of Technology, The Netherlands}
\email{s.dellacorte@tudelft.nl}

\author{Fabian Fuchs}
\address{Center for Mathematical Economics, Bielefeld University, Germany}
\email{fabian.fuchs@uni-bielefeld.de}

\author{Richard C. Kraaij}
\address{Delft Institute of Applied Mathematics, Delft University of Technology, The Netherlands}
\email{r.c.kraaij@tudelft.nl}

\author{Max Nendel}
\address{Center for Mathematical Economics, Bielefeld University, Germany}
\email{max.nendel@uni-bielefeld.de}

\thanks{This work was funded by the Deutsche Forschungsgemeinschaft (DFG, German Research Foundation) -- SFB 1283/2 2021 -- 317210226 and by The Netherlands Organisation for Scientific Research (NWO), grant number 613.009.148}

\date{\today}

\begin{document}

\begin{abstract}
    This paper is concerned with a comparison principle for viscosity solutions to Hamilton--Jacobi (HJ), --Bellman (HJB), and --Isaacs (HJI) equations for general classes of partial integro-differential operators.
    Our approach innovates in three ways:\ (1) We reinterpret the classical doubling-of-variables method in the context of second-order equations by casting the Ishii--Crandall Lemma into a test function framework. This adaptation allows us to effectively handle non-local integral operators, such as those associated with Lévy processes. (2) We translate the key estimate on the difference of Hamiltonians in terms of an adaptation of the probabilistic notion of couplings, providing a unified approach that applies to differential, difference, and integral operators. (3) We strengthen the sup-norm contractivity resulting from the comparison principle to one that encodes continuity in the strict topology.
    We apply our theory to a variety of examples, in particular, to second-order differential operators and, more generally, generators of spatially inhomogeneous Lévy processes.\smallskip
    
     \smallskip
    
\end{abstract}

\maketitle

%%%%%%%%%%%%%%%%%%%%%%%%%%%%%%%%%%%%%

\section{Introduction}

In this work, we provide a new perspective on comparison principles for viscosity solutions to the Hamilton--Jacobi equation
\begin{equation} \label{eqn:HJ_intro}
    f - \lambda H f = h, \qquad \lambda > 0,\, h \in C_b(\bR^q),
\end{equation}
for Hamiltonians $H$ of the type
\begin{multline} \label{eqn:operator_linear_intro}
    H f(x) =\ip{b(x)}{\nabla f (x)} +  \frac{1}{2} \Tr\left(\Sigma\Sigma^T (x) D^2 f(x) \right) \\ 
    + \int \left[f(x+\bfz) - f(x) - \rchi_{B_1(0)} (\bfz) \ip{\bfz}{\nabla f(x)}\right] \mu_x(\dd \bfz) + \cH(\nabla f(x))
\end{multline}
and, more generally, for those in Bellman and Isaacs form
\begin{align*}
    H f(x) & = \sup_{\theta \in \Theta}\, \{H_\theta f(x) - \cI(x,\theta)\}\quad\text{and} \\
    H f(x) & = \sup_{\theta_1 \in \Theta_1} \inf_{\theta_2 \in \Theta_2} \{H_{\theta_1,\theta_2} f(x) - \cI(x,\theta_1,\theta_2)\},
\end{align*}
with $H_\theta$ and $H_{\theta_1,\theta_2}$ as in \eqref{eqn:operator_linear_intro} but with $\theta$ and $(\theta_1,\theta_2)$ dependent coefficients, respectively, and an appropriate cost functional $\cI$.

Motivated by convex Hamiltonians, for which no unique classical or weak solutions exist in general, \cite{CrLi83} introduced the notion of viscosity solutions. The seminal works \cite{Li83, CrEvLi84, Is84, Is86, CrIsLi87} explore this framework for first-order equations.

Most modern comparison proofs for operators containing second-order terms are based on results of \cite{jensen.1988b,jensen1988maximum}. Using then recent advances for generalized differentials, \cite{CrIs90} provided what is nowadays known as the Crandall--Ishii Lemma. An overview over uniqueness results for viscosity solutions to degenerate elliptic equations is given in the User's Guide \cite{CIL92}.  

The treatment of non-local operators was initially motivated by problems in optimal control theory; see \cite{MR0861089,MR1116853,alvarez.1996} for early examples with non-local operators.
The work \cite{BaIm08} gives a non-local version of the Crandall--Ishii Lemma by adapting the original procedure in \cite{CIL92}, and %, notably splitting the non-local term into a small-jump and large-jump regime reminiscent of the theory of L\'{e}vy measures. 
 \cite{Ho16} extends these results to unbounded solutions.
We also refer to \cite{FaGoSw17} for an overview of the Hilbertian setting, \cite{bdkn} for comparison principles for convex monotone semigroups on spaces of continuous functions, to \cite{MR4278437} for the classical well-posedness of convex Cauchy problems on $L^p$, to \cite{hu.2021} for a comparison principle in the framework of $G$-L\'evy processes, and to \cite{Ber2023} for a comparison principle for HJB equations on the set of probability measures.

Our approach and our main results, Theorem \ref{th:comparison_HJI} and Corollary \ref{cor:HJ_HJB}, innovate upon classical comparison principles in the following three ways: 
\begin{enumerate}[(1)]
    \item We reinterpret the classical doubling-of-variables method in the context of second-order equations by casting the Crandall--Ishii Lemma into a test function framework.  This adaptation allows us to effectively handle non-local integral operators, such as generators of Lévy processes, in the same framework as second-order operators, paving the way for stability results.
    \item We translate the key estimate on the difference of Hamiltonians in terms of an adaptation of the probabilistic notion of couplings, providing a unified approach that applies to both continuous and discrete operators. We point out that \cite{CrIsLi87} also discusses a coupling point of view, but only for first order operators.
    \item We strengthen the typical comparison principle using Lyapunov functionals from a sup-norm contractivity result to what we call the \textit{strict comparison principle}, cf.\ Definition \ref{definition:comparison}, which encodes continuity in the strict or sometimes also called mixed topology, cf.\ \cite{Bu58,Se72}.
\end{enumerate}
The results are illustrated in various examples in Section \ref{section:examples}.
To introduce the first two innovations, we heuristically trace back the classical doubling-of-variables procedure used to obtain comparison principles for first and second-order equations. For the sake of exposition, we focus on the $\theta$-independent case.

Given a subsolution $u$ and supersolution $v$ to an equation of type \eqref{eqn:HJ_intro} and, for $\alpha > 1$, optimizers $(x_\alpha, y_\alpha)$ to
\begin{equation} \label{eqn:intro_optimizing_point_construction}
    u(x_\alpha) - v(y_\alpha) - \frac{\alpha}{2} d^2(x_\alpha,y_\alpha) = \sup_{x,y \in \bR^q} \left\{u(x) - v(y) - \frac{\alpha}{2}d^2(x,y) \right\},
\end{equation}
one estimates
\begin{equation}
    \sup_{x \in \bR^q} u(x) - v(x) \leq h(x_\alpha) - h(y_\alpha) + \lambda \left[H\left(\frac{\alpha}{2}d^2(\cdot,y_\alpha)\right)(x_\alpha) -  H \left(-\frac{\alpha}{2}d^2(x_\alpha,\cdot) \right)(y_\alpha)\right].
\end{equation}
Consequently, comparison then holds, if
\begin{equation}\label{eqn:intro_final_abstractestimate}
    \liminf_{\alpha \rightarrow \infty} H\left(\frac{\alpha}{2}d^2(\cdot,y_\alpha)\right)(x_\alpha) -  H \left(-\frac{\alpha}{2}d^2(x_\alpha,\cdot) \right)(y_\alpha) \leq 0.
\end{equation}
The estimate \eqref{eqn:intro_final_abstractestimate}, then translates into explicit conditions on $H$.

\smallskip

When $H$ is, for example, of the form
\begin{equation} \label{eqn:intro_firstOrderOperator}
    Hf(x) = \ip{b(x)}{\nabla f(x)} + \frac{1}{2} |\nabla f(x)|^2,
\end{equation}
the estimate \eqref{eqn:intro_final_abstractestimate} translates into
\begin{align*}
    & H\left(\frac{\alpha}{2}d^2(\cdot,y_\alpha)\right)(x_\alpha) -  H \left(-\frac{\alpha}{2}d^2(x_\alpha,\cdot) \right)(y_\alpha) \\
    & = \left[\ip{b(x_\alpha)}{\alpha(x_\alpha-y_\alpha)} + \frac{\alpha^2}{2} d^2(x_\alpha,y_\alpha) \right] -  \left[\ip{b(y_\alpha)}{\alpha(x_\alpha-y_\alpha)} + \frac{\alpha^2}{2} d^2(x_\alpha,y_\alpha) \right] \\
    & \leq \ip{b(x_\alpha) - b(y_\alpha)}{\alpha(x_\alpha-y_\alpha)},
\end{align*}
which goes to $0$ for $\alpha \rightarrow \infty$, if $b$ is one-sided Lipschitz.

\smallskip

For second order operators, however, the same strategy fails since, considering, for example, the Laplacian $Hf(x) = \frac{1}{2}\Delta f(x) = \frac{1}{2} \Tr\left(D^2f(x)\right)$, we get
\begin{equation} \label{eqn:intro_estimate_for_Laplacian}
    H\left(\frac{\alpha}{2}d^2(\cdot,y_\alpha)\right)(x_\alpha) -  H \left(-\frac{\alpha}{2}d^2(x_\alpha,\cdot) \right)(y_\alpha) = 2 \alpha,
\end{equation}
which diverges as $\alpha \rightarrow \infty$.

The works \cite{jensen.1988b,jensen1988maximum} use the key insight that, while the first order-viscosity solution method explores the sequences of optimizers of \eqref{eqn:intro_optimizing_point_construction} separately (fix $y_\alpha$ and vary $x$ for the subsolution part and vice versa), for second order equations, one needs to treat the two sequences jointly. This insight was later formalized in \cite{CrIs90} and as Theorem 3.2 in the User's Guide \cite{CIL92}, now known as the Crandall--Ishii Lemma. The lemma states for equations of type $Hf(x) = \frac{1}{2} \Tr\left(D^2f(x)\right)$ that, given $X_\alpha = D^2 u(x_\alpha)$ and $Y_\alpha = D^2 v(y_\alpha)$ or their appropriate generalizations, we have the estimate
\begin{equation}
    \begin{pmatrix}
        X_\alpha & 0 \\
        0 & - Y_\alpha
    \end{pmatrix}
    \leq 
    3\alpha 
    \begin{pmatrix}
        \bONE & - \bONE \\
        - \bONE & \bONE
    \end{pmatrix}.
\end{equation}
Conjugating the matrices with
\begin{equation} \label{eqn:intro_def_transformationS}
    C \coloneqq \frac{1}{\sqrt{2}} \begin{pmatrix} \bONE & \bONE \\ \bONE & \bONE \end{pmatrix},
\end{equation}
i.e.\ essentially using $C$ to couple the subsolution and supersolution problems, we arrive at the desired estimate
\begin{equation} \label{intro:estimate_Laplacian_using_CIL}
    \frac{1}{2} \Tr(X_\alpha) - \frac{1}{2} \Tr(Y_\alpha) = \frac{1}{4} \Tr \begin{pmatrix}
        X_\alpha - Y_\alpha & X_\alpha - Y_\alpha \\
        X_\alpha - Y_\alpha & X_\alpha - Y_\alpha
    \end{pmatrix}
    \leq 0.
\end{equation}
We now briefly describe the three innovations (1)--(3).

\smallskip
{\bfseries Innovation 1: A test function framework.}
Examining the proof of the Crandall--Ishii Lemma, we can interpret the procedure as the construction of two test functions $\phi_\alpha, \psi_\alpha \in C^2(\bR^q)$ that are squeezed between $u$ and $v$ on one-hand and $\frac{\alpha}{2}d^2$ on the other. To be more precise, we find $\phi_\alpha, \psi_\alpha \in C^2(\bR^q)$ such that
\begin{equation} \label{eqn:intro_testFunctionPair_ours}
    u(x_\alpha) - \phi_\alpha(x_\alpha) = \sup_{x \in \bR^q} \{ u(x) - \phi_\alpha(x) \} \quad\text{and}\quad
    v(y_\alpha) - \psi_\alpha(y_\alpha) = \inf_{y \in \bR^q} \{ v(y) - \psi_\alpha(y) \},
\end{equation}
and
\begin{equation}\label{eqn:intro_testFunctionCombined_ours}
    \phi_\alpha(x_\alpha) - \psi_\alpha(y_\alpha) - \frac{\alpha}{2}d^2(x_\alpha,y_\alpha) = \sup_{x,y \in \bR^q} \left\{u(x) - v(y) - \frac{\alpha}{2}d^2(x,y) \right\}.
\end{equation}
As before, comparison now follows from the estimate 
\begin{equation}
    \liminf_{\alpha \rightarrow \infty} H \phi_\alpha(x_\alpha) -  H \psi_\alpha(y_\alpha) \leq 0.
\end{equation}
For the Laplacian $Hf(x) = \frac{1}{2}\Tr \left(D^2 f(x)\right)$, this translates to
\begin{equation} \label{eqn:intro:estimate_Laplacian_ours_1}
    H\phi_\alpha(x_\alpha) - H \psi_\alpha(y_\alpha) = \frac{1}{2} \Tr(D^2\phi_\alpha(x_\alpha)) - \frac{1}{2} \Tr(D^2 \psi_\alpha(y_\alpha)).
\end{equation}
At this point in proofs using the Crandall--Ishii Lemma, the estimate \eqref{intro:estimate_Laplacian_using_CIL} is performed by conjugation with the matrix $C$ in \eqref{eqn:intro_def_transformationS}. We formalize this step by adapting the probabilistic notion of couplings, cf.\ \cite{Li92,Th00,BuKe00}, and identify the choice of the matrix $C$ in \eqref{eqn:intro_def_transformationS} with the \emph{synchronous coupling} (also called \emph{co-monotone coupling}).

\smallskip

{\bfseries Innovation 2: The coupling approach.}
Indeed, given two Brownian motions starting in $x$ and $y$, one can construct a coupling of the two by considering
\begin{equation} \label{eqn:intro_def_syncrhonousCoupling}
    (X(t),\, Y(t)) = (x+B(t),\, y+B(t)),
\end{equation}
where $B(t)$ is a standard Brownian motion.
The generator of the coupled process \eqref{eqn:intro_def_syncrhonousCoupling} is given by
\begin{equation*}
    \widehat{H}g(x,y) := \frac{1}{2} \left(\partial_x + \partial_y\right)^2 g(x,y) = \frac{1}{2}\Tr \left(\begin{pmatrix}
        \bONE & \bONE \\ \bONE & \bONE
    \end{pmatrix} D^2 g(x,y) \right) = \frac{1}{2} \Tr \left(C D^2g(x,y) C^T\right),
\end{equation*} 
where we recover the matrix $C$ of \eqref{eqn:intro_def_transformationS}. Note that $\widehat{H}$ is indeed a coupling: For $f_1, f_2 \in C_b(\bR^q)$ and $(f_1 \oplus f_2)(x,y) := f_1(x) + f_2(y)$, we have
\begin{equation}
    \widehat{H}(f_1 \oplus f_2)(x,y) = Hf_1(x) + H f_2(y).
\end{equation}
Using the coupling $\widehat{H}$, we can now rewrite \eqref{eqn:intro:estimate_Laplacian_ours_1} as
\begin{equation} \label{eqn:intro_abstract_coupling_estimate}
    \begin{aligned}
    H\phi_\alpha(x_\alpha) -  H \psi_\alpha(y_\alpha) & = \widehat{H}\left(\phi_\alpha \oplus - \psi_\alpha\right)(x_\alpha,y_\alpha) \\
    & \leq \widehat{H} \left(\frac{\alpha}{2} d^2\right)(x_\alpha,y_\alpha) = 0, \\
    %& = 0,
\end{aligned}
\end{equation}
where the first equality follows by the definition of a coupling, the inequality is based on the positive maximum principle with the optimizers from equation \eqref{eqn:intro_testFunctionCombined_ours}, and the final equality is due to the fact that the synchronous coupling controls distance growth.

\smallskip

A similar strategy can be used to treat a discretized version of the Brownian Motion by considering the generator $Hf(x) = \frac{1}{2}\left[f(x+1) - f(x)\right] + \frac{1}{2}\left[f(x-1) - f(x)\right]$ of a random walk: We synchronously couple the random walk with itself using the operator
\begin{equation*}
    \widehat{H} f(x,y) = \frac{1}{2}\left[f(x+1,y+1) - f(x,y)\right] + \frac{1}{2}\left[f(x-1,y-1) - f(x,y)\right].
\end{equation*}
The argument in \eqref{eqn:intro_abstract_coupling_estimate} then works for the random walk exactly as it did for the Brownian motion.

This coupling approach is one of the main contributions of this paper, allowing for a unifying framework to show comparison for Hamilton--Jacobi equations with Hamiltonians of type \eqref{eqn:operator_linear_intro} and their Bellman and Isaacs versions, cf. Theorem \ref{th:comparison_HJI} and Corollary \ref{cor:HJ_HJB}.

\smallskip

{\bfseries Innovation 3: The strict comparison principle.}
Our third innovation is on the final estimate that is obtained as the comparison principle. For a subsolution $u$ to 
\begin{equation*}
    f - \lambda H f = h_1
\end{equation*}
and a supersolution $v$ to 
\begin{equation*}
    f - \lambda Hf = h_2
\end{equation*}
the comparison principle amounts to establishing that
\begin{equation*}
    \sup_{x \in \bR^q} u(x) - v(x) \leq \sup_{x \in \bR^q} h_1(x) - h_2(x).
\end{equation*}

The comparison principle, once established, thus implies sup-norm contractivity for the solution map $R(\lambda) : C_b(\bR^q) \rightarrow C_b(\bR^q)$, where $R(\lambda)h$ is the unique viscosity solution for the Hamilton--Jacobi equation \eqref{eqn:HJ_intro}.

\smallskip

It is well-known from examples, cf.\ \cite{BaCD97, YoZh99,CaSi04,FlSo06}, that the map $R(\lambda)h$ takes the form of an exponentially discounted Markovian control problem. If the dynamics admits a Lyapunov function $V$, having compact sublevel sets and satisfying $HV \leq c$, then the controlled Markov processes satisfy tightness properties. More precisely, if the controlled process starts in a compact set $K$, one can find, for any time horizon $T > 0$ and $\varepsilon > 0$, a compact set $\widehat{K} \supseteq K$, given in terms of the sublevel sets of $V$ such that, with probability $1-\varepsilon$, the process remains in $\widehat{K}$ up to time $T$. Rewriting this in terms of an estimate on the solution map $R(\lambda)$, we then find
\begin{equation} \label{eqn:intro_strict_contractivity}
    \sup_{x \in K} R(\lambda)h_1(x) - R(\lambda)h_2(x) \leq \varepsilon \vn{h_1 - h_2} + \sup_{x \in \widehat{K}} h_1(x) - h_2(x).
\end{equation}
Estimates of this type are indeed characterized by the strict topology, as was first established for linear functionals in \cite[Theorem 5.1]{Se72} and for convex, monotone functionals in \cite[Corollary 2.10]{Ne24}. Note that in this paper, we do not establish convexity of $h \mapsto R(\lambda)h$, but want to point out that given a convex $H$, convexity of $R(\lambda)h$ is to be expected by performing a comparison principle in terms of three variables using variants of the, e.g., three dimensional Theorem 3.2 of \cite{CIL92}, see also the domination principle of Theorem 2.22 and Corollary 2.26 of \cite{Ho16}. We leave this for future work.

Building upon the notion of Lyapunov functions, we will show that we can directly establish a variant of \eqref{eqn:intro_strict_contractivity} for a subsolution $u$ and a supersolution $v$. Given its motivation, we will call this estimate the \textit{strict comparison principle}, see Definition \ref{definition:comparison}  and the main result, Theorem \ref{th:comparison_HJI}, below.

\smallskip
\subsection*{Organization of the paper}
The rest of the paper is organized as follows: Section \ref{sec:framework} introduces the notation and definitions. Section \ref{section:main_results} introduces the framework by stating the necessary assumptions and formalizing the main results. In Section \ref{section:examples}, we show how to apply our framework to operators of the form \eqref{eqn:operator_linear_intro}. Section \ref{section:testfunctions_construction} contains the construction of the required optimizing points and test functions. Finally, Section \ref{section:main_proof} contains the proof of the main theorems.

\section{Preliminaries and general setting}\label{sec:framework}
\subsection{Notation and Preliminaries}\label{sec:framework_notations}

Throughout the paper, let $q\in \bN$ and $E = \bR^q$. We write $C(E)$ for the set of all real-valued continuous functions on $E$, where $E$ is endowed with the topology induced by the Euclidean distance $d$ on $\bR^q$. 

Let $C(E)$ and $C_b(E)$ be the set of continuous and bounded continuous functions. For $k \in \bN$, let $C^k(E)$ denote the space of all real-valued functions on $E$ that are $k$-times continuously differentiable. Let $C^k_b(E)$ the set of all functions in $C^k(E)$ with bounded derivatives up to order $k$. We denote the space of all smooth functions that are constant outside of a compact set by $C_c^\infty(E)$. 
We write $C_u(E)$ and $C_l(E)$ for the set of continuous functions on $E$ that are uniformly bounded from above and below, respectively. 
Moreover, we write
\begin{align*}
	C_+(E) & := \{f \in C(E) \, | \, f \text{ has compact sub-level sets}\}, \\
	C_-(E) & := \{f \in C(E) \, | \, f \text{ has compact super-level sets}\}, \\
	C_c(E) & := \{f \in C(E) \, | \, f \text{ is constant outside of a compact set}\}.
\end{align*}
We furthermore define the following intersections: $C_c^2(E) = C_c(E) \cap C^2(E)$,
\begin{equation*}
	C_+^2(E) \coloneqq C_+(E) \cap C^2(E), \qquad C_-^2(E) \coloneqq C_-(E) \cap C^2(E).
\end{equation*}

For $a,b \in \bR$, we write $a \vee b \coloneqq \max \{a,b\}$ and $a \wedge b \coloneqq \min \{a,b\}$.
We denote the supremum norm by $\vn{\,\cdot\,}$, that is
\begin{equation}
    \vn{f} = \sup_{x\in E} |f(x)|,
\end{equation}
for $f \in C_b(E)$, while, for $u \in C(E)$, we use the notation
\begin{equation}
\ssup{u} \coloneqq \sup_{x \in E} u(x), \qquad \iinf{u} \coloneqq \inf_{x \in E} u(x)
\end{equation}
for a supremum or infimum over the entire space and 
\begin{equation*}
   \ssup{u}_C \coloneqq \sup_{x \in C} u(x), \qquad \iinf{u}_C \coloneqq \inf_{x \in C} u(x)
\end{equation*}
for a supremum or infimum over a subset $C \subseteq E$.

We say that a function $\omega \colon [0,\infty) \rightarrow [0,\infty)$ is a \emph{modulus of continuity}, if $\omega$ is upper semi-continuous with $\omega(0) = 0$.
We say that a function $f \in C(E)$ \emph{admits a modulus of continuity}, if, for every compact $K \subseteq E$, there exists a modulus of continuity $\omega_K \colon [0, \infty) \rightarrow [0, \infty)$ such that, for all $x, y \in K$, we have
\begin{equation}\label{eq:modulus_admission}
    | f(x) - f(y) | \leq \omega_K(d(x,y)).
\end{equation}

A function $\phi \colon E \rightarrow \bR$ is called \emph{semi-convex} with constant $\kappa \in \bR$ if for any $x_0 \in E$ the map
\begin{equation*}
    x \mapsto \phi(x) + \frac{\kappa}{2}d^2(x,x_0)
\end{equation*}
is convex. Moreover, $\phi$ is called \emph{semi-concave} with constant $\kappa \in \bR$ if $-\phi$ is semi-convex with constant $-\kappa$.

We say that a function $f \in C(E,\bR^q)$ is \emph{one-sided Lipschitz} if, for all $x,y \in E$ and some constant  $C \in \bR$, we have
\begin{equation}\label{def:OneSideLip}
    \ip{x-y}{f(x) - f(y)} \leq C d^2(x,y).
\end{equation}

For any $z \in E$, let $s_z : E \rightarrow \bR^q$ be the \emph{shift map}
\begin{equation}\label{definition:shift_map}
    s_z(x) = x-z.
\end{equation}

For any $z_1, z_2 \in E$, let
\begin{equation*}
        d_{z_1, z_2} (x,y) \coloneqq d\left(s_{z_1} (x), s_{z_2} (y)\right).
\end{equation*}

Let $f_1, f_2 \in C(E)$. Then, we define the \emph{direct sum} $f_1 \oplus f_2, f_1 \ominus f_2 \in C(E\times E)$ as
\begin{equation*}
    (f_1 \oplus f_2)(x_1, x_2) \coloneqq f_1(x_1) + f_2(x_2) \quad\text{and}\quad (f_1 \ominus f_2)(x_1, x_2) \coloneqq f_1(x_1) - f_2(x_2)
\end{equation*}
for all $x_1,x_2\in E$. For two sets of functions $F_1, F_2 \subseteq C(E)$, we define
\begin{equation*}
    F_1 \oplus F_2 \coloneqq \left\{f_1 \oplus f_2 \,\middle|\, f_1 \in F_1, f_2 \in F_2\right\} \quad\text{and}\quad
    F_1 \ominus F_2 \coloneqq \left\{f_1 \ominus f_2 \,\middle|\, f_1 \in F_1, f_2 \in F_2\right\}.
\end{equation*}

\subsection{Operator notions}

We consider operators $H\subseteq C(E) \times C(E)$, where we identify $H$ by its graph. As usual, the \emph{domain} of $H$ is given by
\begin{equation*}
    \cD(H) \coloneqq \left\{f\in C(E)\,\middle|\, \exists \,g\in C(E)\colon (f,g)\in H\right\}.
\end{equation*}
Let $H_1, H_2 \subseteq C(E) \times C(E)$. We define
\begin{equation*}
    H_1 + H_2 \coloneqq \left\{(f,g_1+g_2)\, \middle|\, (f,g_1) \in H_1, (f,g_2) \in H_2 \right\},
\end{equation*}
which is an operator with domain
\begin{equation*}
    \cD(H_1 + H_2) \coloneqq \cD(H_1) \cap \cD(H_2).
\end{equation*}
We say that $H$ is \emph{linear on its domain} if, for any $f,g \in \cD(H)$ and $a \in \bR$ such that $af + g \in \cD(H)$, we have
\begin{equation}\label{definition:extended_linear}
    H \left(af+g\right) = aHf + Hg.
\end{equation}
We will prove the comparison principle for the equation in terms of $H$ by relating it to two equations in terms of two restrictions of $H$. To do so, we will need to be able to construct test functions in the domain of $H$ from functions in the domain of the restrictions. In particular, we will need the following notion.
\begin{definition}[Sequential Denseness]\label{def:order_dense}
    Let $\cD \subseteq C_b(E)$, $\cD_+ \subseteq C_+(E)$, and $\cD_- \subseteq C_-(E)$.

    \begin{itemize}
        \item We say that $\cD$ is \emph{upward sequentially dense} in $\cD_+$ if, for any $f_\dagger \in \cD_+$ and constant $a \in \bR$, there exists a function $f_{\dagger,a} \in \cD$ such that
        \begin{equation*}
            \begin{cases}
                f_{\dagger,a}(x) = f_\dagger(x) & \text{if } f_\dagger(x) \leq a, \\
                a < f_{\dagger,a}(x) \leq f_\dagger(x) & \text{if } f_\dagger(x) > a.
            \end{cases}
        \end{equation*}
        
        \item We say that $\cD$ is \emph{downward sequentially dense} in $\cD_-$ if, for any $f_\ddagger \in \cD_-$ and constant $a\in \bR$, there exists a function $f_{\ddagger, a} \in \cD$ such that
        \begin{equation*}
            \begin{cases}
                f_{\ddagger,a}(x) = f_\dagger(x) & \text{if } f_\ddagger(x) \geq a, \\
                a > f_{\ddagger,a}(x) \geq f_\ddagger(x) & \text{if } f_\ddagger(x) < a.
            \end{cases}
        \end{equation*}
    \end{itemize} 
\end{definition}
\subsection{Viscosity solutions}

For $\lambda > 0$, consider $h_1 \in C_l(E)$ and $h_2 \in C_u(E)$ and two operators $H_1 \subseteq C_l(E) \times C(E)$ and $H_2 \subseteq C_u(E) \times C(E)$. We study the pair of equations
\begin{align}
    f - \lambda H_1 f & \leq h_1, \label{eqn:HJ_subsolution} \\
    f - \lambda H_2 f & \geq h_2. \label{eqn:HJ_supersolution}
\end{align}
 The notion of viscosity solution is built upon the maximum principle.
 \begin{definition}[Maximum principle]\label{def:max_principle}
    We say that an operator $H \subseteq C(E) \times C(E)$ satisfies the \emph{maximum principle} if, for all $f_1,f_2 \in \cD(H)$ and $x_0 \in E$ with
    \begin{equation*}
        f_1(x_0) - f_2(x_0) = \sup_{x \in E} \{f_1(x) - f_2(x)\},
    \end{equation*}
    we have
    \begin{equation*}
        H f_1(x_0) \leq H f_2(x_0)
    \end{equation*}
    and, analogously, for all $f_1,f_2 \in \cD(H)$ and $x_0 \in E$ with
    \begin{equation*}
        f_1(x_0) - f_2(x_0) = \inf_{x \in E} \{f_1(x) - f_2(x)\},
    \end{equation*}
    we have
    \begin{equation*}
        H f_1(x_0) \geq H f_2(x_0).
    \end{equation*}
\end{definition}
Observe that every operator $H \subseteq C(E) \times C(E)$ that satisfies the maximum principle is single-valued, i.e., for all $f\in \cD(H)$,
\begin{equation*}
\#\{g\in C(E)\mid (f,g)\in H\}=1.
\end{equation*}
\begin{definition}[Viscosity sub- and supersolutions] \label{definition:viscosity_solutions} \label{def:viscosity_solution}

    Let $H_1 \subseteq C_l(E) \times C(E)$ and $H_2 \subseteq C_u(E) \times C(E)$ be two operators with domains $\cD(H_1)$ and $\cD(H_2)$, respectively. Moreover, let $\lambda > 0$, $h_1 \in C_l(E)$, and $h_2 \in C_u(E)$.
    \begin{enumerate}[(a)]
        \item A bounded, upper semicontinuous function $u \colon E\to \R$ is called a \emph{(viscosity) subsolution} to \eqref{eqn:HJ_subsolution} if, for all $(f,g) \in H_1$, there exists a sequence $(x_n)_{n\in \bN} \subseteq E$ such that
        \begin{gather*}
            \lim_{n \rightarrow \infty} u(x_n) - f(x_n)  = \sup_{x\in E} u(x) - f(x), \\
            \limsup_{n \rightarrow \infty} u(x_n) - \lambda g(x_n) - h_1(x_n) \leq 0.
        \end{gather*}
        
        \item A bounded, lower semicontinuous function $v \colon E\to \R$ is called a \emph{(viscosity) supersolution} to \eqref{eqn:HJ_supersolution} if, for all $(f,g) \in H_2$, there exists a sequence $(x_n)_{n\in \bN} \subseteq E$ such that
        \begin{gather*}
            \lim_{n \rightarrow \infty} v(x_n) - f(x_n) = \inf_{x\in E} v(x) - f(x), \\
            \liminf_{n \rightarrow \infty} v(x_n) - \lambda g(x_n) - h_2(x_n) \geq 0.
        \end{gather*}
    \end{enumerate}
    
    If $H_1 = H_2$ and $h_1 = h_2$, a function $u \in C_b(E)$ is called a \emph{(viscosity) solution} to the pair of equations \eqref{eqn:HJ_subsolution} and \eqref{eqn:HJ_supersolution} if it is both a subsolution to \eqref{eqn:HJ_subsolution} and a supersolution to \eqref{eqn:HJ_supersolution}.
\end{definition}
Working with test functions that have compact sub- or superlevel sets respectively, an approximating sequence can be replaced by an optimizing point in the definition of sub- or supersolution. See Lemma \ref{lemma:def_equiv} in Appendix \ref{section:def_equiv}.

Associated with the definition of viscosity solutions, we introduce the comparison principle, which, for $h_1 = h_2$, implies uniqueness in the viscosity sense for solutions of the Hamilton--Jacobi equation $f-\lambda H f = h$. We additionally introduce a new, stronger notion: the strict comparison principle. This name is inspired by the observation that the comparison principle implies contractivity in the sup-norm of the solution map. The strict comparison principle implies continuity in terms of the weaker strict topology, see e.g. \cite{Se72}.

\begin{definition} \label{definition:comparison}
    We say that the equations \eqref{eqn:HJ_subsolution} and \eqref{eqn:HJ_supersolution} satisfy 
    
    \begin{enumerate}[(a)]
        \item the \emph{comparison principle} if, for any subsolution $u$ to \eqref{eqn:HJ_subsolution} and any supersolution $v$ to \eqref{eqn:HJ_supersolution}, we have 
    \begin{equation*}
        \sup_{x\in E} u(x) - v(x) \leq \sup_{x \in E} h_1(x) - h_2(x).
    \end{equation*}
    \item  the \emph{strict comparison principle} if, for any subsolution $u$ to \eqref{eqn:HJ_subsolution}, any supersolution $v$ to \eqref{eqn:HJ_supersolution}, any compact set $K\subseteq E$ and $\varepsilon >0$, there exist a compact set $\widehat{K} = \widehat{K}(K,\varepsilon,\vn{u},\vn{v})$ and a constant $C = C(u,v,K,h_1,h_2,\lambda)$ such that we have 
    \begin{equation*}
        \sup_{x\in K} u(x) - v(x)  \leq  \varepsilon C + \sup_{x \in \widehat{K}} h_1(x) - h_2(x).
    \end{equation*}
    \end{enumerate}
\end{definition}

Observe that the strict comparison principle implies the comparison principle. Indeed, by the strict comparison principle, for all $x_0\in E$ and $\varepsilon>0$, there exists a constant $C$, independent of $\varepsilon$, and a compact set $\widehat K\subseteq E$ such that
\[
u(x_0) - v(x_0)  \leq  \varepsilon C + \sup_{x \in \widehat{K}} h_1(x) - h_2(x)\leq \varepsilon C + \sup_{x \in E} h_1(x) - h_2(x).
\]
Letting $\varepsilon \downarrow 0$, we find that $u(x_0) - v(x_0)\leq \sup_{x \in E} h_1(x) - h_2(x)$. Taking the supremum over all $x_0\in E$,  the comparison princple follows. 
\subsection{Notions for our framework}\label{sec:notion_our_framework} One of the main innovations of this work is the use of a new approach to prove the comparison principle based on the notion of couplings of operators. In the following we give the main definitions underlying our new framework.
\begin{definition}[Coupling]\label{def:coupling:only_coupling}
    Let $H \subseteq C(E) \times C(E)$ and $\widehat{H} \subseteq C(E^2) \times C(E^2)$ be linear on their respective domains. We say $\widehat{H}$ is a \emph{coupling of $H$} if  $\cD(H) \oplus \cD(H) \subseteq \cD(\widehat{H})$ and, for any $f_1, f_2 \in \cD(H)$, we have
    \begin{equation*}
        \widehat{H} \left(f_1 \oplus f_2\right) = H f_1 + H f_2.
    \end{equation*}
\end{definition}

\begin{definition}[Controlled growth]\label{def:coupling:growth}
    Let $\widehat{H} \subseteq C(E^2) \times C(E^2)$. We say that $\widehat{H}$ has \emph{controlled growth} if, for any $\alpha > 1$ and $z,z'\in E$, we have $\frac{\alpha}{2}d^2_{z,z'} \in \cD(\widehat{H})$. In addition, for any compact set $K \subseteq E$, there exists a modulus of continuity $\omega_K: [0, \infty) \rightarrow [0, \infty)$ and $x,x',y,y' \in K$ such that 
    \begin{multline}
        \widehat{H}\left(\frac{\alpha}{2}d^2_{x-y,\; x'-y'}\right)(x,x') \leq \omega_{\widehat{H}, K}\left(\alpha \left(d(x,y) + d(y,y') + d(y', x')\right)^2 + \right.\\
        \left(d(x,y) + d(y,y') + d(y', x')\right)\Big).
    \end{multline}
\end{definition}

\begin{definition}[Controlled growth coupling]\label{def:coupling}
    Let $H \subseteq C(E) \times C(E)$ and $\widehat{H} \subseteq C(E^2) \times C(E^2)$ be linear on their respective domains. We say $\widehat{H}$ is a \emph{controlled growth coupling of $H$} if the following properties are satisfied:
    \begin{enumerate}[(a)]
        \item $\widehat{H}$ satisfies the \emph{maximum principle}, cf. Definition \ref{def:max_principle}.
        \item $\widehat{H}$ is a \emph{coupling} of $H$, cf. Definition \ref{def:coupling:only_coupling}.
        \item $\widehat{H}$ has \emph{controlled growth}, cf. Definition \ref{def:coupling:growth}.
    \end{enumerate}
\end{definition}

We will split our Hamiltonian into a stochastic part that we can couple in the sense of the above definitions, and a deterministic part that we require to be a \textit{convex semi-monotone} operator. Here we give the precise definitions.

\begin{definition}[Local first-order operator]\label{def:first_order:only_first_order}
    We say that $H \subseteq C(E) \times C(E)$ is a \emph{local first-order} operator if there exists a continuous map $\cB : E \times \bR^q \rightarrow \bR$ such that, for any $f \in \cD$, we have $Hf(x) = \cB(x,\nabla f(x))$.
\end{definition}

\begin{definition}[Local semi-monotonicity]\label{def:oneSidedLipschitz}
    Let $H \subseteq C(E) \times C(E)$ be local first-order for some $\cB$, cf. Definition \ref{def:first_order:only_first_order}.
    We say that $H$ is \emph{locally semi-monotone} if, for any compact sets $K \subseteq E$, there exists a modulus of continuity $\omega_{\cB,K}: [0,\infty)\to [0,\infty)$ such that, for all $x,y \in K$ and $\alpha > 1$,
    \begin{equation*}
        \cB(x,\alpha(x-x')) - \cB(y,\alpha(x-x')) \leq \omega_{\cB,K}\left(\alpha d^2(x,x') + d(x,x') \right).
    \end{equation*}
\end{definition}

\begin{definition}[Convex semi-monotone operator] \label{definition:first_order}
    We say that $H \subset C(E) \times C(E)$ is a \emph{convex semi-monotone operator} if the following properties are satisfied:
    \begin{enumerate}[(a)]
        \item $H \subseteq C(E) \times C(E)$ is \emph{locally semi-monotone} for some $\cB$, cf. Definition \ref{def:oneSidedLipschitz}.
        \item For all $x\in E$, the map $p \mapsto \cB(x, p)$ is \emph{convex}.
    \end{enumerate}
\end{definition}
Finally, to work with Hamilton--Jacobi-Isaacs equations we need the following condition to be satisfied by our Hamiltonian.
\begin{definition}[Isaacs' condition]\label{def:isaacs_cond}
    Let $\Theta_1$ and $\Theta_2$ be two compact, metric spaces. We say that a collection $\{H_{\theta_1,\theta_2}\}_{\theta_1\in\Theta_1,\theta_2\in\Theta_2}\subseteq C(E)\times C(E)$ satisfies \emph{Isaacs' condition} if, for all $f\in\bigcap_{\theta_1 \in \Theta_1, \theta_2 \in \Theta_2} \cD(H_{\theta_1, \theta_2})$, 
    \begin{equation}
        \sup_{\theta_1\in \Theta_1} \inf_{\theta_2\in\Theta_2} \left\{H_{\theta_1,\theta_2} f(x)\right\}
        = \inf_{\theta_2\in\Theta_2} \sup_{\theta_1\in \Theta_1} \left\{H_{\theta_1,\theta_2} f(x) \right\}.
    \end{equation}
\end{definition}

Following typical comparison principle proofs, we will be perturbing the optimization problem
\begin{equation*}
    \sup_{x \in E} u(x) - v(x),
\end{equation*}
using a variant of the doubling of variables procedure to ensure that we can use the properties of sub- and supersolutions. Our perturbations consist of two components:
\begin{itemize}
    \item We need a Lyapunov-type function, which ensures that we can work on compact sets, see Definition \ref{definition:perturbation_containment}.
    \item We perform a variant of the Jensen perturbation to construct optimizers in which we can differentiate twice, see Definition \ref{definition:perturbation_first_second_order}.
\end{itemize}

\begin{definition} \label{definition:perturbation_containment}
    We call $V : E \rightarrow [0,\infty)$ a \emph{containment function} if
    \begin{enumerate}[(a)]
        \item $\inf_{y \in E} V(y) = 0$,  
        \item $V$ is semi-concave with semi-concavity constant $\kappa_V$,
        \item for every $c \geq 0$ the set $\{y \, | \, V(y) \leq c\}$ is compact. 
    \end{enumerate}
\end{definition}

Typically, the containment function is 
\begin{equation}\label{eq:containment}
V(x) = \log \left( 1 + \frac{1}{2} x^2\right).
\end{equation}

The next definition aims to produce optimizers for which we have twice differentiability via Jensen's Lemma, cf. Lemma A.3 of \cite{CIL92}. The variant used here creates a unique, global optimizer from a local one using $\xi$ and then shifts it slightly with $\zeta$.% to produce an appropriate optimizer in which we can differentiate twice. 
The two sets of perturbations are based on the prototypical examples of lines, i.e., 

\begin{equation}\label{eq:perturbation_1}
\zeta_{z,p}(x) = \ip{p}{x-z},
\end{equation}
and parabolas, i.e., 
\begin{equation}\label{eq:perturbation_2}
\xi_z(x) = \frac{1}{2}d^2(x,z),
\end{equation}
both centered at some $z \in E$. We give them as pair to capture the idea that quadratic growth dominates linear growth, cf. Definition \ref{definition:perturbation_first_second_order} \ref{item:definition:penalization:domination} below, which is used in our variant of Jensen's Lemma in the Appendix, cf. Proposition \ref{proposition:Jensen_Alexandrov_cutoff}.

\begin{definition} \label{definition:perturbation_first_second_order}
    We call collections of maps $\{\zeta_{z,p}\}_{z \in E, p \in \bR^q} \subset C(E)$ and $\{\xi_{z}\}_{z \in E} \subset C^1(E)$ $\zeta_{z,p} : E \rightarrow \bR$ and $\xi_z : E \rightarrow \bR$ sets of \emph{first} and \emph{second order point penalizations}, respectively, if there exist constants $R>0$ and $\kappa_\xi > 0$ such that for all $z \in E$:
    \begin{enumerate}[(a)]
        \item \label{item:definition:penalization:linear} $\zeta_{z,p}$ is linear in terms of $p$ around $z$:
            \begin{equation*}
                \zeta_{z,p}(y) = \ip{p}{y-z}
            \end{equation*}
            if $y \in B_R(z)$.
        \item \label{item:definition:penalization:semi-concave} The map $\xi_z$ is semi-concave with constant $\kappa_\xi$.
        \item \label{item:definition:penalization:second_order} The map $\xi_z$ is a penalization away from $z$:
            \begin{equation*}
                \xi_z(z) = 0, \qquad \xi_z(y) > 0, \qquad \text{if } y \neq z.
            \end{equation*}
        \item \label{item:definition:penalization:domination}  We have
            \begin{equation*}
                \inf_{|p| \leq 1} \inf_{y \notin B_R(z)} \xi_z(y) + \zeta_{z,p}(y) > 0.
            \end{equation*}
    \end{enumerate}
    For any given $z_0,z_1 \in E$ and $p \in \bR^q$, we consider the maps
\begin{align}\label{definition:Jensen_penalization_Xi}
    \Xi^0(y) = \Xi_{z_0,p}^0(y) & := \xi_{z_0}(y) + \zeta_{z_0,p}(y), \\
     \Xi(y) = \Xi_{z_0,p,z_1}(y) & := \xi_{z_0}(y) + \zeta_{z_0,p}(y) + \xi_{z_1}(y).
\end{align}
\end{definition}

\section{Assumptions and main result} \label{section:main_results}

In this section, we present our main result, Theorem \ref{th:comparison_HJI}, and outline the fundamental assumptions underlying our analysis.

Theorem \ref{th:comparison_HJI} states the strict comparison principle for operators in Hamilton--Jacobi-Isaacs (HJI) form, satisfying Isaacs' condition. Comparison principles for Hamilton--Jacobi (HJ) and Hamilton--Jacobi-Bellman (HJB) equations readily follow.

Heuristically, we assume that the base operator $\bH$ can be split into a stochastic part $\bA$, which we can couple, a semi-monotone deterministic term $\bB$, and a cost functional $\cI$, and that the action of the operator on the containment function $V$ is controlled.

\begin{theorem}[Strict comparison principle]\label{th:comparison_HJI}
    Let $\bH \subseteq C(E) \times C(E)$ be given by 
    \begin{equation*}
    \bH f(x) = \sup_{\theta_1\in \Theta_1} \inf_{\theta_2\in\Theta_2} \left\{\bA_{\theta_1,\theta_2} f(x) + \bB_{\theta_1,\theta_2} f(x) - \cI(x,\theta_1,\theta_2)\right\}
    \end{equation*}
    with $\Theta_1$ and $\Theta_2$ compact, metric spaces and $\cI : E \times \Theta_1 \times \Theta_2 \rightarrow (-\infty,\infty]$ a cost functional. Furthermore, consider a containment function $V$ and penalization functions $\{\zeta_{z,p}\}_{z \in E, p \in \bR^q}$, $\{\zeta_z\}_{z \in E}$. Let $\bH$ satisfy the technical Assumptions \ref{assumption:domain_setup} and \ref{assumption:compatibility} below and assume that
    \begin{enumerate}[(a)]
        \item \label{item:assumption:isaacs_cond} The collection of operators $\{\bA_{\theta_1,\theta_2} + \bB_{\theta_1,\theta_2} - \cI\}_{\theta_1 \in \Theta_1, \theta_2 \in \Theta_2}$ satisfies Isaacs' condition, cf. Definition \ref{def:isaacs_cond}.
        \item \label{item:assumption:Acoupling} For all $\theta_1 \in \Theta_1, \theta_2 \in \Theta_2$, $\bA_{\theta_1,\theta_2}$ is linear on its domain and admits a controlled growth coupling $\widehat{\bA}_{\theta_1,\theta_2}$ as in Definition \ref{def:coupling} with a modulus uniform in $\theta_1$ and $\theta_2$. 
        \item \label{item:assumption:Bmonotone} For all $\theta_1 \in \Theta_1, \theta_2 \in \Theta_2$, $\bB_{\theta_1,\theta_2}$ is a convex semi-monotone operator as in Definition \ref{definition:first_order} with a modulus uniform in $\theta_1$ and $\theta_2$.      
        \item \label{item:assumption_directHJI_lsc} The cost functional $\cI$ is lower semi-continuous in $(x,\theta_1,\theta_2)$, upper semi-continuous in $\theta_2$ for fixed $(x, \theta_1)$, and admits a modulus of continuity in $x$ uniformly in $(\theta_1, \theta_2)$.   
        \item \label{item:assumption_Lyapunovfunction} $V$ is a \emph{Lyapunov function} for $\bH$: $V \in \cD(\bH)$ and
        \begin{equation}\label{eq:mainth_HJI_bound}
            c_V := \sup_{x \in E} \sup_{\theta_1\in\Theta_1} \sup_{\theta_2\in\Theta_2} \left\{(\bA_{\theta_1,\theta_2} + \bB_{\theta_1,\theta_2}) V(x) -\cI(x,\theta_1,\theta_2) \right\} < \infty.
        \end{equation}
    \end{enumerate}

    Let $H \coloneqq \left\{(f,g) \in \bH \, \middle| \, f \in C_b(E) \right\}$ be the restriction of $\bH$ to $C_b(E)$ and consider
    \begin{equation} \label{eqn:mainTheoremHJ}
        f - \lambda H f = h
    \end{equation}
    for $\lambda > 0$ and $h \in C_b(E)$. Let $u$ and $v$ be a sub- and supersolution to \eqref{eqn:mainTheoremHJ} with $h_1$ and $h_2$ instead of $h$, respectively.
     Then for any compact set $K \subseteq E$ and $\varepsilon \in (0,1)$, we have
    \begin{equation}\label{eq:th:final}
        \sup_{x \in K} u(x) - v(x) \leq \varepsilon C_\eps + \sup_{x \in \widehat{K}_\eps} h_1(x) - h_2(x),
    \end{equation}
    where $\widehat{K}_\eps := \widehat{K}_\eps(K,u,v)$ and $C_\varepsilon := C_\varepsilon(K,u,v,h_1,h_2)$ are given by
    \begin{align*}
        \widehat{K}_\eps & := \left\{z \in E \, \middle| \, V(z) \leq \frac{\vn{u} + \vn{v}}{\varepsilon} + \ssup{V}_K \right\}, \\
        C_\eps & :=  \frac{2}{1-\eps^2} \left( \ssup{V}_K + \lambda c_{V}\right)  + \frac{1}{1-\eps} \vn{h_1} + \frac{1}{1-\eps} \vn{h_2} - \iinf{\frac{1}{1-\varepsilon}u - \frac{1}{1+\varepsilon}v }_K .
    \end{align*}
    In particular, the strict comparison principle holds for \eqref{eqn:mainTheoremHJ}.
\end{theorem}

The proof of the above theorem is carried out in Sections \ref{section:testfunctions_construction} and \ref{section:main_proof}. The next result follows by restricting the choice on $\Theta_1$ and $\Theta_2$.

\begin{corollary}\label{cor:HJ_HJB}\ 
    \begin{enumerate}[(a)]
        \item\label{item:cor:HJB} The strict comparison principle for HJB equations follows from Theorem \ref{th:comparison_HJI} by taking $\Theta_2$ to be a singleton.
        \item\label{item:cor:HJ} The strict comparison principle for HJ equations follows from Theorem \ref{th:comparison_HJI} by taking both $\Theta_1$ and $\Theta_2$ to be singletons.
    \end{enumerate}
\end{corollary}

\begin{remark} \label{remark:martingale_problem}
    If $H$ is the generator of a Markov process, the comparison principle implies uniqueness of the martingale problem using Theorem 3.7 in \cite{CoKu15}. We also refer to \cite{SV79,EK86} for details on the martingale problem.
\end{remark}

\subsection{Regularity and compatibility assumptions}

In this section, we state the technical assumptions necessary for the proof the main theorem.

As we have a choice for the domain of our operator and only need functions with compact sub- and superlevel sets, we need to ensure that the domains of the restrictions are regular enough to perform our analysis. In particular, the action of the operator on test functions and their combinations with perturbations needs to be well-defined. Furthermore, we require that the domains are large enough to allow for approximations in the sense of Definition \ref{def:order_dense}.

\begin{assumption}[Regularity of $\bH$] \label{assumption:domain_setup}
    Let $\bH \subseteq C(E) \times C(E)$ be an operator with the following three restrictions
    \begin{align*}
        H & := \left\{(f,g) \in \bH \, \middle| \, f \in C_b(E) \right\}, \\
        H_+ & := \left\{(f,g) \in \bH \, \middle| \, f \in C_+(E) \right\}, \\
        H_- & := \left\{(f,g) \in \bH \, \middle| \, f \in C_-(E) \right\},
    \end{align*}
    satisfying
    \begin{enumerate}[(a)]
        \item \label{item:assumption_domain_setup:max_principle} $\bH$ satisfies the maximum principle,
        \item \label{item:assumption_domain_setup:domainH} $\cD(H)$ is linear and $C_c^\infty(E) \subseteq \cD(H) \subseteq C_b(E)$,
        \item \label{item:assumption_domain_setup:upward_dense} $\cD(H)$ is upward sequentially dense in $\cD(H_+)$, as in Definition \ref{def:order_dense},
        
        \item \label{item:assumption_domain_setup:downward_dense} $\cD(H)$ is downward sequentially dense in $\cD(H_-)$ as in Definition \ref{def:order_dense},
        \item \label{item:domain_setup_extended_convex} $\cD(H_+)$ is convex,
        \item \label{item:domain_setup_extended_affine} 
        for any $f \in \cD(H)$ and $g \in \cD(H_+)$ and $\delta \in (0,1)$ we have
        \begin{align*}
            & (1-\delta) f + \delta g \in \cD(H_+),
            & (1+\delta) f - \delta g \in \cD(H_-).
        \end{align*}
    \end{enumerate}
\end{assumption}

In our main theorem, Theorem \ref{th:comparison_HJI}, we assume that the Hamiltonian $\bH$ has an Isaacs-type structure
\begin{equation*}
    \bH f(x) = \sup_{\theta_1\in \Theta_1} \inf_{\theta_2\in\Theta_2} \left\{\bA_{\theta_1,\theta_2} f(x) + \bB_{\theta_1,\theta_2} f(x) - \cI(x,\theta_1,\theta_2)\right\}.
\end{equation*}
To ensure it is well-behaved, we need that the collections of operators $\{\bA_{\theta_1, \theta_2}\}_{\theta_1 \in \Theta_1, \theta_2 \in \Theta_2}$ and $\{\bB_{\theta_1, \theta_2}\}_{\theta_1 \in \Theta_1, \theta_2 \in \Theta_2}$ themselves are well-behaved as a functions of $(\theta_1,\theta_2)$. We additionally assume that these collections behave well on the families of penalization functions introduced in Section \ref{sec:notion_our_framework}.

\begin{assumption}[Compatibility of $\bA_{\theta_1, \theta_2}$ and $\bB_{\theta_1, \theta_2}$] \label{assumption:compatibility}
Let $\Theta_1$ and $\Theta_2$ be compact, metric spaces. For $\theta_1 \in \Theta_1$ and $\theta_2 \in \Theta_2$, let $\bA_{\theta_1, \theta_2}, \bB_{\theta_1, \theta_2}  \subseteq C(E) \times C(E)$.
Consider a containment function $V$ as in Definition \ref{definition:perturbation_containment} and penalization functions $\{\zeta_{z,p}\}_{z \in E, p \in \bR^q}$ and $\{\zeta_z\}_{z \in E}$ as in Definition \ref{definition:perturbation_first_second_order}.
  \begin{enumerate}[(a)]
        \item\label{item:compatA} Let the collection $\{\bA_{\theta_1, \theta_2}\}_{\theta_1 \in \Theta_1, \theta_2 \in \Theta_2}$ be \emph{compatible} with $V$, $\{\zeta_{z,p}\}_{z \in E, p \in \bR^q}$, and $\{\zeta_z\}_{z \in E}$, i.e.,
        \begin{enumerate}[(1)]
            \item \label{item:compatA:domain} we have
            \begin{equation*}
                V \circ s_z \in \cD(\bA_{\theta_1,\theta_2}), \qquad \Xi_{z_0,p,z_1} \circ s_z \in \cD(\bA_{\theta_1,\theta_2})
            \end{equation*}
            for any $\theta_1 \in \Theta_1$ and $\theta_2 \in \Theta_2$ and $z \in \overline{B_1(0)}$,
            \item \label{item:compatA:cont} the maps
            \begin{align*}
            (\theta_1,\theta_2,x,z_0,p,z_1,z) & \mapsto \bA_{\theta_1,\theta_2} \left( \Xi_{z_0,p,z_1} \circ s_z\right)(x),\\
            (\theta_1,\theta_2,x,z) & \mapsto \bA_{\theta_1,\theta_2}\left(V \circ s_z\right)(x)
            \end{align*}        
            are continuous,
            \item \label{item:compatA:contf} the map
            \begin{equation*}
                (\theta_1,\theta_2) \mapsto \bA_{\theta_1,\theta_2} f(x)
            \end{equation*}
            is continuous for any $x \in E$ and $f \in \bigcap_{\theta_1 \in \Theta_1, \theta_2 \in \Theta_2} \cD(\bA_{\theta_1,\theta_2})$.
        \end{enumerate}

        \item\label{item:compatB} Let the collection $\{\bB_{\theta_1, \theta_2}\}_{\theta_1 \in \Theta_1, \theta_2 \in \Theta_2}$ be \emph{compatible} with $V$, $\{\zeta_{z,p}\}_{z \in E, p \in \bR^q}$, and $\{\zeta_z\}_{z \in E}$, i.e.,
        \begin{enumerate}[(1)]
            \item \label{item:compatB:domain} 
            we have
            \begin{equation*}
                V \circ s_z \in \cD(\bB_{\theta_1,\theta_2}), \qquad \Xi_{z_0,p,z_1} \circ s_z  \in \cD(\bB_{\theta_1,\theta_2})
            \end{equation*}
            for any $\theta_1 \in \Theta_1$ and $\theta_2 \in \Theta_2$ and $z \in \overline{B_1(0)}$,
            \item \label{item:compatB:cont} the maps
            \begin{align*}
            (\theta_1,\theta_2,x,z_0,p,z_1) & \mapsto \bB_{\theta_1,\theta_2} \Xi_{z_0,p,z_1} (x),\\
            (\theta_1,\theta_2,x) & \mapsto \bB_{\theta_1,\theta_2} V(x)
            \end{align*}        
            are continuous,
            \item \label{item:compatB:contf} the map
            \begin{equation*}
                (\theta_1,\theta_2) \mapsto \bB_{\theta_1,\theta_2} f(x)
            \end{equation*}
            is continuous for any $x \in E$ and $f \in \bigcap_{\theta_1 \in \Theta_1, \theta_2 \in \Theta_2} \cD(\bB_{\theta_1,\theta_2})$.
        \end{enumerate}
    \end{enumerate}
\end{assumption}

\section{Application to partial integro-differential operators}\label{section:examples}

% \todo[inline]{R: I think we should think about whether we actually want to check compatibility. It is really straightforward, except perhaps for the integral operators. Not sure. 
% These functions are smooth, operator coefficients are continuous, so this is all completely trivial? Perhaps only integral term is a bit special here.}

In this section, we discuss the application of our framework to partial integro-differential operators of the type 
\begin{multline} \label{eqn:operator_linear_examples}
    \bH f(x) = \ip{b(x)}{\nabla f (x)} + \frac{1}{2} \Tr\left(\Sigma\Sigma^T (x) D^2 f(x) \right) \\ 
    + \int \left[f(x+\bfz) - f(x) - \rchi_{B_1(0)}(\bfz) \ip{\bfz}{\nabla f(x)}\right] \mu_x(\dd \bfz) + \cH(\nabla f(x)),
\end{multline}
which, for simplicity, we consider without Bellman or Isaacs structure.
More precisely, we split $\bH$ into $\bA + \bB$ with
\begin{equation}
    \bA = \frac{1}{2} \Tr\left(\Sigma\Sigma^T (x) D^2 f(x) \right) + \int \left[f(x+\bfz) - f(x) - \rchi_{B_1(0)}(\bfz)  \ip{\bfz}{\nabla f(x)}\right] \mu_x(\dd \bfz),
\end{equation}
and 
\begin{equation}
    \bB = \ip{b(x)}{\nabla f (x)} + \cH(\nabla f(x))
\end{equation}
and specify conditions under which 
\begin{itemize}
    \item we can find a Lyapunov function $V$ and construct a coupling that has controlled growth for $\bA$,
    \item we can find a Lyapunov function $V$ and establish local semi-monotonicity for $\bB$,
    \item we can verify that $\bA$ and $\bB$ are compatible.
\end{itemize}

For the verification of compatibility, we need to choose $V$ and families $\{\zeta_{z,p}\}_{z \in E, p \in \bR^q}$ and $\{\xi_z\}_{z \in E}$. We introduce two related families: The first family is suitable for local operators; the second is constructed from the first by suitable cut-off procedures, thus making them suitable for integral operators.

We point out that the second family is suitable for local operators as well. However, this comes at the cost of minor complexity in their construction.

\begin{definition}[Example containment function and penalizations] \label{definition:penalizations}
\ Consider the containment function
\begin{equation*}
    V(x) = \log \left(1+\frac{1}{2}x^2\right)
\end{equation*}
and the following two collections of penalization functions:
    \begin{description}
        \item[Collection 1]\label{definition:canonical_penalization} The base penalizations are
        %Consider the following collection:
          \begin{align*}
        \zeta_{z,p}(x) &= \ip{p}{x-z},\\
        \xi_{z}(x) &= \frac{1}{2} d^2(x,z).
    \end{align*}
    \item[Collection 2] \label{definition:levy_penalizations}
    Let $R'' > R' > R$ with $R$ as in Definition \ref{definition:perturbation_first_second_order}. 
    Let $\overline{\ell}: [0, \infty) \rightarrow [0, \infty)$ be a smooth function satisfying $\overline{l}(r) = 1$ for $r < R'$ and $\overline{l}(r) = 0$ for $x > R''$. 
    Let
    \begin{align*}
        \overline{\xi}_z (x) &= (1-\overline{\ell}(d(x, z)))(R'' + 1)^2 + \overline{\ell}(d(x, z)) \frac{1}{2} d^2(x,z),\\
        \overline{\zeta}_{p, z} (x) &= \overline{\ell} (d(x, z)) \ip{p}{x-z}.
    \end{align*}
    \end{description}
\end{definition}

 \begin{comment}
\begin{definition}[Example containment function and penalizations for integral operators]\label{definition:levy_penalizations}

    Let $R'' > R' > R$ with $R$ as in Definition \ref{definition:perturbation_first_second_order}. 
    Let $\overline{\ell}: [0, \infty) \rightarrow [0, \infty)$ be a smooth function \ri{satisfying $\overline{l}(r) = 1$ for $r < R'$ and $\overline{l}(r) = 0$ for $x > R''$. \sout{on that is $1$ for $x< R'$ and $0$ for $x>R''$.}}
    Let
    \begin{align*}
        V(x) &= \log \left(1+\frac{1}{2}x^2\right),\\
        \overline{\xi}_z (x) &= (1-\overline{\ell}(d(x, z)))(R'' + 1)^2 + \overline{\ell}(d(x, z)) \frac{1}{2} d^2(x,z),\\
        \overline{\zeta}_{p, z} (x) &= \overline{\ell} (d(x, z)) \ip{p}{x-z}.
    \end{align*}
\end{definition}
\end{comment}

As all considered functions are smooth, part \ref{item:compatA:domain} of the compatibility assumptions for $\bA$ and $\bB$, cf. Assumptions \ref{assumption:compatibility} \ref{item:compatA} and \ref{item:compatB}, hold for every part of $\bH$ except the integral term immediately. %For the intregral term, we introduce a new set of penalizations, for which $\zeta$ and $\xi$ are replaced by variants that are constant outside of a compact set.

\smallskip

To simplify the verification of our conditions, we have the following two observations.

% {\color{red} REMOVE BELOW?
% More precisely, we will show that the operator is compatible as in Definition \ref{definition:compatible} and that we have control on the action of operator on the containment function $V$, i.e. that the bound in equation \eqref{eq:mainth_HJI_bound} holds, for 
% \begin{equation}
%     \bA = \frac{1}{2} \Tr\left(\Sigma^T \Sigma (x) D^2 f(x) \right) + \int \left[f(x+\bfz) - f(x) - \bONE_{B_1(0)} \ip{\bfz}{\nabla f(x)}\right] \mu_x(\dd \bfz),
% \end{equation}
% and 
% \begin{equation}
%     \bB = s(x) + \ip{b(x)}{\nabla f (x)} + \cH(\nabla f(x))
% \end{equation}
% with the choices
% \begin{align*}
%     V(x) &= \log \left(1+\frac{1}{2}x^2\right),\\
%     \zeta_{z,p}(x) &= \ip{p}{x-z},\\
%     \xi_{z}(x) &= \frac{1}{2} d^2(x,z)
% \end{align*}
% as discussed in Section \ref{section:perturbations_definitions}. Note that, for all $z\in E$ and $p \in \bR^q$, $V, \zeta_{z, p}, \xi_{z} \in C^\infty(E)$. Consequently, part \ref{item:compatA:domain} of the compatibility assumption for $\bA$ and $\bB$, cf. Definitions  \ref{definition:compatible} \ref{item:compatA} and \ref{item:compatB}, hold for every part of $\bH$ except the integral term immediately.

% Note that, by the following remarks, we can indeed show the assumption for each term of $\bA$ and $\bB$ separately. }
\begin{remark} \label{remark:additive_linear}
    Let $\bA_1, \bA_2 \subseteq C(E) \times C(E)$ be linear on their respective domains and compatible with $V$, $\{\zeta_{z,p}\}_{z \in E, p \in \bR^q}$, and $\{\zeta_z\}_{z \in E}$ and with associated controlled growth couplings $\widehat{\bA}_1,\widehat{\bA}_2 \subseteq C(E^2) \times C(E^2)$.
    Then, the operator $\bA := \bA_1 + \bA_2$ is linear on its domain and compatible with $V$, $\{\zeta_{z,p}\}_{z \in E, p \in \bR^q}$, and $\{\zeta_z\}_{z \in E}$ and with associated controlled growth coupling $\widehat{\bA} := \widehat{\bA}_1 + \widehat{\bA}_2$.
\end{remark}

\begin{remark} \label{remark:additive_nonlinear}
    Let $\bB_1, \bB_2 \subseteq C(E) \times C(E)$ be compatible with $V$, $\{\zeta_{z,p}\}_{z \in E, p \in \bR^q}$, and $\{\zeta_z\}_{z \in E}$ and convex semi-monotone operators. Then $\bB := \bB_1 + \bB_2$ is compatible with $V$, $\{\zeta_{z,p}\}_{z \in E, p \in \bR^q}$, and $\{\zeta_z\}_{z \in E}$ and convex semi-monotone operator.
\end{remark}

The rest of this section is organized as follows:
\begin{itemize}
    \item In Section \ref{section:sub:drift}, we consider drift terms and convex first-order Hamiltonians;
    \item In Section \ref{section:sub:diffusion}, we consider diffusion operators;
    \item In Section \ref{section:sub:integral}, we consider integral operators.
\end{itemize}

\subsection{Deterministic Example: Drift terms and convex first-order Hamiltonians} \label{section:sub:drift}

In this section, we consider the deterministic part of the operator \eqref{eqn:operator_linear_examples}.

\begin{proposition} \label{proposition:ex_drift}
    Suppose that $\bB$ is given by 
    \begin{equation}
        \bB f(x) = \ip{b(x)}{\nabla f (x)} + \cH(\nabla f (x))
    \end{equation}
    with the drift term $x \mapsto b(x)$ locally, one-sided Lipschitz with constant $L_{b,K}$ and $\vn{b(x)} \leq \frac{c_b}{2}(1+\vn{x})$ for some constant $c_b > 0$, and $p \mapsto \cH(p)$ continuous and convex.\\
    Then, $\bB$ is compatible with both collections of Definition \ref{definition:penalizations}, cf. Assumption \ref{assumption:compatibility} \ref{item:compatB}, and convex semi-monotone. Furthermore, $V=\log(1+\frac{x^2}{2})$ is a Lyapunov function:
    \begin{equation*}
    \sup_{x \in E} \bB V(x) < \infty.
    \end{equation*}
\end{proposition}

\begin{proof}
    {\bfseries Convex semi-monotonicity:}
    Clearly, $\bB$ is locally first-order with $\bB f(x) = \ip{b(x)}{\nabla f (x)} + \cH(\nabla f (x)) = \cB(x, \nabla f(x))$. Additionally, for any compact set $K \subseteq E$ $\alpha >0$, and $x,x' \in K$, we have
    \begin{align*}
        \cB(x,\alpha(x-x')) - \cB(y,\alpha(x-x')) &= \ip{b(x)}{\alpha(x-x')} + \cH(\alpha(x-x'))\\
        &\qquad - \ip{b(x')}{\alpha(x-x')} - \cH(\alpha(x-x'))\\
        &= \ip{b(x) - b(x')}{\alpha (x-x')} \\
        &\qquad+ \cH(\alpha(x-x')) - \cH(\alpha(x-x'))\\
        &\leq \alpha L_{b, K} d^2(x,x'), 
    \end{align*}
    establishing semi-monotonicity. As convexity of $p \mapsto \cB(x,p)$ is immediate, we conclude that $\bB$ is convex semi-monotone.

    {\bfseries Lyapunov control:}
    Using that $V(x) = \log\left(1+\frac{x^2}{2}\right)$, $\nabla V (x) = \frac{2x}{2+|x|^2}$ is bounded as a function of $x$, $b$ has linear growth, and that $\cH$ is continuous, we find that
    \begin{equation}
        \sup_{x\in E} \bB V (x) = \sup_{x\in E} \ip{b(x)}{\frac{2x}{2+|x|^2}} + \cH \left( \frac{2x}{2+|x|^2} \right) < \infty.
    \end{equation}

    {\bfseries Compatibility:}
    %As all considered functions and gradients are continuous, Assumption \ref{assumption:compatibility} \ref{item:compatB} is immediate.
    We show the compatibility of $\bB$, cf. Assumption \ref{assumption:compatibility} \ref{item:compatB}, by evaluation of the perturbation and containment function in the operator.\\
    Using $\xi_z(x) = \frac{1}{2} d^2(x,z)$ and $\zeta_{z, p}(x) = \ip{p}{x - z}$, we find for $z_0, z_1, z \in E$ and $p \in B_1(0)$
    \begin{multline*}
        \bB (\Xi_{z_0, p, z_1} \circ s_z) (x) = \ip{b(x)}{(x -z - z_0) + p + (x -z - z_1)}\\
        + \cH\left((x -z - z_0) + p + (x -z - z_1)\right),
    \end{multline*} 
    which is continuous in $(x, z_0, p, z_1, z)$ as $b$ and $\cH$ are continuous.
    For $V(x) = \log \left(1 + \frac{1}{2} x^2\right)$ and $z \in E$, we find
    \begin{equation}\label{eq:first_order_ex:lyapunov}
        \bB (V \circ s_z) (x) = \ip{b(x)}{\frac{2(x-z)}{2+|x-z|^2}} + \cH\left( \frac{2(x-z)}{2+|x-z|^2} \right),
    \end{equation}
    which is continuous in $(x,z)$ as $b$ and $\cH$ are continuous. Thus, $\bB$ is compatible.
\end{proof}

\subsection{Stochastic Example: Diffusion operators}\label{section:sub:diffusion}
In this section, we focus on diffusion operators of the form
\begin{equation} \label{eqn:diffusion_example}
    \bA f(x) = \frac{1}{2} \Tr \left( \Sigma(x)\Sigma^T(x) D^2 f(x) \right),
\end{equation}
where $\Sigma(x)$ is a positive semi-definite matrix for each fixed $x \in E$.

Our main goal is to construct a controlled growth coupling for the operator $\bA$. To illustrate the idea behind our approach, consider the simpler case of the Laplacian operator
\begin{equation}
    \bA_0 f (x) = \frac{1}{2} \Tr(D^2 f(x)),
\end{equation}

which is the infinitesimal generator of Brownian motion. The well-known synchronous coupling of two Brownian motions started from $x$ and $x'$, respectively, is given by
$(X(t),X'(t)) = (x+B(t),x'+B(t))$ with $B(t)$ a standard Brownian motion, having generator
\begin{equation*}
    \widehat{\bA}_0 g(x,x') = \frac{1}{2} \left(\partial_x + \partial_{x'}\right)^2 g(x,x'),
\end{equation*}
which satisfies $\widehat{\bA}_0 d^2 = 0$. Aiming to generalize this, we rewrite
\begin{equation*}
    \widehat{\bA}_0 g(x,x') = \Tr \left(C C^T D^2 g(x,x') \right) \quad \text{with} \quad C= \frac{1}{\sqrt{2}} 
    \begin{pmatrix}
        \bONE & \bONE \\ \bONE & \bONE
    \end{pmatrix}.
\end{equation*}
In general we obtain the following result.

\begin{comment}

In this setting, a natural coupling for $\bA_0$ is given by
\begin{equation*}
    \widehat{\bA}_0 g(x,y) := \frac{1}{2} \left(\partial_x + \partial_y\right)^2 g(x,y) = \frac{1}{2} \Tr \left( 
    \begin{pmatrix}
        \bONE & \bONE \\ \bONE & \bONE
    \end{pmatrix} 
    D^2 g(x,y) \right) = \Tr \left(C C^T D^2 g(x,y) \right)
\end{equation*}
with $C = \frac{1}{\sqrt{2}} 
    \begin{pmatrix}
        \bONE & \bONE \\ \bONE & \bONE
    \end{pmatrix}$.

This operator $\widehat{\bA}_0$ is, in fact, the generator of the process that arises from the synchronous coupling of two Brownian motions starting in $x$ and $y$, respectively.

To extend this idea and construct a coupling for the operator $\bA$, the key task is to determine the appropriate matrix $CC^T$. The following lemma addresses this construction.
\end{comment}

\begin{proposition}\label{proposition:coupling_diffusion}
    Suppose that $\bA$ is given by
    \begin{equation*}
        \bA f(x) = \frac{1}{2} \Tr \left( \Sigma(x)\Sigma^T(x) D^2 f(x) \right) 
    \end{equation*}
    with $\Sigma(x)$ positive semi-definite for all $x \in E$, $x \mapsto \Sigma(x)$ locally Lipschitz with constant $L_{\Sigma,K}$ and  $\vn{b(x)} \leq \frac{c_\Sigma}{2}(1+\vn{x})$ for some constant $c_\Sigma > 0$. Consider 
    \begin{equation*}
        \widehat{\bA}f(x,y) := \Tr \left(\widehat{\Sigma}^2(x,x') D^2 f(x,x') \right),
    \end{equation*}
    where
    \begin{equation*}
        \widehat{\Sigma}^2(x,y) := \begin{pmatrix}
            \Sigma(x)\Sigma^T(x) & \Sigma(x')\Sigma^T(x) \\
            \Sigma(x)\Sigma^T(x') & \Sigma(x')\Sigma^T(x') 
        \end{pmatrix}.
    \end{equation*}
    Then, $\bA$ is compatible, cf. Assumption \ref{assumption:compatibility} \ref{item:compatA}, linear on its domain, and admitting the controlled growth coupling $\widehat{\bA}$. Furthermore, $V=\log(1+\frac{x^2}{2})$ is a Lyapunov function:
    \begin{equation*}
        \sup_x \bA V(x) < \infty.
    \end{equation*}
\end{proposition}

% Before we proceed with the proof, note that the matrix $\widehat{\Sigma}^2(x,y)$ is positive semi-definite. The following lemma highlights that the operator $\widehat{\bA}$ is indeed a coupling for $\bA$. The proof is straightforward.

% \begin{lemma}\label{lemma:coupling_diffusion_1}
%     Let 
%     \begin{equation*}
%         Af(x) = \frac{1}{2} \Tr\left(B(x)D^2 f(x)\right).
%     \end{equation*}
%     For any $x,y \in E$, let $\widehat{B}(x,y)$ be a positive semi-definite matrix  such that it has a block-structure satisfying
%     \begin{equation*}
%         \widehat{B}(x,y) = \begin{pmatrix}
%             B(x) & B(x,y) \\
%             B(x,y)^T & B(y)
%         \end{pmatrix}
%     \end{equation*}
%     Define
%     \begin{equation*}
%         \widehat{A}f(x,y) := \frac{1}{2} \Tr \left(\widehat{B}(x,y) D^2f(x,y) \right).
%     \end{equation*}
%     Then, $\widehat{A}$ is a coupling of $A$ as
%     \begin{align}
%         \widehat{A}(f_1 \oplus f_2)(x,y)  &= \frac{1}{2}\Tr\left(\widehat{B}(x,y) D^2(f_1 \oplus f_2)(x,y) \right) \\
%         &= \frac{1}{2}\Tr\left( B(x) D^2f(x) \right) + \frac{1}{2}\Tr\left( B(y) D^2f(y) \right)\\
%         &= Af_1 + Af_2
%     \end{align}
%     and it satisfies the maximum principle. 
% \end{lemma}

For the proof we make use of the following auxiliary lemma.

\begin{lemma}\label{lemma:coupling_diffusion_1}
    For each $x \in E$, let $B(x)$ be a positive semi-definite matrix and consider
    \begin{equation*}
        \bA f(x) = \frac{1}{2} \Tr\left(B(x)D^2 f(x)\right).
    \end{equation*}
    For any $x,x' \in E$, let $\widehat{B}(x,x')$ be a positive semi-definite matrix having block-structure
    \begin{equation*}
        \widehat{B}(x,x') = \begin{pmatrix}
            B(x) & B(x,x') \\
            B(x,x')^T & B(x')
        \end{pmatrix}.
    \end{equation*}
    Define
    \begin{equation*}
        \widehat{\bA}f(x,x') := \frac{1}{2} \Tr \left(\widehat{B}(x,x') D^2f(x,x') \right).
    \end{equation*}
    Then, $\widehat{\bA}$ is a coupling of $\bA$.
\end{lemma}

\begin{proof}
    % {\color{red} 
    % REDO.
    
    % \sout{As $\widehat{A}$ is a second-order differential operator, it satisfies the maximum principle. By construction, $\widehat{A}$ is a coupling of $A$.}}
    \begin{align}
        \widehat{\bA}(f_1 \oplus f_2)(x,y)  &= \frac{1}{2}\Tr\left(\widehat{B}(x,y) D^2(f_1 \oplus f_2)(x,y) \right) \\
        &= \frac{1}{2}\Tr\left( B(x) D^2f(x) \right) + \frac{1}{2}\Tr\left( B(y) D^2f(y) \right)\\
        &= \bA f_1 + \bA f_2
    \end{align}
    and it satisfies the maximum principle.
\end{proof}

\begin{proof}[Proof of Proposition \ref{proposition:coupling_diffusion}]
    {\bfseries Controlled growth coupling:}
    By Lemma \ref{lemma:coupling_diffusion_1}, $\widehat{\bA}$ is a coupling for $\bA$. We thus verify that $\widehat{\bA}$ has controlled growth. Consider $\alpha >1$, $K \subseteq E$ a compact set, and $x, x',y ,y' \in K$. Then,
    \begin{align*}
        \widehat{\bA} \left(\frac{\alpha}{2} d^2_{x-y, x'-y'} \right)(x,x') &= \frac{1}{2}  \Tr\left(\widehat{\Sigma}^2(x,x') D^2\left(\frac{\alpha}{2} d^2_{x-y, x'-y'} \right)(x,x') \right)\\
        & = \frac{1}{2} \Tr\left(\widehat{\Sigma}^2(x,x') \left(\alpha
        \begin{pmatrix}
            \bONE & - \bONE \\
            - \bONE & \bONE
        \end{pmatrix}
        \right)(x,x') \right)\\
        & = \frac{\alpha}{2} \Tr((\Sigma^T(x)-\Sigma^T(y))(\Sigma(x)-\Sigma(y)))\\
        & \leq \alpha L^2_{\Sigma, K} d^2(x, x'),
    \end{align*}
    establishing controlled growth.

    {\bfseries Lyapunov control:}
    Using $V(x) = \log (1+\frac{x^2}{2})$ and the fact that $\Sigma$ has linear growth, we find that
    \begin{equation} \label{eq:second_order_ex:lyapunov}
        \sup_{x\in E} \bA V (x) = \sup_{x\in E} \frac{1}{2}\Tr \left( \Sigma(x)\Sigma^T(x) D^2 V(x) \right) < \infty.
    \end{equation}

    {\bfseries Compatibility:}
    Using $\xi_z(x) = \frac{1}{2} d^2(x,z)$, $\zeta_{z, p}(x) = \ip{p}{x - z}$ and $V(x)=\log(1+\frac{x^2}{2})$, we find for $z_0, z_1, z \in E$ and $p \in B_1(0)$
    \begin{align*}
        \bA (\Xi \circ s_z) (x) &= 2 \Tr(\Sigma(x)\Sigma^T(x)),\\
        \bA (V \circ s_z) (x) &= \frac{1}{2}\Tr \left( \Sigma(x)\Sigma^T(x) D^2 (V \circ s_z)(x) \right), 
    \end{align*}
    which, by an analogous calculation as in equation \eqref{eq:second_order_ex:lyapunov}, is continuous in $(x, z_0, p, z_1, z)$ and $(x,z)$. Consequently, $\bA$ is compatible.
\end{proof}

\subsection{Stochastic Example: Integral operators} \label{section:sub:integral}
In this section, we cover examples of spatially inhomogeneous Lévy processes that have generators of the type
    % \begin{equation} \label{eqn:def_of_Levy_type}
    %     \bA f(x) = \int \left[f(x+\bfz) - f(x) - \bONE_{B_1(0)} \ip{\bfz}{\nabla f(x)}\right] \mu_x(\dd \bfz)
    % \end{equation}
    \begin{equation} \label{eqn:def_of_Levy_type}
        \bA f(x) = \int \left[f(x+\bfz) - f(x) - \rchi_{B_1(0)}(\bfz) \ip{\bfz}{\nabla f(x)}\right] \mu_x(\dd \bfz),
    \end{equation}
where
$\rchi_{B_1(0)}(\bfz) = l (|\bfz|)$ for some smooth non-decreasing function $l$ satisfying $l = 1$ on a neighborhood of $0$ and $l(r) = 0$ for $r \geq 1$.

\smallskip

We next specify the space from which we can take our jump measures $\mu_x$. For this, we need to control the mass close to $0$ as for large values of $\bfz$. The following function controls both:
\begin{equation*}
    W(\bfz) := \rchi_{B_1(0)}(\bfz) |\bfz|^2 + (1-\rchi_{B_1(0)}(\bfz)) \log \left( 1 + |\bfz|^2 \right).
\end{equation*}
We take the family of jump measures $\{\mu_x\}_{x \in E}$ from the set of equivalence classes $\cM_W(\bR^q) := \cM(\bR^q)/\sim$ with
\begin{equation}
    \cM(\bR^q) \coloneqq \left\{ \mu \in M(\bR^q) \,\middle|\, \int W(\bfz) \, \mu(\dd \bfz) < \infty \right\},
\end{equation}
where $M(\bR^q)$ is the set of all Borel measures on $\bR^q$ and where
\begin{equation*}
    \mu \sim \nu \qquad \text{if and only if} \qquad \mu|_{\bR^q \setminus \{0\}} = \nu|_{\bR^q \setminus \{0\}}.
\end{equation*}
We topologize the set $\cM_W(\bR^q)$ by the weak topology $\sigma_W$ induced by the pairings
\begin{equation} \label{equation:definition_sigmaW}
    \mu \mapsto \int g(\bfz) \mu(\dd \bfz) \qquad \forall \, g \in C_W,
\end{equation}
where
\begin{equation*}
    C_W \coloneqq \left\{g \in C(\bR^q) \, \middle| \,  g(0) = 0, \text{ and } \sup_{\bfz \neq 0} \frac{|g(\bfz)|}{W(\bfz)} < \infty. \right\}.
\end{equation*}
Below, we construct controlled growth couplings for operators of the type \eqref{eqn:def_of_Levy_type}. To clarify the concepts, we consider the example of an uncompensated process, i.e., having an operator of the type
\begin{equation*}
    \bA f(x) = \int f(x+\bfz) - f(x) \mu_x(\dd \bfz).
\end{equation*}
Couplings for this type of operator are of the form
\begin{equation} \label{eqn:coupled_jump_operator}
    \widehat{\bA} f(x,x') = \int \left[f(x+\bfz_1,y+\bfz_2) - f(x,x')\right] \pi_{x,x'}(\dd \bfz_1,\bfz_2),   
\end{equation}
where $\pi_{x,x'}$ couples $\mu_x$ and $\mu_{x'}$. In the following example, we illustrate the need of being able to couple jumps synchronously.

% In the first example, we illustrate the need of being able to couple jumps. The second highlights the possibility of location dependent jump measures. The last example shows that $\mu_x$ does not necessarily have equal mass across different $x$, indicating the need to couple non-zero jumps to jumps with size $0$.

% \todo[inline]{R: fix sentence above, as only one is left}

\begin{example}[Random Walk]
Consider the simple random walk on $\bR$ making jumps of size $1$, i.e $\mu_x = \mu = \delta_{-1} + \delta_{1}$ leading to the operator
\begin{equation*}
    \bA f(x) = \left[f(x-1)  + f(x+1) - 2f(x)\right].
\end{equation*}
% A coupling will thus be an operator of the type \eqref{eqn:def_of_Levy_type} on $E\times E = \bR \times \bR$:
% \ri{
% \begin{equation*}
%     \widehat{\bA} f(x,x') = \int \left[f(x+\bfz_1,y+\bfz_2) - f(x,x')\right] \pi_{x,x'}(\dd \bfz_1,\bfz_2)   
% \end{equation*}
% }
% with appropriate two-dimensional integrability conditions. 
Well known couplings include walks with simultaneous jumps but independent directions, fully independent jumps, and synchronous jumps. The corresponding generators are given as in \eqref{eqn:coupled_jump_operator} with jump measures
\begin{align*}
    \pi^1 & := \mu \otimes \mu, \\
    \pi^2 & := \delta_{(-1,0)} + \delta_{(1,0)} + \delta_{(0,-1)} + \delta_{(0,1)} , \\
    \pi^3 & := \delta_{(-1,-1)} + \delta_{(1,1)},
\end{align*}
respectively. This leads to the operators
\begin{align*}
    \widehat{\bA}^1 f(x,x') & = 
             f(x-1,x'-1) + f(x-1,x'+1)  \\
    & \qquad \quad + f(x+1,x'-1) + f(x+1,x'+1) - 4 f(x,x'), \\
    \widehat{\bA}^2f(x,x') & = f(x-1,x') + f(x+1,x') - 2f(x,x')  \\
    & \qquad\quad + f(x,x'-1) + f(x,x'+1) - 2f(x,x') , \\
    \widehat{\bA}^3 f(x,x') & = f(x-1,x'-1) + f(x+1,x'+1) - 2f(x,x') .
\end{align*}
Only for the final example, we see that $\widehat{\bA}^3 d^2 \leq 0$, pointing at the necessity of the alignment of jumps.
\end{example}

Note that the third coupling above has different total mass, and we thus work outside the realm of the typical notion of couplings of probability measures. A second feature of coupling jump measures, not present in the example above, is that we can make one process jump, whereas the other does not.

We formalize this in the following definition.

% {\color{red} REMOVE? Move \cite{FiGi10} to a remark below?

% In general, we need to construct a measure $\pi_{x,x'}$ that combines $\mu_x$ and $\mu_x'$. The probabilistic notion of couplings, however, does not exactly work, as we are not working with probability measures, and, additionally, $\mu_x$ and $\mu_x'$ might not have the same total mass.

% This can be remedied by realizing that the amount of mass at $0$, i.e. jumps of size $0$, is irrelevant. The following notion of an extended coupling was introduced in \cite{FiGi10} in the context where mass can be moved to the boundary of a domain. In our context, this boundary is the point $0$.}

\begin{definition}
    Let $\mu,\nu \in \cM_W(\bR^q)$. We say that $\pi \in M(\bR^q \times \bR^q)$ is an \emph{extended coupling of $\mu$ and $\nu$}, if
    \begin{align*}
		\pi\left((A \setminus \{0\}) \times \bR^q \right) & = \mu(A \setminus \{0\}) \qquad \forall \, A \in \cB(\bR^q), \\
		\pi\left(\bR^q \times (B \setminus\{0\})\right) & = \nu(B \setminus \{0\}) \qquad  \forall \, B \in \cB(\bR^q).
	\end{align*}
\end{definition}

\begin{remark}
    A variant of this coupling was introduced in \cite{FiGi10}. There mass can be moved to the boundary of a domain. In our context, this boundary is the point $0$.
\end{remark}

% \todo[inline]{R : introduce topology $\sigma_{\widehat{W}}$ on $\cM_{\widehat{W}}(\bR^q \times \bR^q)$? by using $\widehat{W}(\bfz_1,\bfz_2) = W(\bfz_1) + W(\bfz_2)$.

% This is a bit special... You can make your tensor product in various ways, we do not want to take out only $(0,0)$ and not the union $\{(x,0)\} \cup \{(0,x)\}$}

\begin{definition}\label{def:extended_coupling}
    Let $x \mapsto \mu_x$ be a map from $E$ into $\cM(\bR^q)$. Let $(x,x') \mapsto \pi_{x,x'}$ be a map from $E^2$ into $M(\bR^q \times \bR^q)$.
    
    \begin{enumerate}[(a)]
        \item\label{item:extended_coupling:mu2pi} We say that $(x,x') \mapsto \pi_{x,x'}$ is an extended coupling of $x \mapsto \mu_x$, if for all $x,x' \in E$, we have that $\pi_{x,x'}$ is an extended coupling of $\mu_x$ and $\mu_x'$.
        \item\label{item:extended_coupling:lip} We say that $(x,x') \mapsto \pi_{x,x'}$ is locally Lipschitz, if, for any compact set $K \subseteq E$, there exits a constant $L_{\pi,K}$ such that, for $x,x' \in K$, we have
        \begin{equation*}
            \int  d^2(\bfz_1,\bfz_2) \pi_{x,x'}( \dd \bfz_1, \dd \bfz_2) \leq L_{\pi,K} d^2(x,x').
        \end{equation*}
    \end{enumerate}
\end{definition}

\begin{remark}
    Note that conditions $(12)$, $(34)$, and $(35)$ in \cite{BaIm08} for $\mu$ and $j$ correspond to our choice of $\cM_W(\bR^q)$ and locally Lipschitz extended coupling $\pi_{x,x'}$.
\end{remark}

% Corresponding to Example \ref{example:Levy_coupling_from_Lipschitzmap}, we obtain the following.
 
\begin{remark}
Let $\eta: \bR \rightarrow \bR$ be any locally Lipschitz map with local Lipschitz constants $L_{\eta,K}$. Set $\mu_x := \delta_{\eta(x)} \bONE_{\eta(x) \neq 0} (x)$ and $\pi_{x,x'} = \delta_{(\eta(x),\eta(x'))}$. 
Then, $(x,x') \mapsto \pi_{x,x'}$ is a locally Lipschitz coupling of $x \mapsto \mu_x$ with $L_{\pi,K} = L_{\eta,K}$.
\end{remark}

% \todo[inline]{R: perhaps for the revision. Does local Lipschitz imply continuity for $\sigma_{\widehat{W}}$?}

The main proposition of this subsection below aims to show that integral operators of the form \eqref{eqn:def_of_Levy_type} can be treated analogous to the other examples above. We work with the second collection of penalization functions, cf. Definition \ref{definition:levy_penalizations}, to avoid integrability issues.

%While we use the same containment function to deal with possible integrability issues, we \ri{work our second class of penalization functions, cf. Definition \ref{definition:levy_penalizations}. \sout{smoothly cut off the penalization in Definition \ref{definition:canonical_penalization}.}}

% using the typical choices of the containment and perturbation functions. To get around possible integrability issues, we smoothly cut off $\Xi_{z_0, p, z_1}$ outside of $B_{R'}(z_0)$ for some $R'>R$ with $R$ as in Definition \ref{definition:perturbation_first_second_order}. For this purpose we can, for example, use the cutoff $\Omega^+_{M}$, cf. Definition \ref{def:cutoff_function}, which we also use for the construction of the test functions in Section \ref{section:testfunctions_construction}.
% For the remainder of the section, we denote the cut off perturbations as
% \begin{equation}\label{eq:NonLocEx:Xi}
%     \overline{\Xi}_{z_0, p, z_1} (x)= \Omega^+_{M'} \circ \Xi_{z_0, p, z_1} (x),
% \end{equation}
% where $M' = \sup_{x \in B_{R'}(z_0)} \Xi_{z_0, p, z_1} (x)$.

\begin{proposition}\label{proposition:coupling_jump}
    Consider
    \begin{equation*}
        \bA f(x) = \int \left[f(x+\bfz) - f(x) - \rchi_{B_1(0)} \ip{\bfz}{\nabla f(x)}\right] \mu_x(\dd \bfz).
    \end{equation*}
    Suppose there exists a $\sigma_W$-continuous map $x \mapsto \mu_x$ in $\cM_W(\bR^q)$, cf.\ \eqref{equation:definition_sigmaW},
    and that there exists a locally Lipschitz extended coupling $(x,x') \mapsto \pi_{x,x'}$ of $x \mapsto \mu_x$ with Lipschitz constant $L_{\pi, K}$ and, for $\widehat{\rchi}(\bfz_1,\bfz_2) := \rchi_{B_1(0)}(\bfz_1) \rchi_{B_1(0)}(\bfz_2)$, set
    \begin{multline*}
        \widehat{\bA}g(x,x') := \int \Big[ g(x+\bfz_1,x'+\bfz_2) - g(x,x')\\
        - \widehat{\rchi}(\bfz_1,\bfz_2) \ip{(\bfz_1, \bfz_2)^T}{\nabla g(x, x')}\Big] \pi_{x,x'}(\dd \bfz_1, \dd \bfz_2).
    \end{multline*}
    Assume furthermore that
    \begin{equation*}
        \sup_{x \in E} \int \log\left(1 +  \frac{\frac{1}{2}|\bfz|^2  +  \ip{x}{\bfz}}{1+ \frac{1}{2} |x|^2 } \right) \mu_x (\dd \bfz) < \infty.
    \end{equation*}

    % \todo[inline]{We need additional assumptions. Building up the list here, nice writing can come later.}
% {\color{red}
% R: if all is well, this is superseded by $W$ continuity.

% REMOVE?

%     Assume furthermore that, for $\phi(\bfz) = \log ( 1 + |\bfz|^2)$ and any compact set $K \subseteq \bR^q$, we have
%     \begin{equation} \label{eqn:Levy_uniform_integrability}
%     \lim_{R \rightarrow \infty} \sup_{x \in K}  \int_{|\phi| \geq R} \phi(\bfz) \, \mu_x(\dd \bfz) = 0,
%     \end{equation}
%     % Integrability should be ok based on this assumption, but am struggling with the loss of continuity due to $B_1(0)$.
%     % and for Lyapunov:
%     }
    
    Then, $\bA$ is compatible, cf.\ Definition \ref{assumption:compatibility} \ref{item:compatA}, and linear on its domain admitting the controlled growth coupling $\widehat{\bA}$. Furthermore, $V=\log(1+\frac{x^2}{2})$ is a Lyapunov function:
    \begin{equation*}
        \sup_{x \in E} \bA V(x) < \infty.
    \end{equation*}
\end{proposition}

\begin{remark}
    Corresponding to Remark \ref{remark:martingale_problem}, we refer to \cite[Corollary 2.3]{Bas88} for a uniqueness result for a Lévy process martingale problem.
\end{remark}

The proof of Proposition \ref{proposition:coupling_jump} is based on the following two auxiliary lemmas. In the first, we obtain bounds on the integrand of our operator acting on the Lyapunov function $V$. In the second, we compute the integrand of our Lévy type operator acting on the shifted squared metric. We prove these two lemmas following the proof of Proposition \ref{proposition:coupling_jump}.

% {\color{red} where to put it?}
% \begin{proposition}[Taylor's Theorem] \label{proposition:Taylor}

% {\color{red} citation}

%     Let $f \in C^2(E)$, then for any $\bfz \in B_1(0)$, we have
%     \begin{equation*}
%         \left| f(x+ \bfz) - f(x) - \ip{\nabla f(x)}{\bfz} \right| \leq \frac{1}{2} |\bfz|^2 \ssup{\sup_{i,j}| \nabla_{i,j}^2 f(x + \cdot)|}_{B_1(0)}.
%     \end{equation*}
% \end{proposition}
\newpage

\begin{lemma} \label{lemma:boundLevy}
Fix $x, z \in E$.
\begin{enumerate}[(a)]
    \item \label{item:boundLevy_large}  For $\bfz \in \bR^q$, we have
    \begin{align}
    - \log \left(1 + \frac{1}{2}|x-z|^2 \right) & \leq   V \circ s_z(x + \bfz) - V \circ s_z(x) \\
    & \leq \log\left(1 +  \frac{\frac{1}{2}|\bfz|^2  +  \ip{x-z}{\bfz}}{1+ \frac{1}{2} (x - z)^2 } \right) \label{eqn:lemma:boundLevyLyapunov} \\
    &  \leq \log \left(1 + |\bfz|^2 \right). \label{eqn:lemma:boundLevyDomConvergence} 
    \end{align}   
    \item \label{item:boundLevy_small} For $\bfz \in B_1(0)$, we have
    \begin{align*}
        & \left| V \circ s_z(x + \bfz) - V \circ s_z(x) - \ip{\bfz}{\nabla (V \circ s_z)(x)}\right| \leq \frac{1}{2} |\bfz|^2.
    \end{align*}
\end{enumerate}
\end{lemma}

\begin{lemma} \label{lemma:control_difference_dsquared}
    We have
    \begin{align*}
        & \begin{multlined}
            \frac{1}{2} d^2_{x-y,x'-y'}(x + \bfz_1, x' + \bfz_2) - \frac{1}{2} d^2_{x-y,x'-y'}(x, x') \\ - \widehat{\rchi}(\bfz_1,\bfz_2) \ip{
            \begin{pmatrix}
                \bfz_1 \\ \bfz_2
            \end{pmatrix}
            }{\nabla \left(\frac{1}{2} d^2_{x-y,x'-y'}\right) (x, x')}
        \end{multlined} \\
        &\leq \left(1 - \frac{1}{2}\widehat{\rchi}(\bfz_1,\bfz_2) \right) d^2(\bfz_1,\bfz_2) + (1 - \widehat{\rchi}(\bfz_1,\bfz_2)) \frac{1}{2} d^2(y,y').
    \end{align*}
\end{lemma}

\begin{proof}[Proof of Proposition \ref{proposition:coupling_jump}]
    {\bfseries Controlled growth coupling}:
    As $(x, x')\mapsto\pi_{x,x'}$ is a locally Lipschitz extended coupling of $x\mapsto\mu_x$, cf. Definition \ref{def:extended_coupling}, we have that $\widehat{\bA}$ is a coupling. Thus, we need to verify the controlled growth property of $\widehat{\bA}$.

    Let $x, x', y, y' \in K$ for $K \subseteq E$ a compact set. Using Lemma \ref{lemma:control_difference_dsquared}, we then have
    \begin{align*}
        \widehat{\bA} \left(\frac{\alpha}{2}d^2_{x-y,x'-y'}\right)(x,x') &\leq \frac{\alpha}{2} \int \left(1 - \frac{1}{2}\widehat{\rchi}(\bfz_1,\bfz_2) \right) d^2(\bfz_1,\bfz_2) \pi_{x,x'}(\dd \bfz_1,\dd \bfz_2)\\
        &\qquad + \frac{\alpha}{2} \int (1 - \widehat{\rchi}(\bfz_1,\bfz_2)) \frac{1}{2} d^2(y,y')  \pi_{x,x'}(\dd \bfz_1,\dd \bfz_2) \\ 
        &\leq \frac{\alpha}{2} L_{\pi, K} d^2(x, x') + \frac{\alpha}{4} c'_\pi d^2(y, y'),
    \end{align*}
    where the second inequality is due to the local Lipschitz property of the map $(x,x') \mapsto \pi_{x,x'}$ and $c'_\pi > 0$ exists since, for every $x, x' \in E$, $\pi_{x, x'} \in M(\bR^q \times \bR^q)$. As such, $\bA$ admits the controlled growth coupling $\widehat{\bA}$.

    {\bfseries Lyapunov control:}
    Using Lemma \ref{lemma:boundLevy}, we find
    \begin{equation}\label{eq:non_local_ex:lyapunov}
        \sup_{x \in E} \bA V (x) \leq \sup_{x \in E} \int (1 - \rchi_{B_1(0)})\log(1+|\bfz|^2) + \rchi_{B_1(0)} |\bfz|^2 \mu_x (\dd \bfz)<\infty.
    \end{equation}

    {\bfseries Compatibility:}
    % Using the cut-off version $\overline{\Xi}_{z_0, p, z_1}$, cf. equation \eqref{eq:NonLocEx:Xi}, we find for $x, z_0, z_1, z \in E$ and $p \in B_1(0)$, that there exists a constant $c'_{\overline{\Xi}}>0$ and $c'_{V}>0$ as in equation \eqref{eq:non_local_ex:lyapunov} such that
    % \begin{align*}
    %     \bA(\overline{\Xi}_{z_0, p, z_1} \circ s_z)(x)  &\leq \int c'_{\overline{\Xi}}(1 \wedge |\bfz|^2) \mu_x(\dd \bfz) < \infty, \nonumber\\
    %     \bA (V \circ s_z)(x) &\leq \int c'_V (1+ (|\bfz|\wedge|\bfz|^2))\mu_x (\dd \bfz)<\infty.
    % \end{align*}
    % Consequently, we can use the dominated convergence theorem to show that the maps
    % \begin{align*}
    %     (x, z_0, p, z_1, z) &\mapsto \bA (\overline{\Xi}_{z_0, p, z_1} \circ s_z) (x),\\
    %     (x, z) &\mapsto \bA (V \circ s_z)(x)
    % \end{align*}
    % are continuous. As such, $\bA$ is compatible.
    % {\color{red} Compatibility needs a serious upgrade. Not at all trivial..}
    We start by establishing the continuity of $(x,z) \mapsto \bA (V \circ s_z)(x)$. Let $(x_n,z_n)$ converge to $(x,z)$. We aim to apply Lemma \ref{lemma:converging_integrals} with $\cX = \bR^q \setminus \{0\}$, $\nu_n = \mu_{x_n}$, and  
    \begin{align*}
        \phi_n(\bfz) & := V \circ s_{z_n}(x_n + \bfz) - V \circ s_{z_n}(x_n) - \rchi_{B_1(0)} \ip{\bfz}{\nabla (V \circ s_{z_n})(x_n)}, \\
        \phi_\infty(\bfz) & := V \circ s_{z}(x + \bfz) - V \circ s_{z}(x) - \rchi_{B_1(0)} \ip{\bfz}{\nabla (V \circ s_z)(x)}.
    \end{align*}
    As $\phi_n$ is continuous, it remains to show that $\sup_{n \in \bN} \sup_{\bfz \neq 0} \frac{|\phi_n (\bfz)|}{W(\bfz)} < \infty$.
    By Lemma \ref{lemma:boundLevy}, we can estimate
    \begin{align*}
        |\phi_n(\bfz)| & \leq \rchi_{B_1(0)} \frac{1}{2} |\bfz|^2 + (1 - \rchi_{B_1(0)}) \max \left\{ - \log \left( 1 + \frac{1}{2}\left| x_n-z_n \right|^2\right), \log \left( 1+|\bfz|^2 \right) \right\}.
    \end{align*}
    Since $(x_n, z_n)$ is convergent, hence bounded, we obtain the desired estimate. Continuity of $(x,z) \mapsto \bA (V \circ s_z)(x)$ now follows by Lemma  \ref{lemma:converging_integrals}.

    Using the particular form of $\Xi_{z_0,p,z_1}$, cf.\ Definition \ref{definition:levy_penalizations}, one readily verifies that the map $(x,z_0,p,z_1,z)  \mapsto \bA \left( \Xi_{z_0,p,z_1} \circ s_z\right)(x)$ is continuous with an analogous argumentation.
\end{proof}

\begin{proof}[Proof of Lemma \ref{lemma:boundLevy}]

%We start with the proof of \ref{item:boundLevy_large}. 
Let $y = x-z$, then we can write
\begin{multline*}
 V \circ s_z(x + \bfz) - V \circ s_z(x) \\
    = \log\left(1+ \frac{1}{2} (y + \bfz)^2 \right) - \log\left(1+ \frac{1}{2} |y|^2 \right)  = \log\left(1 +  \frac{\frac{1}{2} |\bfz|^2  + \ip{y}{\bfz}}{1+ \frac{1}{2} |y|^2 } \right).
\end{multline*}
Applying Young's inequality to $\ip{y}{\bfz}$ leads to the upper bound
\begin{equation*}
     V \circ s_z(x + \bfz) - V \circ s_z(x) \leq \log\left(1 +  \frac{|\bfz|^2  +  \frac{1}{2}|y|^2}{1+ \frac{1}{2} |y|^2 } \right)  =  \log\left(2 +  \frac{ |\bfz|^2 - 1}{1+ \frac{1}{2} |y|^2 } \right) \leq \log \left(1 + |\bfz|^2 \right).
\end{equation*}
Using that the first term is positive, we obtain the lower bound
\begin{equation*}
    V \circ s_z(x + \bfz) - V \circ s_z(x) \geq \log\left(1 -  \frac{\frac{1}{2}|y|^2}{1+ \frac{1}{2} |y|^2 } \right) = - \log \left(1 + \frac{1}{2}|y|^2 \right).
\end{equation*}
This establishes \ref{item:boundLevy_large}. For the proof of \ref{item:boundLevy_small}, we apply Taylor's Theorem to obtain
\begin{align*}
    \left| V \circ s_z(x + \bfz) - V \circ s_z(x) - \ip{\nabla (V \circ s_z)(x)}{\bfz} \right| & \leq \frac{1}{2} |\bfz|^2 \sup_{\bfz \in B_1(0)} \sup_{i,j} \left|\nabla^2_{i,j} V(y + \bfz) \right| \\
    & \leq \frac{1}{2} |\bfz|^2,
\end{align*}
which follows by a direct inspection of 
\begin{equation*}
    \nabla^2_{i,j} V(x) = \frac{2\delta_{i,j}\left(1 + \frac{1}{2} |x|^2\right) - 2 x_i x_j}{(1 + \frac{1}{2} |x|^2)^2}.
\end{equation*}

\end{proof}

\begin{proof}[Proof of Lemma \ref{lemma:control_difference_dsquared}]
    Evaluating the shift maps, calculating the gradient of the squared Euclidean distance, and expanding the squares leads to
    \begin{align*}
        & \begin{multlined}
            \frac{1}{2} d^2_{x-y,x'-y'}(x + \bfz_1, x' + \bfz_2) - \frac{1}{2} d^2_{x-y,x'-y'}(x, x') \\ - \widehat{\rchi}(\bfz_1,\bfz_2) \ip{
            \begin{pmatrix}
                \bfz_1 \\ \bfz_2
            \end{pmatrix}
            }{\nabla \left(\frac{1}{2} d^2_{x-y,x'-y'}\right) (x, x')}
            \end{multlined} \\
        & = \frac{1}{2} d^2(y + \bfz_1, y' + \bfz_2) - \frac{1}{2} d^2(y, y') -  \widehat{\rchi}(\bfz_1,\bfz_2) \ip{y-y'}{\bfz_1 - \bfz_2}\\
        & = \frac{1}{2}d^2(\bfz_1,\bfz_2) + \ip{y-y'}{\bfz_1 - \bfz_2} - \widehat{\rchi}(\bfz_1,\bfz_2) \ip{y-y'}{\bfz_1 - \bfz_2}\\
        & \leq \left(1 - \frac{1}{2}\widehat{\rchi}(\bfz_1,\bfz_2) \right) d^2(\bfz_1,\bfz_2) + (1 - \widehat{\rchi}(\bfz_1,\bfz_2)) \frac{1}{2} d^2(y,y'),
    \end{align*}
    where in the second equality we use properties of the Euclidean distance $d$ and the final line is due to Young's inequality.
\end{proof}

\section{Construction of test functions} \label{section:testfunctions_construction}

In classical proofs of comparison principles, the approach to estimate $\sup u-v$ for a subsolution $u$ and supersolution $v$ is variable doubling or quadruplication, cf. \cite[Theorem 3.1]{BaCD97} or \cite[introduction of Section 3]{CIL92}: For $\alpha > 1$
\begin{equation} \label{eqn:duplicationvariables_1}
    \sup_{x \in E} u(x) - v(x) \leq \sup_{x,x' \in E} u(x) - v(x') - \frac{\alpha}{2} d^2(x,x').
\end{equation}
Letting $\alpha \rightarrow \infty$, forces optimizing points, if they exist, of the right-hand side together. In addition, by varying either of the two components, one obtains basic test functions in terms of $\frac{\alpha}{2}d^2$ for the use in the definition of the sub- and supersolution properties of $u$ and $v$. 

\smallskip

To ensure that optimizers in \eqref{eqn:duplicationvariables_1} exist, we will consider instead, for small $\varepsilon > 0$, the following problem that includes the containment function $V$ and upper bounds $\sup u-v$ up to a term of order $\varepsilon$:
\begin{multline} \label{eqn:duplicationvariables}
    \sup_{x \in E} \frac{1}{1-\varepsilon} u(x) - \frac{1}{1+\varepsilon} v(x) \\
    \leq \sup_{x,x' \in E} \frac{1}{1-\varepsilon} u(x) - \frac{1}{1+\varepsilon} v(x') - \frac{\alpha}{2} d^2(x,x') - \frac{\varepsilon}{1-\varepsilon}V(x) - \frac{\varepsilon}{1+\varepsilon}V(x').
\end{multline}
The particular form of the factors $1-\varepsilon$ and $1+\varepsilon$ is motivated by convexity based arguments, which will show up in the proofs of Proposition \ref{proposition:basic_comparison_using_Hestimate} and Theorem \ref{th:comparison_HJI} below.

\smallskip

The procedure in \eqref{eqn:duplicationvariables} would be sufficient for a standard, first-order Hamilton--Jacobi equation. The test functions produced by this procedure, however, will not be sufficient to treat second-order or integral operators. This problem was considered in \cite{CIL92} and \cite{BaIm08}. We will follow their approach by considering a quadruplication of variables, which we also phrase in terms of sup- and inf-convolutions. We then perform a Jensen-type perturbation.

As we aim to unify proofs for both integral and differential operators, we revisit the full proof and state our result in terms of test functions.

\smallskip

In Propositions \ref{proposition:optimizing_point_construction} and \ref{proposition:test_function_construction} below, which can be considered to be an extended two-variable variant of the Crandall--Ishii construction \cite[Theorem 3.2]{CIL92}, we start out by considering the optimization \eqref{eqn:duplicationvariables} in terms of the sup- and inf-convolution of $u$ and $v$, respectively, effectively leading to a quadruplication problem, see \eqref{eqn:optimizing_points_Lambda} below.

We then perform the Jensen perturbation, see \eqref{eqn:proposition:optimizing_point_construction:optimizing_points}. The rest of the proposition deals with various properties of the optimizers in relation to $u$ and $v$.

\smallskip

In Proposition \ref{proposition:test_function_construction}, we carry out an additional layer of smoothing operations to obtain $C^\infty$-test functions. Consequently, we can move away from the notion of solutions in terms of sub- and superjets, which is of paramount importance to effectively treat diffusive and jump-type processes in a common framework.

For readability, we express suprema and infima using $\ssup{\cdot}$ and $\iinf{\cdot}$, respectively, as defined in Section \ref{sec:framework_notations}.

\begin{proposition}[Construction of optimizers]\label{proposition:optimizing_point_construction}
    Let $u$ be bounded and upper semi-continuous, $v$ be bounded and lower semi-continuous, $V$ be a containment function as in Definition \ref{definition:perturbation_containment}, and $\{\zeta_{z,p}\}_{z \in E, p \in \bR^q} \subset C(E)$ and $\{\xi_{z}\}_{z \in E} \subset C^1(E)$ be collections of functions as in Definition \ref{definition:perturbation_first_second_order}. Fix $\varepsilon \in (0,1)$ and $\varphi \in (0,1]$. 

    \smallskip
        
    Then, there exist compact sets $K_{\varepsilon,0} \subseteq K_\varepsilon \subseteq E$ and, for any $\alpha > 1$, three pairs of variables $(y_{\alpha,0},y_{\alpha,0}')$, $(y_{\alpha},y_{\alpha}')$, $(x_{\alpha},x_{\alpha}')$ in $E^2$ and $p_{\alpha},p_{\alpha}' \in B_{1/\alpha}(0)$ such that the following four sets of properties hold.

    \begin{description}
        \item[Properties of $y_{\alpha,0},y_{\alpha,0}'$]
    \end{description} 
    
    The variables $y_{\alpha,0},y_{\alpha,0}'$ optimize $\ssup{\Lambda_\alpha}$, where
    \begin{multline} \label{eqn:optimizing_points_Lambda}
        \Lambda_\alpha(y,y') := \frac{1}{1-\varepsilon} P^{\alpha}[u](y) - \frac{1}{1+\varepsilon} P_{\alpha}[v](y') - \frac{\alpha}{2}d^2(y,y')  \\
        - \frac{\varepsilon}{1-\varepsilon} (1-\varphi) V(y) - \frac{\varepsilon}{1+\varepsilon} (1-\varphi) V(y')
    \end{multline}
    and satisfy the following property
    \begin{enumerate}[(a)]
        \item \label{item:proposition:optimizing_point_construction_compact0} $y_{\alpha,0},y_{\alpha,0}' \in K_{\varepsilon,0}$.
    \end{enumerate}

    \begin{description}
        \item[Properties of $y_{\alpha},y_{\alpha}'$ and $p_{\alpha},p_{\alpha}'$]
    \end{description} 

    The pair $y_{\alpha},y_{\alpha}'$ optimizes
    \begin{equation} \label{eqn:proposition:optimizing_point_construction:optimizing_points}
        \ssup{ \Lambda_\alpha  - \frac{\varepsilon}{1-\varepsilon} \varphi \Xi^0_1 - \frac{\varepsilon}{1+\varepsilon} \varphi \Xi^0_2  }
    \end{equation}
    and uniquely optimizes
    \begin{equation} \label{eqn:proposition:optimizing_point_construction:optimizing_points_unique}
        \ssup{\Lambda_\alpha  - \frac{\varepsilon}{1-\varepsilon} \varphi \Xi_1 - \frac{\varepsilon}{1+\varepsilon} \varphi \Xi_2  }
    \end{equation}
    where $\Lambda_\alpha$ is as in \eqref{eqn:optimizing_points_Lambda} and 
    \begin{align*}
        \Xi_1^0(y) &\coloneqq \Xi^0_{y_{\alpha,0}, p_{\alpha}} (y), &
        \Xi_2^0(y') &\coloneqq \Xi^0_{y_{\alpha,0}', p_{\alpha}'} (y'),\\
        \Xi_1(y) &\coloneqq \Xi_{y_{\alpha,0}, p_{\alpha}, y_{\alpha}}(y), &  \Xi_2(y') &\coloneqq \Xi_{y_{\alpha,0}', p_{\alpha}', y_{\alpha}'}(y')
    \end{align*}
    as in Definition \ref{definition:perturbation_first_second_order}.    Moreover, the optimizers $y_{\alpha},y_{\alpha}'$ of \eqref{eqn:proposition:optimizing_point_construction:optimizing_points} and \eqref{eqn:proposition:optimizing_point_construction:optimizing_points_unique} satisfy
    \begin{enumerate}[(a),resume]
        \item \label{item:proposition:optimizing_point_Jensencontrol} We have
        \begin{equation*}
            d(y_{\alpha},y_{\alpha,0})  \leq \frac{1}{\alpha}, \qquad
            d(y_{\alpha}',y_{\alpha,0}')  \leq \frac{1}{\alpha}.
        \end{equation*}
        \item \label{item:proposition:optimizing_point_construction_twice_diff} $P^\alpha[u]$ and $P_\alpha[v]$ are twice differentiable in $y_{\alpha}$ and $y_{\alpha}'$, respectively.
    \end{enumerate}

    \begin{description}
        \item[Properties of $x_{\alpha},x_{\alpha}'$]
    \end{description} 
        The variables $x_{\alpha}, x_{\alpha}'$ optimize
    \begin{equation}\label{eqn:construction_optimizing_x}
        \begin{aligned}
            P^\alpha[u](y_{\alpha}) &  = u(x_{\alpha}) - \frac{\alpha}{2} d^2(x_{\alpha}, y_{\alpha}), \\
            P_\alpha[v](y_{\alpha}') & = v(x_{\alpha}') + \frac{\alpha}{2} d^2(x_{\alpha}',y_{\alpha}'),
        \end{aligned}
    \end{equation}
    and satisfy
    \begin{enumerate}[(a),resume]
        \item \label{item:proposition:optimizing_point_construction:Ralpha_u_Ralpha_v_optimizers} $x_{\alpha}$ and $x_{\alpha}'$ are the unique optimizers in the definition of $P^\alpha[u](y_{\alpha})$ and $P_\alpha[v](y_{\alpha}')$, respectively.
        \item \label{item:item:proposition:optimizing_point_construction:shift_optimal}
        We have that
        \begin{align*}
            u(x_{\alpha}) - P^\alpha[u] \circ s_{x_{\alpha}-y_{\alpha}}(x_{\alpha}) & = \ssup{ u - P^\alpha[u] \circ s_{x_{\alpha}-y_{\alpha}} }, \\
            v(x_{\alpha}') - P_\alpha[v] \circ s_{x_{\alpha}'-y_{\alpha}'}(x_{\alpha}') & = \iinf{ v - P_\alpha[v] \circ s_{x_{\alpha}'-y_{\alpha}'} }.
        \end{align*}
    \end{enumerate}
    
    \begin{description}
        \item[Behaviour as $\alpha \rightarrow \infty$]
    \end{description} 

    \begin{enumerate}[(a),resume]
        \item \label{item:proposition:optimizing_point_construction_0optimizers_convergence} We have $\lim_{\alpha \rightarrow \infty} \alpha d^2(y_{\alpha,0},y_{\alpha,0}') = 0$.
        \item \label{item:proposition:optimizing_point_construction_2optimizers_convergence} 
        We have 
        \begin{equation*}
            \lim_{\alpha \rightarrow \infty} \alpha  \left( d\left(x_{\alpha},y_{\alpha} \right) + d\left(y_{\alpha}, y_{\alpha}'\right) + d\left(y_{\alpha}',x_{\alpha}' \right) \right)^2  = 0.
        \end{equation*}
        \item \label{item:proposition:optimizing_point_construction_compact2} $x_{\alpha}, y_{\alpha}, y_{\alpha}', x_{\alpha}' \in K_{\varepsilon}$.
        \end{enumerate}
        In addition, the following estimate on $u-v$ holds: For any compact set $K \subseteq E$, there is a compact set $\widehat{K} = \widehat{K}(K,\varepsilon, u,v)$ given by 
        \begin{equation} \label{eqn:definition_hatK}
            \widehat{K} := \left\{z \in E \, \middle| \,  V(z) \leq \frac{\vn{u}+\vn{v}}{\varepsilon} + \ssup{V}_K \right\},
        \end{equation}
        such that
        \begin{enumerate}[(a),resume]
        \item \label{item:proposition:optimizing_point_construction_estimate_u-v}
        
        For any compact set $K \subseteq E$,
        \begin{equation*}
             \ssup{u-v}_K \leq \frac{1}{1-\eps}u(x_{\alpha}) - \frac{1}{1+\eps}v(x_{\alpha}) + \varepsilon\left(c_{\eps,\varphi} + o(1) \right),
        \end{equation*}
        where
        \begin{equation} \label{eqn:Cvarphi_construction}
            c_{\eps,\varphi}:= \frac{2}{1-\varepsilon^2} (1-\varphi) \ssup{V}_K - \iinf{\frac{1}{1-\varepsilon}u - \frac{1}{1+\varepsilon}v }_K,
        \end{equation}
         and $o(1)$ is in terms of $\alpha \rightarrow \infty$ for fixed $\varepsilon$ and $\varphi$.
        \item \label{item:proposition:optimizing_limits}
        Any limit point of the sequence $(x_{\alpha}, y_{\alpha},  y_{\alpha,0}, y_{\alpha,0}', y_{\alpha}', x_{\alpha}')$ as $\alpha \rightarrow \infty$ is of the form $(z,z,z,z,z,z)$ with $z \in \widehat{K}$. 
    \end{enumerate}    
\end{proposition}

Figure \ref{fig:optimizers} visualizes the relation between the different optimizing points.

\begin{figure}[h]
    \centering
    \includegraphics[width=0.6\textwidth]{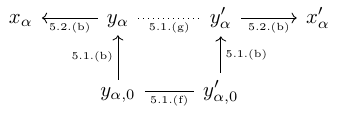}
    \caption{Relation between the optimizing points with a note which parts of the propositions give us distance control.}
    \label{fig:optimizers}
\end{figure}

The proof of Proposition \ref{proposition:optimizing_point_construction} uses various properties of sup- and inf-convolutions, which we gather in the next lemma. Its proof is relegated to Appendix \ref{appendix:proof_supconvolution_properties}.

\begin{lemma} \label{lemma:properties_supinf_convolution}
    Let $u: E \rightarrow \bR$ be bounded and upper semi-continuous and $v: E \rightarrow \bR$ be bounded and lower semi-continuous. For $\alpha > 1$, set
    \begin{align}
        P^\alpha[u](y) &\coloneqq \sup_{x \in E} \left\{u(x) - \frac{\alpha}{2} d^2(x,y) \right\} = \ssup{u - \frac{\alpha}{2} d^2 (\cdot, y)}, 
        \label{eqn:properties_supinf_convolution:defRalpha_u} \\
        P_\alpha[v](y) &\coloneqq \inf_{x \in E} \left\{v(x) + \frac{\alpha}{2}d^2(x,y) \right\} = \iinf{u + \frac{\alpha}{2} d^2 (\cdot, y)}. \label{eqn:properties_supinf_convolution:defRalpha_v}
    \end{align}
    Then,
    \begin{enumerate}[(a)] 
        \item \label{item:properties_supinf_convolution:bounded} we have $\vn{P^\alpha[u]} \leq \vn{u}$ and $\vn{P_\alpha[v]} \leq \vn{v}$.
        \item \label{item:properties_supinf_convolution:convergence_optimizers} for any $x,y \in E$ such that
        \begin{equation*}
            P^\alpha[u](y) = u(x) - \frac{\alpha}{2}d^2(x,y),
        \end{equation*}
        we have $\frac{\alpha}{2}d^2(x,y) \leq u(x) - u(y)$. Similarly, for any $x,y \in E$ with
        \begin{equation*}
            P_\alpha[v](y) = v(x) + \frac{\alpha}{2}d^2(x,y),
        \end{equation*} we have $\frac{\alpha}{2}d^2(x,y) \leq v(y) - v(x)$.
        \item \label{item:properties_supinf_convolution:decreasing}   $P^{\alpha}[u]$ and $- P_{\alpha}[v]$ are decreasing in $\alpha$.
        \item \label{item:properties_supinf_convolution:semi_convex} $P^\alpha[u]$ and $- P_\alpha[v]$ are semi-convex with semi-convexity constant $\alpha$. As a consequence, both are locally Lipschitz continuous.
        \item \label{item:properties_supinf_convolution:optimizers_differentiability} if $P^\alpha[u]$ is differentiable at $y_0$, then there exists a unique optimizer $x_0$ in \eqref{eqn:properties_supinf_convolution:defRalpha_u} such that
        \begin{equation*}
            P^\alpha[u](y_0) = u(x_0) - \frac{\alpha}{2}d^2(x_0,y_0)
        \end{equation*}
        and $D P^\alpha[u](y_0) = \alpha(x_0 - y_0)$.
        Similarly, if $P_\alpha[v]$ is differentiable at $y_0$, then there is a unique optimizer $x_0$ in \eqref{eqn:properties_supinf_convolution:defRalpha_v} such that
        \begin{equation*}
            P_\alpha[v](y_0) = v(x_0) + \frac{\alpha}{2}d^2(x_0,y_0)
        \end{equation*}
        and $D P_\alpha[v](y_0) = -\alpha(x_0 - y_0)$.
    \end{enumerate}
\end{lemma}

\begin{proof}[Proof of Proposition \ref{proposition:optimizing_point_construction}]
    {\bfseries Proof of \ref{item:proposition:optimizing_point_construction_compact0}:}
    As $u$ and $v$ are bounded, by Lemma \ref{lemma:properties_supinf_convolution} \ref{item:properties_supinf_convolution:bounded}, the same holds for $\vn{P^\alpha[u]}$ and $\vn{P_\alpha[v]}$. Using that $V$ has compact sublevelsets, cf. Definition \ref{definition:perturbation_containment}, the existence of optimizers $(y_{\alpha,0},y_{\alpha,0}')$ for $\ssup{\Lambda_\alpha}$ follows.
    
    The definition of $\Lambda_\alpha$ and the convolutions $P^\alpha[u]$ and $P_\alpha[v]$ imply that
    \begin{equation} \label{eqn:upperbondV_via_Lambda_alpha}
        \frac{\varepsilon}{1-\varepsilon} (1-\varphi) V(y_{\alpha,0}) + \frac{\varepsilon}{1+\varepsilon} (1-\varphi) V(y_{\alpha,0}') \leq \frac{1}{1-\varepsilon} \ssup{u} - \frac{1}{1+\varepsilon}\iinf{v} - \ssup{\Lambda_\alpha}.
    \end{equation}
    
    Comparing the optimizers for $\Lambda_\alpha (y, y')$ to, e.g., the suboptimial choice $(y, y') = (\hat{y},\hat{y})$ satisfying $V(\hat{y}) = 0$, we find
    \begin{equation*}
        \frac{\varepsilon}{1-\varepsilon} (1-\varphi) V(y_{\alpha,0}) + \frac{\varepsilon}{1+\varepsilon} (1-\varphi) V(y_{\alpha,0}') \leq \frac{2}{1-\varepsilon} \vn{u} + \frac{2}{1+\varepsilon} \vn{v}.
    \end{equation*}
    From this estimate, we deduce that $(y_{\alpha,0},y_{\alpha,0}') \in K_{\varepsilon,0} \times K_{\varepsilon,0}$ with
    \begin{equation*}
        K_{\varepsilon,0} \coloneqq \left\{y \in E \, \middle| \, V(y) \leq \varepsilon^{-1} C_\varepsilon(\vn{u} + \vn{v}) \right\}
    \end{equation*}
    for some constant $C_\varepsilon > 0$ satisfying $\lim_{\varepsilon \downarrow 0} C_\varepsilon = \frac{2}{1-\varphi}$, establishing \ref{item:proposition:optimizing_point_construction_compact0}.   

    {\bfseries Proof of \ref{item:proposition:optimizing_point_Jensencontrol} and \ref{item:proposition:optimizing_point_construction_twice_diff}:}
    For the proof of these two statements, we first move from $\ssup{\Lambda_\alpha}$ to its perturbed version \eqref{eqn:proposition:optimizing_point_construction:optimizing_points}. To do so, we use Proposition \ref{proposition:Jensen_Alexandrov_cutoff}. Note, that the function $(y,y') \mapsto \Lambda_\alpha(y,y')$ of \eqref{eqn:optimizing_points_Lambda} over which we optimize in $\ssup{\Lambda_\alpha}$ is semi-convex with semi-convexity constant 
    \begin{equation*}
        \kappa = \left(\frac{2}{1-\varepsilon^2} + \frac{1}{2}\right) \alpha + \frac{2\varepsilon}{1-\varepsilon^2} (1-\varphi) \kappa_V > 1
    \end{equation*}
    for $\alpha > 1$. In addition, it is bounded from above and has optimizers $(y_{\alpha,0},y'_{\alpha,0})$. We can thus apply Proposition \ref{proposition:Jensen_Alexandrov_cutoff} with 
    \begin{equation} \label{eqn:optimizer_construct_1}
        \eta = \frac{1}{\alpha}, \qquad \varepsilon_1 = \frac{\varepsilon}{1-\varepsilon} \varphi, \qquad \varepsilon_2 = \frac{\varepsilon}{1+\varepsilon} \varphi.
    \end{equation}

    Consequently, it follows that there exist $p_{\alpha},p_{\alpha}' \in B_{1/\alpha}(0)$ such that  $y_{\alpha},y_{\alpha}'$ are optimizers of
    \begin{equation}\label{eqn:constructionOptimizersHatPhi}
        \ssup{\widehat{\Lambda}_\alpha} = \widehat{\Lambda}_\alpha(y_{\alpha},y_{\alpha}'),
    \end{equation}
    where
    \begin{equation} \label{eqn:constructionOptimizersHatPhi2}
    \widehat{\Lambda}_\alpha(y,y') := \Lambda_\alpha(y,y')  - \frac{\varepsilon}{1-\varepsilon} \varphi \Xi_1^0(y) - \frac{\varepsilon}{1+\varepsilon} \varphi \Xi_2^0(y')
    \end{equation}
    with $\Xi^0_1$ and $\Xi^0_2$ as defined above.
    This establishes \eqref{eqn:proposition:optimizing_point_construction:optimizing_points}. An additional penalization around $(y_{\alpha},y_{\alpha}')$ gives \eqref{eqn:proposition:optimizing_point_construction:optimizing_points_unique}. 
    A secondary outcome of Proposition \ref{proposition:Jensen_Alexandrov_cutoff} is that $\widehat{\Lambda}_\alpha$ is twice differentiable in the optimizing point $(y_{\alpha}, y_{\alpha}')$, establishing \ref{item:proposition:optimizing_point_construction_twice_diff}. Furthermore, the optimizers satisfy
    \begin{equation} \label{eqn:optimizer_construct_2}
    d(y_{\alpha},y_{\alpha,0}) < \eta, \qquad d(y_{\alpha}',y_{\alpha,0}') < \eta,
    \end{equation}
    which, together with \eqref{eqn:optimizer_construct_1}, yields
    \begin{equation}\label{eqn:optimizer_construct_3}
        \max \left\{ d(y_{\alpha},y_{\alpha,0}), d(y_{\alpha}',y_{\alpha,0})'\right\} \leq \frac{1}{\alpha},
    \end{equation}
    establishing \ref{item:proposition:optimizing_point_Jensencontrol}. 
    
    {\bfseries Proof of \ref{item:proposition:optimizing_point_construction:Ralpha_u_Ralpha_v_optimizers}:}
    This follows immediately from Lemma \ref{lemma:properties_supinf_convolution} \ref{item:properties_supinf_convolution:optimizers_differentiability}. 

    {\bfseries Proof of \ref{item:item:proposition:optimizing_point_construction:shift_optimal}:}
    We only establish
    \begin{equation*}
        u(x_{\alpha}) - P^\alpha[u] \circ s_{x_{\alpha}-y_{\alpha}}(x_{\alpha}) = \ssup{ u - P^\alpha[u] \circ s_{x_{\alpha}-y_{\alpha}}},
    \end{equation*}
    as the second equation follows similarly. Note that by definition of $P^\alpha[u]$, we have
    \begin{equation*}
        P^\alpha[u] \circ s_{x_{\alpha}-y_{\alpha}}(x) \geq u(x) - \frac{\alpha}{2}d^2\big(x, s_{x_{\alpha} - y_{\alpha}}(x)\big).
    \end{equation*}
    On the other hand, by \ref{item:proposition:optimizing_point_construction:Ralpha_u_Ralpha_v_optimizers}, we have
    \begin{equation*}
        P^\alpha[u] \circ s_{x_{\alpha}-y_{\alpha}}(x_{\alpha}) = P^\alpha[u](y_{\alpha}) = u(x_{\alpha}) - \frac{\alpha}{2}d^2\left(x_{\alpha}, y_{\alpha}\right).
    \end{equation*}
    Combining the two statements yields, for any $x \in E$, that
    \begin{align*}
        & u(x_{\alpha}) - P^\alpha[u] \circ s_{x_{\alpha}-y_{\alpha}}(x_{\alpha}) \\
        & \qquad = \frac{\alpha}{2} d^2\left(x_{\alpha}, y_{\alpha} \right) + P^\alpha[u] \circ s_{x_{\alpha}-y_{\alpha}}(x) - P^\alpha[u] \circ s_{x_{\alpha}-y_{\alpha}}(x) \\
        & \qquad \geq u(x) - P^\alpha[u] \circ s_{x_{\alpha}-y_{\alpha}}(x) + \frac{\alpha}{2}\left(d^2\left(x_{\alpha}, y_{\alpha} \right) - d^2\left(x, s_{x_{\alpha} - y_{\alpha}}(x)\right)\right) \\
        & \qquad =  u(x) - P^\alpha[u] \circ s_{x_{\alpha}-y_{\alpha}}(x) 
    \end{align*}
    as the shift map preserves distances. This establishes \ref{item:item:proposition:optimizing_point_construction:shift_optimal}.

    \smallskip

    For the proof of the final five properties, we consider the limit $\alpha \rightarrow \infty$.
    
    \smallskip
    
    {\bfseries Proof of \ref{item:proposition:optimizing_point_construction_0optimizers_convergence}:}
    Consider $ \ssup{\Lambda_\alpha}$:
    \begin{align*}
        \ssup{\Lambda_\alpha} & = \frac{1}{1-\varepsilon} P^\alpha[u](y_{\alpha,0}) - \frac{1}{1+\varepsilon} P_\alpha[v](y'_{\alpha,0}) - \frac{\alpha}{2}d^2(y_{\alpha,0},y'_{\alpha,0}) \\
        & \qquad - \frac{\varepsilon}{1-\varepsilon} (1-\varphi) V(y_{\alpha,0}) - \frac{\varepsilon}{1+\varepsilon} (1-\varphi) V(y_{\alpha,0}').
    \end{align*}
    Note, that $\ssup{\Lambda_\alpha}$ is decreasing in $\alpha$, since $-\frac{\alpha}{2} d^2(y_{\alpha,0},y'_{\alpha,0})$, $P^{\alpha}[u]$, and $- P_\alpha [v]$ are decreasing in $\alpha$ by Lemma \ref{lemma:properties_supinf_convolution} \ref{item:properties_supinf_convolution:decreasing}. 
    Note in addition that, by evaluating $\Lambda_\alpha$ in the particular choice $(y,y') = (\hat{y}, \hat{y})$ as above, we have, by Lemma \ref{lemma:properties_supinf_convolution} \ref{item:properties_supinf_convolution:bounded}, that
    \begin{equation*}
        \ssup{\Lambda_\alpha}  \geq \frac{1}{1-\varepsilon} P^\alpha[u](\hat{y}) - \frac{1}{1+\varepsilon} P_\alpha[v](\hat{y}) \geq \frac{1}{1-\varepsilon}\vn{u} - \frac{1}{1+\varepsilon} \vn{v},
    \end{equation*}
    which is lower bounded uniformly in $\alpha$. It follows that the limit $\lim_{\alpha \to \infty} \sup \Lambda_\alpha $ exists. For any $\alpha > 1$, we find
    \begin{align}
         \ssup{\Lambda_{\alpha/2}}  &\geq \frac{1}{1-\varepsilon} P^{\alpha/2}[u](y_{\alpha,0}) - \frac{1}{1+\varepsilon} P_{\alpha/2}[v](y'_{\alpha,0}) - \frac{\alpha}{4}d^2(y_{\alpha,0},y'_{\alpha,0}) \nonumber \\
        &\qquad - \frac{\varepsilon}{1-\varepsilon}(1-\varphi)V(y_{\alpha,0}) - \frac{\varepsilon}{1+\varepsilon} (1-\varphi) V(y_{\alpha,0}') \nonumber \\
        & \geq  \frac{1}{1-\varepsilon} P^\alpha[u](y_{\alpha,0}) - \frac{1}{1+\varepsilon} P_\alpha[v](y'_{\alpha,0}) - \frac{\alpha}{4}d^2(y_{\alpha,0},y'_{\alpha,0})\nonumber \\
        & \qquad - \frac{\varepsilon}{1-\varepsilon}(1-\varphi)V(y_{\alpha,0}) - \frac{\varepsilon}{1+\varepsilon} (1-\varphi) V(y_{\alpha,0}') \\
        &\geq \ssup{\Lambda_\alpha} + \frac{\alpha}{4}d^{2}(y_{\alpha,0},y_{\alpha,0}'),
        \label{eqn:Malpha_bound}
    \end{align}
    which implies that $\lim_{\alpha \to \infty}\alpha d^2(y_{\alpha,0},y_{\alpha,0}')=0$, as $\ssup{\Lambda_\alpha}$ and $\ssup{\Lambda_{\alpha/2}}$ converge to the same limit, establishing \ref{item:proposition:optimizing_point_construction_0optimizers_convergence}.
    
    {\bfseries Proof of \ref{item:proposition:optimizing_point_construction_2optimizers_convergence}:}
    We follow the same approach as in \eqref{eqn:Malpha_bound} but now expanding $P^\alpha[u](y_{\alpha})$ and $P_\alpha[v](y_{\alpha}')$ to obtain an optimization problem in terms of four variables.
    \begin{equation}
        \begin{aligned}
            \ssup{\Lambda_{\alpha/2}} &\geq \frac{1}{1-\varepsilon} P^{\alpha/2}[u](y_{\alpha}) - \frac{1}{1+\varepsilon} P_{\alpha/2}[v](y'_{\alpha}) - \frac{\alpha}{4}d^2(y_{\alpha},y'_{\alpha}) \\
            & \geq \frac{1}{1-\varepsilon} u(x_{\alpha}) - \frac{1}{1+\eps} v(x_{\alpha}') - \frac{\varepsilon}{1-\varepsilon}(1-\varphi)V(y_{\alpha}) - \frac{\varepsilon}{1+\varepsilon}(1-\varphi) V(y_{\alpha}')\\
            & \qquad - \frac{\alpha}{4}\left(\frac{1}{1-\eps}d^2(x_{\alpha},y_{\alpha})  + d^2(y_{\alpha},y'_{\alpha}) + \frac{1}{1+\eps} d^2(y_{\alpha}',x'_{\alpha}) \right) \\
            &  = \ssup{\widehat{\Lambda}_{\alpha}} + \frac{\alpha}{4} \left(\frac{1}{1-\varepsilon}d^2(x_{\alpha},y_{\alpha}) + d^2(y_{\alpha},y_{\alpha}') + \frac{1}{1+\varepsilon} d^2(y_{\alpha}',x_{\alpha}') \right) \\
            & \qquad + \frac{\eps}{1-\eps} \varphi \Xi_1^0(y_{\alpha}) + \frac{\eps}{1+\eps} \varphi \Xi_2^0(y'_{\alpha})
        \end{aligned}
    \end{equation}
    by \eqref{eqn:constructionOptimizersHatPhi2}. It follows that 
    \begin{multline}
        \frac{\alpha}{4} \left(\frac{1}{1-\varepsilon}d^2(x_{\alpha},y_{\alpha}) + d^2(y_{\alpha},y_{\alpha}') + \frac{1}{1+\varepsilon} d^2(y_{\alpha}',x_{\alpha}') \right)\\ 
        \leq \ssup{\Lambda_{\alpha/2}} - \ssup{\widehat{\Lambda}_{\alpha}} - \frac{\eps}{1-\eps} \varphi \Xi_1^0(y_{\alpha}) - \frac{\eps}{1+\eps} \varphi \Xi_2^0(y'_{\alpha}).
    \end{multline}

    By \ref{item:proposition:optimizing_point_construction_0optimizers_convergence}, we obtain
    \begin{equation} %\label{eqn:convergence_LambdaLambdahat_alpha}
        \lim_{\alpha \rightarrow \infty} \ssup{\Lambda_\alpha} = \lim_{\alpha \rightarrow \infty} \ssup{\widehat{\Lambda}_\alpha}
    \end{equation}
    and 
    \begin{equation} %\label{eqn:convergence_perturb_Xi_in_alpha_to_0}
        \lim_{\alpha \rightarrow \infty} \frac{\varepsilon}{1-\varepsilon} \varphi \Xi_{1}^0(y_{\alpha}) + \frac{\varepsilon}{1+\varepsilon} \varphi \Xi_{2}^0(y_{\alpha}') = 0.
    \end{equation}
    Consequently, we have that
    \begin{equation*}
        \lim_{\alpha \rightarrow \infty} \alpha \left(d^2(x_{\alpha},y_{\alpha}) + d^2(y_{\alpha},y_{\alpha}') + d^2(y_{\alpha}',x_{\alpha}') \right) = 0.
    \end{equation*}
    From this, \ref{item:proposition:optimizing_point_construction_2optimizers_convergence} follows using Young's inequality.

    {\bfseries Proof of \ref{item:proposition:optimizing_point_construction_compact2}:}
    \ref{item:proposition:optimizing_point_construction_compact0}, \ref{item:proposition:optimizing_point_Jensencontrol}, \ref{item:proposition:optimizing_point_construction_0optimizers_convergence}, and \ref{item:proposition:optimizing_point_construction_2optimizers_convergence} imply \ref{item:proposition:optimizing_point_construction_compact2} by considering a bounded blow-up $K_{\varepsilon}$ of $K_{\varepsilon,0}$.

    {\bfseries Proof of \ref{item:proposition:optimizing_point_construction_estimate_u-v}: }
    First, note that Corollary \ref{corollary:Jensen_optimizerbound} and the definition of $\eta$ in \eqref{eqn:optimizer_construct_1} yield
    \begin{equation}\label{eqn:Xi_optimizer_nonnegative}
    \begin{split}
    0 \leq - \frac{\varepsilon}{1-\varepsilon} \varphi \Xi_1^0(y_{\alpha})  - \frac{\varepsilon}{1+\varepsilon} \varphi \Xi_2^0(y_{\alpha}') \leq \varphi \frac{2\varepsilon}{1-\varepsilon^2} \frac{1}{\alpha}
    \end{split}
    \end{equation}
    and
    \begin{equation} \label{eqn:Jensen_control_optimizationproblem_inOptimizerConstruction}
        \ssup{\Lambda_\alpha} \leq \ssup{\widehat{\Lambda}_\alpha}  = \widehat{\Lambda}_\alpha(y_{\alpha},y_{\alpha}') 
        \leq \ssup{\Lambda_\alpha} + \varphi\frac{2\eps}{1-\eps^2} \frac{1}{\alpha}.
    \end{equation}
    Let $K \subseteq E$ be compact. We then obtain
    \begin{align}
         \ssup{u-v}_K &= \sup_{x \in K} u(x) - v(x) \notag \\
        & \leq \sup_{x \in K} \left\{ u(x) - v(x) - \frac{2\varepsilon}{1-\varepsilon^2} (1-\varphi) \left(V(x) - \ssup{V}_K\right) \right\} \notag \\
        & \leq \sup_{x \in K} \left\{ \frac{1}{1-\varepsilon} u(x) - \frac{1}{1+\varepsilon} v(x) - \frac{2\varepsilon}{1-\varepsilon^2} (1-\varphi) V(x)  \right\} \\ &\qquad + \frac{2\varepsilon}{1-\varepsilon^2} (1-\varphi) \ssup{V}_K - \varepsilon \iinf{\frac{1}{1-\varepsilon}u - \frac{1}{1+\varepsilon}v }_K  \notag \\
        & \leq \sup_{x \in E} \left\{ \frac{1}{1-\varepsilon} u(x) - \frac{1}{1+\varepsilon} v(x) - \frac{2\varepsilon}{1-\varepsilon^2} (1-\varphi) V(x)  \right\} \\
        &\qquad + \frac{2\varepsilon}{1-\varepsilon^2} (1-\varphi) \ssup{V}_K - \varepsilon \iinf{\frac{1}{1-\varepsilon}u - \frac{1}{1+\varepsilon}v }_K \notag \\
        & \leq \ssup{\Lambda_\alpha} + \frac{2\varepsilon}{1-\varepsilon^2} (1-\varphi) \ssup{V}_K - \varepsilon \iinf{\frac{1}{1-\varepsilon}u - \frac{1}{1+\varepsilon}v }_K. \label{eqn:upperbound_u-v_onK}
    \end{align}
    Combining this estimate with the first inequality of \eqref{eqn:Jensen_control_optimizationproblem_inOptimizerConstruction}, dropping non-positive terms, and then \eqref{eqn:Xi_optimizer_nonnegative}, leads to
    \begin{align*}
         \ssup{u-v}_K &  \leq \widehat{\Lambda}_\alpha(y_{\alpha},y_{\alpha}') \\
         & \leq \frac{1}{1-\varepsilon}u(x_{\alpha}) -  \frac{1}{1+\varepsilon}v(x_{\alpha}') \\
         & \qquad + \varepsilon\left( \varphi \frac{2}{1-\varepsilon^2} \frac{1}{\alpha}  + \frac{2}{1-\varepsilon^2} (1-\varphi) \ssup{V}_K - \iinf{\frac{1}{1-\varepsilon}u - \frac{1}{1+\varepsilon}v }_K \right),
    \end{align*}
    which proves \ref{item:proposition:optimizing_point_construction_estimate_u-v}.
    
    {\bfseries Proof of \ref{item:proposition:optimizing_limits}:}
    We start by proving that any limiting point of 
    \begin{equation*}(x_{\alpha},y_{\alpha},y_{\alpha,0},y_{\alpha,0}',y_{\alpha}',x_{\alpha}')
    \end{equation*}
    as $\alpha \rightarrow \infty$ is of the form $(z,z,z,z,z,z)$. We only prove $\lim_{\alpha \rightarrow \infty} d(x_{\alpha},y_{\alpha}) = 0$, as the other limit follow analogously. 
    
    By  \ref{item:proposition:optimizing_point_construction_compact2}, we find that, along subsequences, $(x_{\alpha},y_{\alpha}) \rightarrow (x_0,y_0)$. Assume by contradiction that $x_0 \neq y_0$. Then, since $\alpha d^2$ is increasing, we get that for all $\alpha_0 >1$,
    \begin{equation}
        \liminf_{\alpha \to \infty} \alpha d^2(x_{\alpha}, y_{\alpha}) \geq \alpha_0 d^2(x_0,y_0).
    \end{equation}
    We can conclude that $\alpha d^2(x_{\alpha},y_{\alpha}) \to \infty$, contradicting \ref{item:proposition:optimizing_point_construction_2optimizers_convergence}.
    
    We proceed to prove that any limiting point $z$ lies in $\widehat{K}$. Similar to \eqref{eqn:upperbondV_via_Lambda_alpha}, but now also using  \eqref{eqn:Xi_optimizer_nonnegative} and the first inequality of \eqref{eqn:Jensen_control_optimizationproblem_inOptimizerConstruction}, we find
    \begin{equation} \label{eqn:upperbondV_via_Lambda_alphahat}
        \begin{aligned}
            & \frac{\varepsilon}{1-\varepsilon} (1-\varphi) V(y_{\alpha}) + \frac{\varepsilon}{1+\varepsilon} (1-\varphi) V(y_{\alpha}') \\
            & \qquad \leq \frac{1}{1-\varepsilon} \ssup{u} - \frac{1}{1+\varepsilon}\iinf{v}  - \frac{\varepsilon}{1-\varepsilon} \varphi \Xi_{1}^0(y_{\alpha}) - \frac{\varepsilon}{1+\varepsilon} \varphi \Xi_{2}^0(y_{\alpha}') - \ssup{\widehat{\Lambda}_\alpha}, \\
            & \qquad \leq \frac{1}{1-\varepsilon} \ssup{u} - \frac{1}{1+\varepsilon}\iinf{v} + \varphi \frac{2\varepsilon}{1-\varepsilon^2} \frac{1}{\alpha} - \ssup{\Lambda_\alpha},
        \end{aligned}
    \end{equation}
    Combining this with the upper bound on $-\ssup{\Lambda_\alpha}$ obtained from \eqref{eqn:upperbound_u-v_onK} leads to
    \begin{align*}
        & \frac{\varepsilon}{1-\varepsilon} (1-\varphi) V(y_{\alpha}) + \frac{\varepsilon}{1+\varepsilon} (1-\varphi) V(y_{\alpha}') \\
        & \qquad \leq \frac{1}{1-\varepsilon} \ssup{u} + \frac{1}{1+\varepsilon} \iinf{v} + \varphi \frac{2\varepsilon}{1-\varepsilon^2} \frac{1}{\alpha}\\
        & \qquad \qquad - \ssup{u-v}_K + \frac{2\varepsilon}{1-\varepsilon^2} (1-\varphi) \ssup{V}_K - \varepsilon \iinf{\frac{1}{1-\varepsilon}u - \frac{1}{1+\varepsilon}v }_K.
    \end{align*}
    This, in turn, yields
    \begin{align*}
         & \frac{\varepsilon}{1-\varepsilon} (1-\varphi) \left(V(y_{\alpha}) - \ssup{V}_K \right) + \frac{\varepsilon}{1+\varepsilon} (1-\varphi) \left( V(y_{\alpha}') - \ssup{V}_K \right) \\
        & \qquad \leq \frac{1}{1-\varepsilon} \ssup{u} + \frac{1}{1+\varepsilon} \iinf{v} + \varphi\frac{2\varepsilon}{1-\varepsilon^2} \frac{1}{\alpha} \\
        & \qquad \qquad - \ssup{u-v}_K + \varepsilon \iinf{\frac{1}{1-\varepsilon}u - \frac{1}{1+\varepsilon}v }_K\\
        & \qquad \leq 2\left(\vn{u} +  \vn{v}\right) + \varphi \frac{\varepsilon}{1-\varepsilon^2} \frac{1}{\alpha}.
    \end{align*}
    The sequences $x_\alpha, y_{\alpha},y_{\alpha}', x_\alpha'$ have limit points $z \in K_\varepsilon$ as $\alpha\rightarrow \infty$ by \ref{item:proposition:optimizing_point_construction_2optimizers_convergence} and \ref{item:proposition:optimizing_point_construction_compact2}. In combination with \ref{item:proposition:optimizing_point_Jensencontrol}, we conclude that, for any such limiting point $z$,
    \begin{equation*}
        \frac{2\varepsilon}{1-\varepsilon^2} (1-\varphi) \left(V(z) - \ssup{V}_{K}\right) \leq 2\left(\vn{u} +  \vn{v}\right),
    \end{equation*}
    establishing \ref{item:proposition:optimizing_limits}.
\end{proof}

The next proposition builds upon Proposition \ref{proposition:optimizing_point_construction}  to build a suitable collection of test functions for the use in the proof of the comparison principle. The sup- and inf-convolution $R^\alpha[u]$ and $R_\alpha[v]$ are not guaranteed to be smooth. However, they are twice differentiable in the relevant optimizing points.

Using the difference between $\Xi_1^0$ and $\Xi_2^0$ on one hand and $\Xi_1$ and $\Xi_2$ on the other, we are able to squeeze in a globally $C^\infty$ function on the basis of Lemma \ref{lemma:smooth_test_function_construction}, that can be used to replace $P^\alpha[u]$ and $P_\alpha[v]$. As an effect, we will approximate 
\begin{align*}
    \widehat{f}_\dagger & \approx P^\alpha[u], & f_\dagger & \approx P^\alpha[u] \circ s_{x_{\alpha} - y_{\alpha}}, \\
    \widehat{f}_\ddagger & \approx P_\alpha[v], & f_\ddagger & \approx P_\alpha[v] \circ s_{x_{\alpha}' - y_{\alpha}'}, 
\end{align*}
which will be made rigorously in next proposition for fixed $\varepsilon$ and $\alpha$. 

\begin{proposition}[Test function construction] \label{proposition:test_function_construction}
    Consider the setting of Proposition \ref{proposition:optimizing_point_construction}. Fix $\varepsilon \in (0,1)$, $\varphi \in (0, 1]$, and $\alpha >1$. Then, there are functions $f_1,f_2, \widehat{f}_1,\widehat{f}_2 \in C_c^\infty(E)$ such that
    \begin{equation*}
        f_1 = \widehat{f}_1 \circ s_{x_{\alpha}-y_{\alpha}}, \qquad f_2 = \widehat{f}_2 \circ s_{x_{\alpha}'-y_{\alpha}'}
    \end{equation*}
    and
      \begin{align*}
        \widehat{f}_\dagger & := (1-\varepsilon)\widehat{f}_1 + \varepsilon (1-\varphi) V + \varepsilon \varphi \Xi_1,  & f_\dagger = \widehat{f}_\dagger \circ s_{x_{\alpha}-y_{\alpha}},  \\ 
        \widehat{f}_\ddagger & := (1+\varepsilon)\widehat{f}_2 - \varepsilon (1-\varphi) V - \varepsilon \varphi \Xi_2, & f_\ddagger  = \widehat{f}_\ddagger \circ s_{x_{\alpha}'-y_{\alpha}'},
    \end{align*}
    satisfying the following properties: 
    
    For $\widehat{f}_1,\widehat{f}_2$ and $f_1,f_2$, we have
    \begin{enumerate}[(a)]
        \item \label{item:test_function_construction:optimizing_f1f2} The pair $(y_{\alpha},y_{\alpha}')$ is the unique optimizing pair of
        \begin{equation*}
            \widehat{f}_1(y_{\alpha}) - \widehat{f}_2(y_{\alpha}') - \frac{\alpha}{2}d^2(y_{\alpha},y_{\alpha}') = \ssup{ \widehat{f}_1 -\widehat{f}_2 - \frac{\alpha}{2}d^2 }.
        \end{equation*}
        and the pair $(x_{\alpha},x_{\alpha}')$ is the unique optimizing pair of
        \begin{equation}
            f_{1}(x_{\alpha}) - f_{2}(x_{\alpha}') - \frac{\alpha}{2}d^2_{x_{\alpha}-y_{\alpha},\; x_{\alpha}'-y_{\alpha}'}(x_{\alpha},x_{\alpha}')
            = \ssup{f_{1} - f_{2} -\frac{\alpha}{2}d^2_{x_{\alpha}-y_{\alpha},\; x_{\alpha}'-y_{\alpha}'} }.
        \end{equation}   
    \end{enumerate}
    For $\widehat{f}_\dagger,\widehat{f}_\ddagger$, and $f_\dagger,f_\ddagger$ we have
    \begin{enumerate}[(a),resume]
        \item \label{item:test_function_construction:f1f2_squeeze_Ralpha_u_Ralpha_v} We have
        \begin{equation} \label{eqn:bounds_for_Ralpha}
            \begin{aligned}
                P^\alpha[u](y) & \leq \widehat{f}_{\dagger}(y), \\
                P_\alpha[v](y') & \geq \widehat{f}_\ddagger(y')
            \end{aligned}
        \end{equation}
        with equality in $y_{\alpha}$ and $y_{\alpha}'$, respectively.
       \item \label{item:test_function_construction:fdagger_f_ddagger_test_functions_uv} We have that $x_{\alpha}, x'_{\alpha}$ are the unique points such that
        \begin{align*}
            u(x_{\alpha}) - f_\dagger(x_{\alpha}) & = \ssup{ u - f_\dagger }, \\
            v(x_{\alpha}') - f_\ddagger(x_{\alpha}') & = \ssup{ v - f_\ddagger }.
        \end{align*}
        % The gradients and second moments align
        \item \label{item:test_function_construction:f1f2_equal_Ralpha_u_Ralpha_v_derivatives} We have
        \begin{equation}\label{eqn:bounds_for_Ralpha_derivatives}
            \begin{aligned} 
                D \widehat{f}_\dagger(y_{\alpha}) & = D f_\dagger (x_{\alpha}) =  \alpha(y_{\alpha}-x_{\alpha}), \\
                D^2 \widehat{f}_\dagger(y_{\alpha}) & = D^2 f_\dagger(x_{\alpha}), \\
                D \widehat{f}_\ddagger(y_{\alpha}') & = D f_\ddagger (x_{\alpha}')  = \alpha(x_{\alpha}'-y_{\alpha}'), \\
                D^2 \widehat{f}_\ddagger(y_{\alpha}') & = D^2f_\ddagger(x_{\alpha}').
            \end{aligned}
        \end{equation}
    \end{enumerate}
\end{proposition}

As noted before the previous proposition, we aim to construct $\widehat{f}_\dagger \approx P^\alpha[u]$, but start out by first constructing $\widehat{f}_1 \in C_c^\infty(E)$, which, by re-arrangement, satisfies
\begin{equation*}
    \widehat{f}_1 \approx \frac{1}{1-\varepsilon} P^\alpha[u](y) - \frac{\varepsilon}{1-\varepsilon} (1-\varphi) V(y)  - \frac{\varepsilon}{1-\varepsilon} \varphi \Xi_1(y)
\end{equation*}
and is constant outside of a compact set. As $V$ has compact sublevel sets and other terms on the right-hand side are bounded from above, it suffices to first perform a smooth approximation and cut off the result. For the cut-off procedure, we use functions $\Omega_M^+$ and $\Omega_M^-$.

\begin{definition}[Cut-off functions]\label{def:cutoff_function}
    Let $M >0$. We call a smooth increasing function $\Omega_M^+: \bR \rightarrow \bR$ a \emph{upper cut-off function at $M$}, if
    \begin{equation*}
        \Omega^+_M(r) = \begin{cases}
            r & \text{if } r \leq M, \\
            M+1 & \text{if } r \geq M+2.
        \end{cases}
    \end{equation*}
    We call $\Omega^-_M$ a \emph{lower cut-off function} at $M$ if $\Omega^-_M(r) = - \Omega^+_{-M}(-r)$.
\end{definition}

\begin{proof}[Proof of Proposition \ref{proposition:test_function_construction}]

    In this proof, we work in the context of Proposition \ref{proposition:optimizing_point_construction} and will, correspondingly, follow its notation.
    We show the construction procedure for the test function $f_1$ used in the subsolution case only, as $f_2$ is constructed analogously.
    Denote
    \begin{align}
        \Pi_1^0(y) & \coloneqq  \frac{1}{1-\varepsilon}P^\alpha[u](y)  - \frac{\varepsilon}{1-\varepsilon} (1-\varphi) V(y) - \frac{\varepsilon}{1-\varepsilon} \varphi \Xi_1^0(y), \\
        \Pi_{1}(y) & \coloneqq  \frac{1}{1-\varepsilon}P^\alpha[u](y)  - \frac{\varepsilon}{1-\varepsilon} (1-\varphi) V(y) - \frac{\varepsilon}{1-\varepsilon} \varphi \Xi_1(y).
    \end{align}
    Note that we have $\Pi_{1}(y_\alpha)=\Pi_1^0(y_\alpha)$ and $\Pi_{1}(y)<\Pi_1^0(y)$ for all $y \in E \setminus \{y_\alpha\}$. By Lemma \ref{lemma:smooth_test_function_construction}, we find a function $\mathfrak{f}_1 \in C^\infty(E)$ such that
    \begin{equation*}
        \Pi_{1}(y) < \mathfrak{f}_1(y) <\Pi_1^0(y), \qquad y \neq y_\alpha.
    \end{equation*}
    The function $\mathfrak{f}_2$ is constructed analogously. By construction of $\mathfrak{f}_1$, $\mathfrak{f}_2$ and \eqref{eqn:proposition:optimizing_point_construction:optimizing_points}, $(y_\alpha, y_\alpha')$ is the unique optimizer of $\ssup{\mathfrak{f}_1 - \mathfrak{f}_2 - \frac{\alpha}{2} d^2}$.

    As our test functions need to be constant outside a compact set, we need to cut them off in an appropriate manner. However, we need to preserve their properties in the optimizer $(y_\alpha, y'_\alpha)$.\ Taking these conditions into account, ensures that the cut-off procedure does not create new optimizers. 

    The above considerations lead to the cut-off procedure $\widehat{f}_1 \coloneqq \Omega^-_{M_1} \circ \mathfrak{f}_1$ and $\widehat{f}_2 \coloneqq \Omega^+_{M_2} \circ \mathfrak{f}_2$, with $\Omega^-_{M_1}$, $\Omega^+_{M_2}$ as in Definition \ref{def:cutoff_function} and the following choice of $M_1$ and $M_2$:

    Pick $m_1, m_2 \in \bR$ such that the level sets
    \begin{equation*}
        \left\{y \in E \, \middle| \, \mathfrak{f}_1(y) \geq m_1\right\}, \qquad \left\{y' \in E \, \middle| \, \mathfrak{f}_2(y') \leq m_1\right\}
    \end{equation*}
    are compact. Set
    \begin{align*}
        M_1  &\coloneqq \min \left\{m_1,\, \mathfrak{f}_1(y_\alpha)  - \left(\mathfrak{f}_2(y_\alpha') - \iinf{\mathfrak{f}_2} \right) - \frac{\alpha}{2}d^2(y_\alpha,y_\alpha')  \right\}, \\
        M_2  &\coloneqq \max \left\{m_2,\, \mathfrak{f}_2(y_\alpha')  + \left(\ssup{\mathfrak{f}_1} - \mathfrak{f}_1(y_\alpha) \right) + \frac{\alpha}{2}d^2(y_\alpha,y_\alpha')\right\}.
    \end{align*}

    Using $M_1$ and $M_2$ as defined above, we find that $(y_\alpha, y_\alpha')$ is the unique optimizer of $\ssup{ \widehat{f}_1 -\widehat{f}_2 - \frac{\alpha}{2}d^2 }$.
    To see this, denote $$A_1 \coloneqq \left\{y \in E \, \middle| \, \mathfrak{f}_1(y) \geq M_1 \right\}\quad \text{and} \quad A_2 \coloneqq \left\{y' \in E \, \middle| \, \mathfrak{f}_2(y') \leq M_2 \right\}.$$ 
    Thus, for $i \in \{1,2\}$, we find $\widehat{f}_i = \mathfrak{f}_i$ on $A_i$, whereas
    \begin{align*}
        \widehat{f}_1(y) & <  \mathfrak{f}_1(y_\alpha) - \left(\mathfrak{f}_2(y_\alpha') - \iinf{\mathfrak{f}_2}\right) - \frac{\alpha}{2}d^2(y_\alpha,y_\alpha'), \\
        \widehat{f}_2(y') & > \mathfrak{f}_2(y_\alpha')  + \left(\ssup{\mathfrak{f}_1} - \mathfrak{f}_1(y_\alpha) \right) + \frac{\alpha}{2}d^2(y_\alpha,y_\alpha')
    \end{align*}
    if $y \notin A_1$ or $y' \notin A_2$, respectively. 
    
    As $\widehat{f}_1 = \mathfrak{f}_1$ on $A_1$ and $\widehat{f}_2 = \mathfrak{f}_2$ on $A_2$, it suffices to show that
    \begin{equation*}
        \widehat{f}_1(y) - \widehat{f}_2(y') - \frac{\alpha}{2}d^2(y,y') < \mathfrak{f}_1(y_\alpha) - \mathfrak{f}_2(y_\alpha') - \frac{\alpha}{2}d^2(y_\alpha,y_\alpha')
    \end{equation*}
    if $y \in A_1^c$ or $y \in A_2^c$. For the proof of this bound, we consider the following three separate cases.
    
    \textit{Case $y \in A_1^c$ and $y' \in A_2$:}
    We have
    \begin{align*}
        \widehat{f}_1(y) - \widehat{f}_2(y') - \frac{\alpha}{2}d^2(y,y') & \leq \widehat{f}_1(y) - \widehat{f}_2(y') \\
        & < \mathfrak{f}_1(y_\alpha)  - \left(\mathfrak{f}_2(y_\alpha') - \iinf{\mathfrak{f}_2}\right) - \frac{\alpha}{2} d^2(y_\alpha,y_\alpha') -  \mathfrak{f}_2(y')  \\
        & = \mathfrak{f}_1(y_\alpha)  - \mathfrak{f}_2(y_\alpha') - \frac{\alpha}{2} d^2(y_\alpha,y_\alpha') - \left(\mathfrak{f}_2(y') - \iinf{\mathfrak{f}_2}\right)  \\
        & \leq \mathfrak{f}_1(y_\alpha)  - \mathfrak{f}_2(y_\alpha') - \frac{\alpha}{2} d^2(y_\alpha,y_\alpha').
    \end{align*}
    
    \textit{Case $y \in A_1$ and $y' \in A_2^c$:} Follows analogously to the case $y \in A_1^c$ and $y' \in A_2$.
    
    \textit{Case $y \in A_1^c$ and $y' \in A_2^c$:}
    We have
    \begin{align*}
         \widehat{f}_1(y) - \widehat{f}_2(y') & - \frac{\alpha}{2}d^2(y,y') \\
        & \leq \widehat{f}_1(y) - \widehat{f}_2(y') \\
        & < \mathfrak{f}_1(y_\alpha)  - \left(\mathfrak{f}_2(y_\alpha') - \iinf{\mathfrak{f}_2}\right) - \frac{\alpha}{2} d^2(y_\alpha,y_\alpha')   \\
        & \hspace{3cm} - \left( \mathfrak{f}_2(y_\alpha')  + \left(\ssup{\mathfrak{f}_1} - \mathfrak{f}_1(y_\alpha) \right) + \frac{\alpha}{2}d^2(y_\alpha,y_\alpha') \right) \\
        & \leq \mathfrak{f}_1(y_\alpha) - \mathfrak{f}_2(y_\alpha') - 2 \frac{\alpha}{2} d^2(y_\alpha,y_\alpha')  - \left(\mathfrak{f}_2(y_\alpha') - \iinf{\mathfrak{f}_2}\right) - \left(\ssup{\mathfrak{f}_1} - \mathfrak{f}_1(y_\alpha) \right)  \\
        & \leq \mathfrak{f}_1(y_\alpha) - \mathfrak{f}_2(y_\alpha') - \frac{\alpha}{2} d^2(y_\alpha,y_\alpha').
    \end{align*}
    We conclude that the pair $(y_\alpha, y_\alpha')$ is also the unique optimizer of $\ssup{\widehat{f}_1 - \widehat{f}_2 - \frac{\alpha}{2} d^2}$. Applying the shift maps $s_{x_\alpha - y_\alpha}$ and $s_{x'_\alpha - y'_\alpha}$, respectively, we find that $(x_\alpha, x'_\alpha)$ uniquely optimize $\ssup{\mathfrak{f}_1 \circ s_{x_\alpha - y_\alpha} - \mathfrak{f}_2 \circ s_{x'_\alpha - y'_\alpha} - \frac{\alpha}{2} d^2_{x_\alpha - y_\alpha, x'_\alpha - y'_\alpha}}$.
    Additionally, as $M_1 \geq m_1$ and $M_2 \leq m_2$, we have $\widehat{f}_1, \widehat{f}_2 \in C_c^\infty(E)$, establishing \ref{item:test_function_construction:optimizing_f1f2}.
    
    \smallskip

    We next prove \ref{item:test_function_construction:f1f2_squeeze_Ralpha_u_Ralpha_v}. As $r \leq \Omega_{M_1}^-(r)$,
    \begin{equation*}
          \frac{1}{1-\varepsilon}P^\alpha[u](y)  - \frac{\varepsilon}{1-\varepsilon} (1-\varphi) V(y) -\frac{\varepsilon}{1-\varepsilon} \varphi \Xi_1(y) = \Pi_1(y) \leq \Omega_{M_1}^- \circ \Pi_1(y) \leq \widehat{f}_1(y),
    \end{equation*}
    which, after rearrangement of terms, implies \ref{item:test_function_construction:f1f2_squeeze_Ralpha_u_Ralpha_v}.

\smallskip

We proceed with the  proof of \ref{item:test_function_construction:fdagger_f_ddagger_test_functions_uv}. By \ref{item:test_function_construction:f1f2_squeeze_Ralpha_u_Ralpha_v} and Proposition \ref{proposition:optimizing_point_construction} \ref{item:item:proposition:optimizing_point_construction:shift_optimal},
\begin{equation} \label{eqn:proof_test_function_inequalitysub1}
    \begin{aligned}
    f_\dagger(x) - f_\dagger(x_{\alpha}) & = \widehat{f}_\dagger\circ s_{x_{\alpha}-y_{\alpha}}(x) - \widehat{f}_\dagger\circ s_{x_{\alpha}-y_{\alpha}}(x_{\alpha}) \\
    & \geq \left(P^\alpha[u] \circ s_{x_{\alpha}-y_{\alpha}}\right)(x) - \left(P^\alpha[u] \circ s_{x_{\alpha}-y_{\alpha}}\right) (x_{\alpha}) \\
    & \geq \left(u(x) - \frac{\alpha}{2}d^2(x,s_{x_{\alpha}-y_{\alpha}}(x)\right) - \left(u(x_{\alpha}) - \frac{\alpha}{2}d^2(x_{\alpha},s_{x_{\alpha}-y_{\alpha}}(x_{\alpha})\right) \\
    & = u(x) - u(x_{\alpha})
    \end{aligned}
    \end{equation}
    with equality uniquely realized at $x_{\alpha}$, establishing \ref{item:test_function_construction:fdagger_f_ddagger_test_functions_uv}.

\smallskip

We conclude with the proof of \ref{item:test_function_construction:f1f2_equal_Ralpha_u_Ralpha_v_derivatives}. First of all, note that the equality of first and second order derivatives for $f_\dagger$ and $\widehat{f}_\dagger$ as well as for $f_\ddagger$ and $\widehat{f}_\ddagger$ follows by the chain rule.

The expressions for $D \widehat{f}_\dagger(y_{\alpha})$ and $D \widehat{f}_\ddagger(y_{\alpha}')$ follow from \ref{item:test_function_construction:f1f2_squeeze_Ralpha_u_Ralpha_v} and Proposition \ref{proposition:optimizing_point_construction} \ref{item:proposition:optimizing_point_construction_twice_diff} and \ref{item:proposition:optimizing_point_construction:Ralpha_u_Ralpha_v_optimizers}.
\end{proof}

\section{Proof of the strict comparison principle}\label{section:main_proof}

In this section, we prove Theorem \ref{th:comparison_HJI}. The proof is based on a variant of the variable quadruplication procedure on the basis of
\begin{multline} \label{eqn:quadruplicationvariables}
    \sup_{x \in E} \frac{1}{1-\varepsilon} u(x) - \frac{1}{1+\varepsilon} v(x) \\
    \leq \sup_{x,,y,y'x' \in E} \frac{1}{1-\varepsilon} u(x) - \frac{1}{1+\varepsilon} v(x') - \frac{\alpha}{2(1-\eps)}d^2(x,y) - \frac{\alpha}{2} d^2(y,y')  \\
     -  \frac{\alpha}{2(1+\eps)}d^2(y',x') - \frac{\varepsilon}{1+\varepsilon}V(x) - \frac{\varepsilon}{1+\varepsilon}V(x'),
\end{multline}
which we have formalized in terms of test functions $f_\dagger,f_\ddagger$ in Propositions \ref{proposition:optimizing_point_construction} and \ref{proposition:test_function_construction}. 

In a first step, we relate sub- and supersolutions for the Hamilton--Jacobi equation for $H$ to those for $H_+$ and $H_-$: This will be carried out in Lemma \ref{lemma:operator_pushover}. A second step is to show that $f_\dagger \in \cD(H_+)$ and $f_\ddagger \in \cD(H_-)$: This will be carried out in Lemma \ref{lemma:test_functions_in_domain}.

After establishing these technical points, we proceed to frame the comparison principle in terms of an estimate on
\begin{equation} \label{eqn:motivation_diff_Hamiltonians}
    \frac{H_+ f_\dagger}{1-\varepsilon}  - \frac{H_- f_\ddagger}{1+\varepsilon}.
\end{equation}
This reduction will be carried out in Proposition \ref{proposition:basic_comparison_using_Hestimate}, the statement of which is more involved than typically in the literature, but leads to the improved \textit{strict} comparison principle. Its formulation and proof hinges on the use of $V$ as a Lyapunov function.

The statements of Lemmas \ref{lemma:operator_pushover}, \ref{lemma:test_functions_in_domain}, and Proposition \ref{proposition:basic_comparison_using_Hestimate} can be found in Section \ref{subsection:proof_comparison_toDiffH}, their proofs in Section \ref{subsection:proof_comparison_toDiffH_proof}.

We finish in Section \ref{subsection:proof_comparison_mainProof} by estimating \eqref{eqn:motivation_diff_Hamiltonians} in two steps leading to our final result. We first establish in Lemma \ref{lemma:decomposition_A,B} that the pre-factors $(1-\varepsilon)^{-1}$ and $(1+\varepsilon)^{-1}$ work well with the combinations of functions that define $f_\dagger,f_\ddagger$ in Proposition \ref{proposition:test_function_construction}. We conclude this section with the proof of Theorem \ref{th:comparison_HJI}, where we use this split, the coupling assumption on $\bA$, the semi-monotonicity of $\bB$, modulus of continuity control on $\cI$ and, again, that $V$ is a Lyapunov function to arrive at our final result.

\subsection{Comparison in terms of estimating the difference of  Hamiltonians} \label{subsection:proof_comparison_toDiffH}

We start with connecting the notion of sub- and supersolutions for $H$ to those for $H_+$ and $H_-$, respectively.
\begin{lemma} \label{lemma:operator_pushover}
    Let $H$ and $\bH$ satisfy Assumption \ref{assumption:domain_setup}. Then, for any $h \in C_b(E)$ and $\lambda >0$, we have the following:
    \begin{enumerate}[(a)]
        \item Any viscosity subsolution of $f - \lambda Hf = h$ is also a viscosity subsolution of $f - \lambda H_+ f = h$.
        \item Any viscosity supersolution of $f - \lambda Hf = h$ is also a viscosity supersolution of $f - \lambda H_- f = h$.
        \end{enumerate}
\end{lemma}

The proof follows in Section \ref{subsection:proof_comparison_toDiffH_proof} below. In the next lemma we show that the test functions that we constructed in the previous section are in the domain of $H_+$ and $H_-$.

\begin{lemma} \label{lemma:test_functions_in_domain}
    Let $\bH$ be an operator satisfying Assumptions \ref{assumption:domain_setup} and \ref{assumption:compatibility}. 
    Let $\widehat{f_\dagger}, f_\dagger$ and $\widehat{f_\ddagger}, f_\ddagger$ be as in Proposition \ref{proposition:test_function_construction}.
    Then, $\widehat{f_\dagger}, f_\dagger \in \cD(H_+)$ and $\widehat{f_\ddagger}, f_\ddagger \in \cD(H_-)$.
\end{lemma}

The proof of the lemma is outlined in Section \ref{subsection:proof_comparison_toDiffH_proof} below. We next state our key proposition, which relates the strict comparison principle to an estimate on the difference of Hamiltonians.

\begin{proposition} \label{proposition:basic_comparison_using_Hestimate}
    Let $\bH \subseteq C(E) \times C(E)$ satisfy Assumptions \ref{assumption:domain_setup} and \ref{assumption:compatibility}.  Let $h_1,h_2 \in C_b(E)$, and $\lambda > 0$. Consider the equations
    \begin{align}
        f - \lambda H_+ f \leq h_1, \label{eqn:theorem:HJ_comparison_subsolution} \\
        f - \lambda H_- f \geq h_2. \label{eqn:theorem:HJ_comparison_supersolution}
    \end{align}
    Let $u$ and $v$ by viscosity sub- and supersolutions to \eqref{eqn:theorem:HJ_comparison_subsolution} and \eqref{eqn:theorem:HJ_comparison_supersolution}, respectively. For each $\varepsilon \in (0,1)$, $\varphi \in (0,1]$ and $\alpha > 1$, consider the construction of optimizers $x_{\alpha},x_{\alpha}'$ and test functions $f_\dagger,f_\ddagger$ as in Propositions \ref{proposition:optimizing_point_construction} and \ref{proposition:test_function_construction}. 
    \smallskip

    Suppose there exists a map $\varepsilon \mapsto C_{\varepsilon}^0$, and for any $\varepsilon \in (0,1)$ a non-negative map $\varphi \mapsto C_{\varepsilon,\varphi}$ satisfying $\limsup_{\varepsilon \downarrow 0} C_{\varepsilon}^0 < \infty$ and $\lim_{\varphi \downarrow 0} C_{\varepsilon,\varphi} = 0$ such that
        \begin{equation} \label{eqn:comparison_principle_proof_estimate_on_H-H}
        \liminf_{\alpha \rightarrow \infty} \frac{H_+ f_\dagger(x_{\alpha})}{1-\varepsilon}  - \frac{H_- f_\ddagger(x_{\alpha}')}{1+\varepsilon} \leq \varepsilon \left(C_\eps^0 + C_{\varepsilon,\varphi} \right).
    \end{equation}
    Then, for any compact set $K \subseteq E$ and $\varepsilon \in (0,1)$,
    \begin{equation}
        \sup_{x \in K} u(x) - v(x) \leq \varepsilon C_\eps + \sup_{x \in \widehat{K}} h_1(x) - h_2(x),
    \end{equation}
    where $\widehat{K}_\eps := \widehat{K}_\eps(K,u,v)$ and $C_\varepsilon := C_\varepsilon(K,u,v,h_1,h_2)$ are given by
    \begin{align*}
        \widehat{K}_\eps & := \left\{z \in E \, \middle| \, V(z) \leq \frac{\vn{u} + \vn{v}}{\varepsilon} + \ssup{V}_K \right\}, \\
        C_\eps & := \lambda C_\eps^0  + \frac{2}{1-\varepsilon^2} \ssup{V}_K + \frac{1}{1-\eps} \vn{h_1} + \frac{1}{1-\eps} \vn{h_2} - \iinf{\frac{1}{1-\varepsilon}u - \frac{1}{1+\varepsilon}v }_K .
    \end{align*}
    In particular, the strict comparison principle holds for \eqref{eqn:theorem:HJ_comparison_subsolution} and \eqref{eqn:theorem:HJ_comparison_supersolution}.
\end{proposition}

\subsection{Proof of Lemmas \ref{lemma:operator_pushover}, \ref{lemma:test_functions_in_domain}, and Proposition \ref{proposition:basic_comparison_using_Hestimate}} \label{subsection:proof_comparison_toDiffH_proof}

\begin{proof}[Proof of Lemma \ref{lemma:operator_pushover}]
    We only prove the first statement, the second one follows analogously. Let $u$ be a subsolution to $f - \lambda Hf = h$ and let $(f,g) \in H_+$. Our claim thus follows if there exists $x_0$ satisfying
    \begin{gather}
        u(x_0) - f(x_0) = \ssup{u-f}, \label{eqn:push_over_claim1} \\
        u(x_0) - \lambda g(x_0) \leq h(x_0).  \label{eqn:push_over_claim2}
    \end{gather}
    As $u$ is upper semi-continous and bounded, and $f$ has compact sublevel sets, the existence of $x_0$ satisfying \eqref{eqn:push_over_claim1} is immediate. We thus proceed with \eqref{eqn:push_over_claim2} using the sequential upward denseness of $\cD(H)$ in $\cD(H_+)$, cf. Assumption \ref{assumption:domain_setup} \ref{item:assumption_domain_setup:upward_dense}. Set
    \begin{equation*}
        a := f(x_0)  + \ssup{u} - u(x_0), \qquad A := \left\{x \, \middle| \, f(x) \leq a \right\}.
    \end{equation*}
    We can thus find $(f_a,g_a) \in H$ with $f_a$ satisfying
    \begin{equation*}
        \begin{cases}
            f_a(x) = f(x) & \text{if } x \in A, \\
            a < f_a(x) \leq f(x) & \text{if } x \notin A.
        \end{cases}
    \end{equation*}
    We first establish that
    \begin{equation} \label{eqn:push_over_claim1_with_a}
        u(x_0) - f_a(x_0) = \ssup{u-f_a}.
    \end{equation}
    Using \eqref{eqn:push_over_claim1} and that $f = f_a$ on $A$, \eqref{eqn:push_over_claim1_with_a} follows by verifying that
    \begin{equation*}
        u(x) - f_a(x) < u(x_0) - f(x_0), \qquad x \in A^c,
    \end{equation*}
    which follows from the definition of $a$:
    \begin{align*}
        u(x) - f(x) & < u(x) - a \\
        & = u(x) - \left(f(x_0)  + \ssup{u} - u(x_0)\right) \\
        & = u(x_0) - f(x_0) - \left(\ssup{u} - u(x)\right) \\
        & \leq u(x_0) - f(x_0).
    \end{align*}
    Thus, by \eqref{eqn:push_over_claim1_with_a}, we can use the subsolution inequality for $(f_a,g_a)$ in the point $x_0$. We obtain:
    \begin{equation} \label{eqn:push_over_subprinciple_fa}
         u(x_0) - \lambda g_a(x_0) \leq h(x_0).
    \end{equation}
    Recalling that $f_a(x_0) = f(x_0)$ and $f_a \leq f$, we have
    \begin{equation*}
        f_a(x_0) - f(x_0) = \ssup{f_a - f}.
    \end{equation*}
    Using the positive maximum principle for $\bH$, cf. Assumption \ref{assumption:domain_setup} \ref{item:assumption_domain_setup:max_principle}, thus yields 
    \begin{equation} \label{eqn:push_over_subprinciple_ga}
        g_a(x_0) \leq g(x_0).
    \end{equation}
    Combining \eqref{eqn:push_over_subprinciple_fa} and \eqref{eqn:push_over_subprinciple_ga}, leads to
    \begin{equation*}
        u(x_0) - \lambda g(x_0) \leq u(x_0) - \lambda g_a(x_0) \leq h(x_0),
    \end{equation*}
    establishing \eqref{eqn:push_over_claim2} and consquently that $u$ is a subsolution to $f - \lambda H_+ f = h$.
\end{proof}

\begin{proof}[Proof of Lemma \ref{lemma:test_functions_in_domain}] 
As $f_1, f_2, \widehat{f}_1, \widehat{f}_2 \in C_c^\infty(E)$, it follows by Assumption \ref{assumption:domain_setup} \ref{item:assumption_domain_setup:domainH} that $f_1, f_2, \widehat{f}_1, \widehat{f}_2 \in \cD(H)$.
By compatibility, cf.\ Assumption \ref{assumption:compatibility}, we have $V\circ s_z, \Xi \circ s_z \in \cD(\bH)$. By Assumption \ref{assumption:domain_setup} \ref{item:domain_setup_extended_convex} and the fact that $V$ has compact sublevel sets, cf. Definition \ref{definition:perturbation_containment}, we thus have $(1-\varphi) V \circ s_z + \varphi \Xi \circ s_z \in \cD(H_+)$.
Consequently, $\widehat{f}_\dagger, f_\dagger \in \cD(H_+)$ and $\widehat{f}_\ddagger, f_\ddagger \in \cD(H_-)$ by Assumption \ref{assumption:domain_setup} \ref{item:domain_setup_extended_affine}.
\end{proof}

\begin{proof}[Proof of Proposition \ref{proposition:basic_comparison_using_Hestimate}]

Let $u$ be a subsolution of $f - \lambda H_+ f = h_1$ and $v$ a supersolution of $f - \lambda H_- f = h_2$. Consider the constructions in Propositions \ref{proposition:optimizing_point_construction} and \ref{proposition:test_function_construction} for the subsolution $u$, supersolution $v$ and $\varepsilon \in (0,1)$ and $\varphi \in (0,1]$.

By Lemma \ref{lemma:test_functions_in_domain}, we have $f_\dagger \in \cD(H_+)$ and $f_\ddagger \in \cD(H_-)$ and, by Proposition \ref{proposition:test_function_construction} \ref{item:test_function_construction:fdagger_f_ddagger_test_functions_uv}, we find that $(x_{\alpha},x_{\alpha}')$ are the unique optimizers in
    \begin{equation} \label{eqn:proof_comparison_optimizinguf1vf2}
        \begin{aligned}
            u(x_{\alpha}) - f_\dagger(x_{\alpha}) & = \ssup{u - f_\dagger }, \\
            v(x_{\alpha}') - f_\ddagger(x_{\alpha}') & = \ssup{v - f_\ddagger },
        \end{aligned}
    \end{equation}
    which, by the sub- and supersolution properties for $H_+$ and $H_-$, respectively, and Lemma \ref{lemma:def_equiv}, implies that
    \begin{equation} \label{eqn:comparison_proof_subsuperestimate}
        \begin{split}
        u(x_{\alpha}) - \lambda H_+ f_\dagger(x_{\alpha}) \leq h_1(x_{\alpha}), \\
        v(x_{\alpha}') - \lambda H_- f_\ddagger(x_{\alpha}') \geq h_2(x_{\alpha}').
        \end{split}
    \end{equation}
 By Proposition \ref{proposition:optimizing_point_construction} \ref{item:proposition:optimizing_point_construction_estimate_u-v}, we find
        \begin{equation} \label{eqn:comparison_proof_initial_estimate}
             \ssup{u - v}_K \leq \frac{1}{1-\eps}u(x_{\alpha}) - \frac{1}{1+\eps}v(x_{\alpha}') + \varepsilon\left(c_{\eps,\varphi} + o(1) \right),
        \end{equation}
        where 
        \begin{equation} \label{eqn:Cvarphi}
            c_{\eps,\varphi}:= \frac{2}{1-\varepsilon^2} (1-\varphi) \ssup{V}_K - \iinf{\frac{1}{1-\varepsilon}u - \frac{1}{1+\varepsilon}v }_K,
        \end{equation}
        and $o(1)$ is in terms of $\alpha \rightarrow \infty$. Using \eqref{eqn:comparison_proof_subsuperestimate}, we estimate
        \begin{align*}
            \ssup{u - v}_K  & \leq \frac{1}{1-\eps}u(x_{\alpha}) - \frac{1}{1+\eps}v(x_{\alpha}') + \varepsilon\left(c_{\eps,\varphi} + o(1) \right) \\
            & \leq \frac{1}{1-\eps}h_1(x_{\alpha}) - \frac{1}{1+\eps}h_2(x_{\alpha}') + \lambda \left[ \frac{H_+ f_\dagger(x_{\alpha})}{1-\varepsilon} - \frac{H_- f_\ddagger(x_{\alpha}')}{1+\varepsilon} \right] + \varepsilon\left(c_{\eps,\varphi} + o(1) \right) \\
            & \leq h_1(x_{\alpha}) - h_2(x_{\alpha}')  + \lambda \left[ \frac{H_+ f_\dagger(x_{\alpha})}{1-\varepsilon}  - \frac{H_- f_\ddagger(x_{\alpha}')}{1+\varepsilon} \right] \\
            & \hspace{5cm} + \frac{\varepsilon}{1-\varepsilon} \vn{h_1} + \frac{\varepsilon}{1+\varepsilon}\vn{h_2} + \varepsilon\left(c_{\eps,\varphi} + o(1) \right).
        \end{align*}
        We next expand $c_{\eps,\varphi}$ from \eqref{eqn:Cvarphi}. Furthermore, taking $\liminf_{\alpha\rightarrow \infty}$ on the right-hand side, using Proposition \ref{proposition:optimizing_point_construction} \ref{item:proposition:optimizing_limits} to treat the difference $h_1 - h_2$, and \eqref{eqn:comparison_principle_proof_estimate_on_H-H} to treat the difference of Hamiltonians, we find
        \begin{multline}
            \ssup{u - v}_K \leq \ssup{h_1 - h_2}_{\widehat{K}} + \lambda \left(\eps C_0 + C_{\eps,\varphi}\right) + \frac{\eps}{1-\eps}\vn{h_1} + \frac{\eps}{1+\eps}\vn{h_2}\\ 
            + \eps\left(\frac{2}{1-\varepsilon^2} (1-\varphi) \ssup{V}_K - \iinf{\frac{1}{1-\varepsilon}u - \frac{1}{1+\varepsilon}v }_K\right).
        \end{multline}
        As $\varphi \in (0,1]$ was arbitrary, we can take the limit for $\varphi \downarrow 0$, which leads to  
        \begin{multline}
            \ssup{u - v}_K \leq \ssup{h_1 - h_2}_{\widehat{K}}\\ 
            + \eps \left(\lambda C_\eps^0  + \frac{2}{1-\varepsilon^2} \ssup{V}_K + \frac{1}{1-\eps}\vn{h_1} + \frac{1}{1+\eps}\vn{h_2}  - \iinf{\frac{1}{1-\varepsilon}u - \frac{1}{1+\varepsilon}v }_K  \right),
        \end{multline}
        establishing the claim.
\end{proof}

\subsection{Proof of Theorem \ref{th:comparison_HJI}} \label{subsection:proof_comparison_mainProof}
We start with an auxiliary lemma that provides a detailed decomposition of the operators $\bA$ and $\bB$ evaluated in the test functions. 

\begin{lemma}\label{lemma:decomposition_A,B}
     Let $\bA$ and $\bB$ both satisfy Assumption \ref{assumption:domain_setup} and Assumption \ref{assumption:compatibility} \ref{item:compatA} and \ref{item:compatB}, respectively. Fix $z_0,z_1\in \bR^q$ and $p \in \bR^q$. Let $\Xi = \Xi_{z_0,p,z_1}$ as in Definition \ref{definition:perturbation_first_second_order} and, for $\widehat{f} \in C_c^\infty (E)$, $\varepsilon \in (0,1)$, and $\varphi \in (0,1]$, set
    \begin{align*}
        \widehat{f}_\dagger & :=  (1-\varepsilon) \widehat{f} + \varepsilon (1-\varphi) V + \varepsilon \varphi \Xi, \\
        \widehat{f}_\ddagger & :=  (1+\varepsilon) \widehat{f} - \varepsilon (1-\varphi) V - \varepsilon \varphi \Xi.
    \end{align*}
    For $z \in E$, set $f_\dagger = \widehat{f}_\dagger \circ s_z$, and $f_\ddagger = \widehat{f}_\ddagger \circ s_z$. Then, the following statements hold:
      \begin{enumerate}[(a)]
          \item $f_\dagger \in \cD(A_+)$ and $f_\ddagger \in \cD(A_-)$. Suppose furthermore that $\bA$ is linear on its domain, then
              \begin{align}\label{eq:lemma:linearity_testfunction}
                \frac{A_+ f_\dagger}{1-\varepsilon} & = A(\widehat{f} \circ s_z) + \frac{\varepsilon}{1-\varepsilon} (1-\varphi) A_+ \left(V \circ s_z\right) + \frac{\varepsilon}{1-\varepsilon} \varphi \bA \left(\Xi \circ s_z\right),\\
                \frac{A_- f_\ddagger}{1+\varepsilon} & = A(\widehat{f}\circ s_z) - \frac{\varepsilon}{1+\varepsilon} (1-\varphi) A_+ \left(V \circ s_z\right) - \frac{\varepsilon}{1+\varepsilon} \varphi \bA \left(\Xi \circ s_z\right),
            \end{align}
            \item $f_\dagger, \widehat{f}_\dagger \in \cD(B_+)$ and $f_\ddagger, \widehat{f}_\ddagger \in \cD(B_-)$. Suppose furthermore that $\bB$ is convex, then for any $x,y$ such that $z = x-y$, we have
            \begin{align}\label{eq:lemma:convexity_testfunction}
                \frac{B_+ f_\dagger}{1-\varepsilon}(x) & \leq \frac{1}{1-\varepsilon} \left(B_+ f_\dagger (x) - B_+ \widehat{f}_\dagger (y) \right)  +  B\widehat{f} (y) \\
                & \hspace{2cm} +  \frac{\varepsilon}{1-\varepsilon} (1-\varphi) B_+ V(y) + \frac{\varepsilon}{1-\varepsilon} \varphi B_+ \Xi(y),\\
                \frac{B_- f_\ddagger}{1+\varepsilon}(x) & \geq \frac{1}{1+\varepsilon}\left(B_- f_\ddagger(x) - B_- \widehat{f}_\ddagger(y) \right) + B \widehat{f}(y) \\
                & \hspace{2cm} - \frac{\varepsilon}{1+\varepsilon} (1-\varphi) B_- V(y) - \frac{\varepsilon}{1+\varepsilon} \varphi B_- \Xi(y).
            \end{align}
      \end{enumerate}
\end{lemma}

\begin{proof}
The domain statements $f_\dagger \in \cD(A_+)$, $f_\ddagger \in \cD(A_-)$, $f_\dagger,\widehat{f}_\dagger \in \cD(B_+)$ and $f_\ddagger,\widehat{f}_\ddagger \in \cD(B_-)$ follow by Lemma \ref{lemma:test_functions_in_domain}. The four statements in \eqref{eq:lemma:linearity_testfunction} and \eqref{eq:lemma:convexity_testfunction} follow from linearity of $A_+$ and convexity of $B_+$.
\end{proof}

\begin{proof}[Proof of Theorem \ref{th:comparison_HJI}]
 To prove inequality \eqref{eq:th:final}, and consequently the strong comparison principle for the Hamilton--Jacobi equation in terms of $H$,  it suffices by Lemma \ref{lemma:operator_pushover} and Proposition \ref{proposition:basic_comparison_using_Hestimate} to establish \eqref{eqn:comparison_principle_proof_estimate_on_H-H}, which we repeat for readability:
\begin{equation} \label{eqn:comparison_principle_proof_estimate_on_H-H_inproof}
        \liminf_{\alpha \rightarrow \infty} \frac{H_+ f_\dagger(x_{\alpha})}{1-\varepsilon}  - \frac{H_- f_\ddagger(x_{\alpha}')}{1+\varepsilon} \leq \varepsilon \left(C_\eps^0 + C_{\varepsilon,\varphi} \right).
\end{equation}
Let $\theta^*_{1, \alpha} \in \Theta_1$ be such that 
\begin{align}
    H_+ f_\dagger(x_\alpha) &= \sup_{\theta_1 \in\Theta_1}\inf_{\theta_2\in\Theta_2}\left\{\bA_{\theta_{1},\theta_2} f_\dagger(x_\alpha) + \bB_{\theta_{1},\theta_2} f_\dagger(x_\alpha) - \cI(x_\alpha,\theta_{1},\theta_2)\right\}\\
    &=\inf_{\theta_2\in\Theta_2} \left\{\bA_{\theta^*_{1, \alpha},\theta_2} f_\dagger(x_\alpha) + \bB_{\theta^*_{1, \alpha},\theta_2} f_\dagger(x_\alpha) - \cI(x_\alpha,\theta^*_{1, \alpha},\theta_2)\right\}.
\end{align}
Such optimizer exists by the compactness of $\Theta_1$ and the lower semi-continuity of $\cI$ in $\theta_1$ assumed in \ref{item:assumption_directHJI_lsc}. By Isaacs' condition \ref{item:assumption:isaacs_cond}, we can write
\begin{equation}
    H_-f_\ddagger(x'_\alpha) = \inf_{\theta_2\in\Theta_2}\sup_{\theta_1\in\Theta_1} \left\{\bA_{\theta_{1},\theta_2} f_\ddagger(x'_\alpha) + \bB_{\theta_{1},\theta_2} f_\ddagger(x_\alpha) - \cI(x'_\alpha,\theta_{1},\theta_2)\right\}.
\end{equation}
Then, by compactness of $\Theta_2$ and the upper semi-continuity of $\cI$ in $\theta_2$ assumed in \ref{item:assumption_directHJI_lsc}, we can find $\theta_{2,\alpha}^*\in\Theta_2$ such that
\begin{equation}
    H_- f_\ddagger(x'_\alpha) = \sup_{\theta_1\in \Theta_1} \left\{\bA_{\theta_1,\theta_{2, \alpha}^*} f_\ddagger(x'_\alpha) + \bB_{\theta_1,\theta_{2, \alpha}^*}f_\ddagger(x'_\alpha) - \cI(x_\alpha', \theta_1,\theta_{2, \alpha}^*)\right\}.
\end{equation}
Consequently, we can estimate
\begin{align}
    \frac{1}{1-\varepsilon}  H_+ f_\dagger(x_\alpha) - \frac{1}{1+\varepsilon} H_- f_\ddagger(x'_\alpha) & \leq \underbrace{\left[\frac{1}{1-\varepsilon} \bA_{\theta^*_{1, \alpha}, \theta^*_{2, \alpha}} f_\dagger(x_\alpha) - \frac{1}{1+\varepsilon}\bA_{\theta^*_{1, \alpha}, \theta^*_{2, \alpha}}f_\ddagger(x'_\alpha)\right]}_{\hypertarget{proof:convex_split:a}{(1)}} \\
    & \;\; + \underbrace{\left[\frac{1}{1-\varepsilon}  \bB_{\theta^*_{1, \alpha}, \theta^*_{2, \alpha}} f_\dagger(x_\alpha) - \frac{1}{1+\varepsilon} \bB_{\theta^*_{1, \alpha}, \theta^*_{2, \alpha}} f_\ddagger(x'_\alpha)\right]}_{\hypertarget{proof:convex_split:b}{(2)}} \\
    & \;\; + \underbrace{\left[ \frac{1}{1+\varepsilon}  \cI(x'_\alpha,\theta^*_{1, \alpha}, \theta^*_{2, \alpha}) - \frac{1}{1-\varepsilon}\cI(x_\alpha,\theta^*_{1, \alpha}, \theta^*_{2, \alpha})\right]}_{\hypertarget{proof:convex_split:i}{(3)}}.
\end{align}
We treat \hyperlink{proof:convex_split:a}{$(1)$}, \hyperlink{proof:convex_split:b}{$(2)$}, and \hyperlink{proof:convex_split:i}{$(3)$} separately. Note, that due the compactness of $\Theta_1$ and $\Theta_2$, the sequences of optimizers $\theta^*_{1, \alpha}$ and $\theta^*_{2,\alpha}$ converge to some $\theta^*_1$ and $\theta^*_2$, respectively.

{\bfseries Estimate \hyperlink{proof:convex_split:a}{$(1)$}:}
Using the expansions of $A_+ f_\dagger$ and $A_- f_\ddagger$ obtained in Lemma \ref{lemma:decomposition_A,B} we find
\begin{align}\label{eq:mainth_proof_decomposition_A}
    \frac{\bA_{\theta^*_{1, \alpha}, \theta^*_{2, \alpha}} f_\dagger (x_\alpha)}{1-\eps} - \frac{\bA_{\theta^*_{1, \alpha}, \theta^*_{2, \alpha}} f_\ddagger (x_\alpha')}{1+\eps}&=
    \frac{A_{\theta^*_{1, \alpha}, \theta^*_{2, \alpha}, +} f_\dagger (x_\alpha)}{1-\eps} - \frac{A_{\theta^*_{1, \alpha}, \theta^*_{2, \alpha},-} f_\ddagger (x_\alpha')}{1+\eps}\nonumber\\
    &\leq A_{\theta^*_{1, \alpha}, \theta^*_{2, \alpha}} f_1(x_{\alpha}) - A_{\theta^*_{1, \alpha}, \theta^*_{2, \alpha}} f_2(x_{\alpha}')\nonumber\\
    &\quad + \frac{\varepsilon}{1-\varepsilon} (1-\varphi) A_{\theta^*_{1, \alpha}, \theta^*_{2, \alpha},+} \left( V \circ s_{x_{\alpha} - y_{\alpha}} \right)(x_{\alpha})\nonumber\\
    &\quad + \frac{\varepsilon}{1+\varepsilon} (1-\varphi) A_{\theta^*_{1, \alpha}, \theta^*_{2, \alpha}, +} \left( V \circ s_{x_{\alpha}' - y_{\alpha}'} \right)(x_{\alpha}')\nonumber \\
    &\quad + \frac{\varepsilon}{1-\varepsilon} \varphi \bA_{\theta^*_{1, \alpha}, \theta^*_{2, \alpha}} \left( \Xi_1 \circ s_{x_{\alpha} - y_{\alpha}} \right)(x_{\alpha})\nonumber\\
    &\quad + \frac{\varepsilon}{1+\varepsilon} \varphi \bA_{\theta^*_{1, \alpha}, \theta^*_{2, \alpha}} \left( \Xi_2 \circ s_{x_{\alpha}' - y_{\alpha}'} \right)(x_{\alpha}').
\end{align}
We first consider the terms involving $V$ and $\Xi$. By Proposition \ref{proposition:optimizing_point_construction} \ref{item:proposition:optimizing_limits}, we have that, along subsequences, the optimizers $(x_\alpha ,y_\alpha ,y_{\alpha,0} ,y'_{\alpha,0} ,y'_\alpha, x'_\alpha)$ converge to $(z,z,z,z,z,z)$ with $z \in \widehat{K}$ and $p_\alpha,p'_\alpha \in B_{1/\alpha}(0)$. Then, using the compatibility of $\bA_{\theta_1, \theta_2}$, cf. Assumption \ref{assumption:compatibility}, we find

\begin{align}
    \liminf_{\alpha\to\infty}  & \frac{\varepsilon}{1-\varepsilon} (1-\varphi) A_{\theta^*_{1, \alpha}, \theta^*_{2, \alpha},+} \left( V \circ s_{x_{\alpha} - y_{\alpha}} \right)(x_{\alpha})\nonumber\\
    & \quad + \frac{\varepsilon}{1+\varepsilon} (1-\varphi) A_{\theta^*_{1, \alpha}, \theta^*_{2, \alpha},+} \left( V \circ s_{x_{\alpha}' - y_{\alpha}'} \right)(x_{\alpha}')\nonumber\\
    & \quad + \frac{\varepsilon}{1-\varepsilon} \varphi \bA_{\theta^*_{1, \alpha}, \theta^*_{2, \alpha}} \left( \Xi_1 \circ s_{x_{\alpha} - y_{\alpha}} \right)(x_{\alpha}) + \frac{\varepsilon}{1+\varepsilon} \varphi \bA_{\theta^*_{1, \alpha}, \theta^*_{2, \alpha}} \left( \Xi_2 \circ s_{x_{\alpha}' - y_{\alpha}'} \right)(x_{\alpha}')\nonumber \\
    & \leq \frac{2\eps}{1-\eps^2} \left( (1-\varphi) A_{\theta^*_{1}, \theta^*_{2},+} (V)(z) + \varphi \bA_{\theta^*_{1}, \theta^*_{2}}(\Xi_{z,0,z})(z) \right).
    \label{eq:mainth_proof_limsupA}
\end{align}
Next, we consider the second line in \eqref{eq:mainth_proof_decomposition_A}. Using that, for all $\theta_1$, $\theta_2$, $\bA_{\theta_1, \theta_2}$ has a controlled growth coupling $\widehat{\bA}_{\theta_1, \theta_2}$ with a modulus uniform in $\theta_1$ and $\theta_2$ satisfying the maximum principle and Proposition \ref{proposition:test_function_construction} \ref{item:test_function_construction:optimizing_f1f2}, we find
    \begin{align}\label{eq:mainth_proof_coupling}
    A_{\theta^*_{1, \alpha}, \theta^*_{2, \alpha}} f_1(x_{\alpha}) - A_{\theta^*_{1, \alpha}, \theta^*_{2, \alpha}} f_2(x_{\alpha}') & = \widehat{\bA}_{\theta^*_{1, \alpha}, \theta^*_{2, \alpha}} \left(f_1 \ominus f_2\right) (x_{\alpha},x_{\alpha}')\nonumber \\
    & \leq \widehat{\bA}_{\theta^*_{1, \alpha}, \theta^*_{2, \alpha}} \left(\frac{\alpha}{2} d^2_{x_{\alpha}-y_{\alpha},x_{\alpha}'-y_{\alpha}'}\right)(x_{\alpha},x_{\alpha}')\nonumber \\
    & \leq \omega_{\widehat{\bA}, \widehat{K}} \left(\alpha \left(d(x_{\alpha},y_{\alpha}) + d(y_{\alpha},y_{\alpha}') + d(y_{\alpha}', x_{\alpha}')\right)^2 \right.\nonumber\\
    & \hspace{0.5cm}   +\left(d(x_{\alpha},y_{\alpha}) + d(y_{\alpha},y_{\alpha}') + d(y_{\alpha}', x_{\alpha}')\right)\Big),
\end{align}
which converges to $0$ as $\alpha \rightarrow \infty$ by Proposition \ref{proposition:optimizing_point_construction} \ref{item:proposition:optimizing_point_construction_2optimizers_convergence}.

{\bfseries Estimate \hyperlink{proof:convex_split:b}{$(2)$}:}
By using the expansions of $B_+ f_\dagger$ and $B_- f_\ddagger$ obtained in Lemma \ref{lemma:decomposition_A,B}, we find
\begin{align}\label{eq:mainth_proof_decomposition_B}
    \frac{\bB_{\theta^*_{1, \alpha}, \theta^*_{2, \alpha}} f_\dagger (x_\alpha)}{1-\eps} &- \frac{B_{\theta^*_{1, \alpha}, \theta^*_{2, \alpha}} f_\ddagger (x_\alpha')}{1+\eps} =
    \frac{B_{\theta^*_{1, \alpha}, \theta^*_{2, \alpha},+} f_\dagger (x_\alpha)}{1-\eps} - \frac{B_{\theta^*_{1, \alpha}, \theta^*_{2, \alpha}, -} f_\ddagger (x_\alpha')}{1+\eps}\\
    &\leq B_{\theta^*_{1, \alpha}, \theta^*_{2, \alpha}} \widehat{f}_1(y_{\alpha}) - B_{\theta^*_{1, \alpha}, \theta^*_{2, \alpha}} \widehat{f}_2(y_{\alpha}') \\
    &+ \frac{1}{1-\varepsilon} \left(B_{\theta^*_{1, \alpha}, \theta^*_{2, \alpha}, +} f_\dagger(x_{\alpha}) - B_{\theta^*_{1, \alpha}, \theta^*_{2, \alpha}, +} \widehat{f}_\dagger(y_{\alpha}) \right) \\
    & + \frac{1}{1+\varepsilon} \left(B_{\theta^*_{1, \alpha}, \theta^*_{2, \alpha}, -} \widehat{f}_\ddagger(y_{\alpha}') - B_{\theta^*_{1, \alpha}, \theta^*_{2, \alpha}, -} f_\ddagger(x_{\alpha}') \right) \\
    &+ \frac{\varepsilon}{1-\varepsilon} (1-\varphi) B_{\theta^*_{1, \alpha}, \theta^*_{2, \alpha}, +} V(y_{\alpha}) + \frac{\varepsilon}{1+\varepsilon} (1-\varphi) B_{\theta^*_{1, \alpha}, \theta^*_{2, \alpha}, +} V(y_{\alpha}')\\
    &+ \frac{\varepsilon}{1-\varepsilon} \varphi \bB_{\theta^*_{1, \alpha}, \theta^*_{2, \alpha}} \Xi_1(y_{\alpha}) + \frac{\varepsilon}{1+\varepsilon} \varphi \bB_{\theta^*_{1, \alpha}, \theta^*_{2, \alpha}} \Xi_2(y_{\alpha}').
\end{align}
Again, by sending $\alpha\to\infty$, using Proposition \ref{proposition:optimizing_point_construction} \ref{item:proposition:optimizing_limits}, and the compatibility of $\bB_{\theta_1, \theta_2}$, cf. Assumption \ref{assumption:compatibility}, we obtain that
\begin{align}\label{eq:mainth_proof_limsupB}
    \liminf_{\alpha\to\infty} & \frac{\varepsilon}{1-\varepsilon} (1-\varphi) B_{\theta^*_{1, \alpha}, \theta^*_{2, \alpha}, +} V(y_{\alpha}) + \frac{\varepsilon}{1+\varepsilon} (1-\varphi) B_{\theta^*_{1, \alpha}, \theta^*_{2, \alpha}, +} V(y_{\alpha}')\\
    & + \frac{\varepsilon}{1-\varepsilon} \varphi \bB_{\theta^*_{1, \alpha}, \theta^*_{2, \alpha}} \Xi_1(y_{\alpha}) + \frac{\varepsilon}{1+\varepsilon} \varphi \bB_{\theta^*_{1, \alpha}, \theta^*_{2, \alpha}} \Xi_2(y_{\alpha}') \\
    &\leq \frac{2\eps}{1-\eps^2} \left( (1-\varphi) B_{\theta^*_1, \theta^*_2, +}(V)(z) + \varphi \bB_{\theta^*_{1}, \theta^*_{2}}(\Xi_{z,0,z})(z)\right).
\end{align}

Using that, for all $\theta_1$, $\theta_2$, $\bB_{\theta_{1}, \theta_{2}}$ is semi-monotone with $\cB_{\theta_{1}, \theta_{2}}$ and the expressions for the gradients obtained in Proposition \ref{proposition:test_function_construction}, we find that
\begin{align}\label{eq:mainth_proof_convexity}
    &\frac{1}{1-\varepsilon} \left(B_{\theta^*_{1, \alpha}, \theta^*_{2, \alpha}, +} f_\dagger(x_{\alpha}) - B_{\theta^*_{1, \alpha}, \theta^*_{2, \alpha}, +} \widehat{f}_\dagger(y_{\alpha}) \right) \nonumber\\
    &\qquad \qquad + B_{\theta^*_{1, \alpha}, \theta^*_{2, \alpha}} \widehat{f}_1(y_{\alpha}) - B_{\theta^*_{1, \alpha}, \theta^*_{2, \alpha}} \widehat{f}_2(y_{\alpha}') \nonumber\\
    &\qquad \qquad + \frac{1}{1+\varepsilon} \left(B_{\theta^*_{1, \alpha}, \theta^*_{2, \alpha}, -} \widehat{f}_\ddagger(y_{\alpha}') - B_{\theta^*_{1, \alpha}, \theta^*_{2, \alpha}, -} f_\ddagger(x_{\alpha}') \right)\nonumber\\
    & = \frac{1}{1-\varepsilon} \left(\cB_{\theta^*_{1, \alpha}, \theta^*_{2, \alpha}}(x_{\alpha},\alpha(x_{\alpha} - y_{\alpha}) )- \cB_{\theta^*_{1, \alpha}, \theta^*_{2, \alpha}}(y_{\alpha},\alpha(x_{\alpha} - y_{\alpha}))\right)\nonumber\\
    &\qquad \qquad + \cB_{\theta^*_{1, \alpha}, \theta^*_{2, \alpha}}(y_{\alpha},\alpha(y_{\alpha} - y_{\alpha}')) - \cB_{\theta^*_{1, \alpha}, \theta^*_{2, \alpha}}(y_{\alpha}',\alpha(y_{\alpha} - y_{\alpha}'))\nonumber \\
    &\qquad \qquad + \frac{1}{1+\varepsilon} \left(\cB_{\theta^*_{1, \alpha}, \theta^*_{2, \alpha}}(y_{\alpha},\alpha(y_{\alpha}' - x_{\alpha}')) - \cB_{\theta^*_{1, \alpha}, \theta^*_{2, \alpha}}(x_{\alpha}',\alpha(y_{\alpha}' - y_{\alpha}'))\right).
\end{align}
By the semi-monotonicity property of $\bB_{\theta_1,\theta_2}$, \eqref{eq:mainth_proof_convexity} is bounded by
\begin{align}\label{eq:mainth_proof_convexity_modulus}
     &\frac{1}{1-\varepsilon}\omega_{\cB,\widehat{K}}(d(x_\alpha,y_{\alpha})+ \alpha d^2(x_{\alpha},y_{\alpha})) + \omega_{\cB, \widehat{K}}(d(y_\alpha,y'_\alpha) +\alpha d^2(y_{\alpha},y_{\alpha}'))\nonumber\\
    &\qquad \qquad + \frac{1}{1+\varepsilon}\omega_{\cB, \widehat{K}}(d(y_\alpha',x_\alpha') +\alpha d^2(y_{\alpha}',x_{\alpha}')).
\end{align}
Thus, taking the $\liminf_{\alpha \rightarrow \infty}$ gives $0$ by Proposition \ref{proposition:optimizing_point_construction} \ref{item:proposition:optimizing_point_construction_2optimizers_convergence}. 

{\bfseries Estimate \hyperlink{proof:convex_split:i}{$(3)$}:}
We have
\begin{align*}
    & \frac{1}{1+\varepsilon}  \cI(x'_\alpha,\theta^*_{1, \alpha}, \theta^*_{2, \alpha}) - \frac{1}{1-\varepsilon}\cI(x_\alpha,\theta^*_{1, \alpha}, \theta^*_{2, \alpha})\\
    & =  \left[ \cI(x'_\alpha,\theta^*_{1, \alpha}, \theta^*_{2, \alpha}) - \cI(x_\alpha,\theta^*_{1, \alpha}, \theta^*_{2, \alpha})\right] - \frac{\varepsilon}{1-\varepsilon} \cI(x_\alpha,\theta^*_{1, \alpha}, \theta^*_{2, \alpha}) - \frac{\varepsilon}{1+\varepsilon} \cI(x'_\alpha,\theta^*_{1, \alpha}, \theta^*_{2, \alpha}).
\end{align*}
By assumption, $\cI$ admits a modulus of continuity $\omega_{\cI, K}$, uniform in $\theta_1$, $\theta_2$, implying
\begin{multline}
    \frac{1}{1+\varepsilon}  \cI(x'_\alpha,\theta^*_{1, \alpha}, \theta^*_{2, \alpha}) - \frac{1}{1-\varepsilon}\cI(x_\alpha,\theta^*_{1, \alpha}, \theta^*_{2, \alpha}) \\
    \leq \omega_{\cI, \widehat{K}} (d(x_{\alpha}, x'_{\alpha}))  - \frac{\varepsilon}{1-\varepsilon} (1-\varphi) \cI(x_\alpha,\theta^*_{1, \alpha}, \theta^*_{2, \alpha}) - \frac{\varepsilon}{1+\varepsilon} (1-\varphi) \cI(x'_\alpha,\theta^*_{1, \alpha}, \theta^*_{2, \alpha}).
\end{multline}
Sending $\alpha\to\infty$, using the lower semi-continuity of $\cI$, and using Proposition \ref{proposition:optimizing_point_construction} \ref{item:proposition:optimizing_limits}, we find
\begin{equation}\label{eq:mainth_proof_costterm}
\begin{aligned}
        & \liminf_{\alpha \rightarrow \infty} \frac{1}{1+\varepsilon} \cI(x'_\alpha,\theta^*_{1, \alpha}, \theta^*_{2, \alpha}) - \frac{1}{1-\varepsilon} \cI(x_\alpha,\theta^*_{1, \alpha}, \theta^*_{2, \alpha}) \\
        & \quad \leq \liminf_{\alpha \rightarrow \infty} \omega_{\cI, \widehat{K}} (d(x_{\alpha}, x'_{\alpha})) \\
        & \qquad + \limsup_{\alpha \rightarrow \infty} \left[ - \frac{\varepsilon}{1-\varepsilon} (1-\varphi) \cI(x_\alpha,\theta^*_{1, \alpha}, \theta^*_{2, \alpha}) - \frac{\varepsilon}{1+\varepsilon} (1-\varphi) \cI(x'_\alpha,\theta^*_{1, \alpha}, \theta^*_{2, \alpha}) \right] \\
        & \quad \leq - \frac{2\varepsilon}{1-\varepsilon^2} (1-\varphi) \cI(z,\theta^*_{1}, \theta^*_{2}).
\end{aligned}
\end{equation}

{\bfseries Conclusion:}
Putting together \eqref{eq:mainth_proof_limsupA}, \eqref{eq:mainth_proof_coupling}, \eqref{eq:mainth_proof_limsupB}, \eqref{eq:mainth_proof_convexity_modulus}, and \eqref{eq:mainth_proof_costterm}, we can conclude that
\begin{align}
     \liminf_{\alpha \rightarrow \infty} \frac{H_+ f_\dagger (x_\alpha)}{1-\eps} - \frac{H_- f_\ddagger (x_\alpha')}{1+\eps} &\leq \frac{2\eps}{1-\eps^2} \left( (1-\varphi) A_{\theta^*_1, \theta^*_2, +} V(z) + \varphi \bA_{\theta^*_1, \theta^*_2}(\Xi_{z,0,z})(z) \right) \\
    &   \quad + \frac{2\eps}{1-\eps^2} \left( (1-\varphi) B_{\theta^*_1, \theta^*_2, +} V(z) + \varphi \bB_{\theta^*_1, \theta^*_2}(\Xi_{z,0,z})(z)\right) \\
    &   \quad - \frac{2\varepsilon}{1-\varepsilon^2} (1-\varphi) \cI(z,\theta^*_{1}, \theta^*_{2}) \\
    &   \leq \frac{2\eps}{1-\eps^2} (1-\varphi)\ssup{(A_{\theta^*_1, \theta^*_2, +} + B_{\theta^*_1, \theta^*_2, +})(V) - \cI(\cdot, \theta^*_{1}, \theta^*_{2})} \\
    & \quad + \frac{2\eps}{1-\eps^2} \varphi \ssup{(\bA_{\theta^*_1, \theta^*_2} + \bB_{\theta^*_1, \theta^*_2})(\Xi_{\cdot,0,\cdot})}_{\widehat{K}_\eps}  \\
    & \leq \eps \left( \frac{2}{1-\eps^2} c_{V} + \frac{2}{1-\eps^2} \varphi \ssup{(\bA_{\theta^*_1, \theta^*_2} + \bB_{\theta^*_1, \theta^*_2})(\Xi_{\cdot,0,\cdot})}_{\widehat{K}_\eps} \right)  \\
    & \leq \eps \left(C_\eps^0 + C_{\eps,\varphi} \right)
\end{align}
with $c_{V}$ given by \eqref{eq:mainth_HJI_bound}, and $C_\eps^0$ and $C_{\eps,\varphi}$ defined via the last two lines. The estimate on the difference of Hamiltonians \eqref{eqn:comparison_principle_proof_estimate_on_H-H_inproof} and thus \eqref{eqn:comparison_principle_proof_estimate_on_H-H} of Proposition \ref{proposition:basic_comparison_using_Hestimate} are satisfied. As a consequence our final estimate \eqref{eq:th:final} and, consequently, the strong comparison principle follow.
\end{proof}

\appendix
\section{The Jensen perturbation}\label{sec:jensen}

The main result of this section is Proposition \ref{proposition:Jensen_Alexandrov_cutoff} that allows us to perturb a semi-convex function with a unique extreme point such that we get a new extreme point close by, in which the function is twice differentiable. The result is a variant of the well-known perturbation result by Jensen, see e.g. \cite[Lemma A.3]{CIL92}.

\begin{proposition} \label{proposition:Jensen_Alexandrov_cutoff}
    Fix $\eta > 0$.  Let $\phi : E \times E \rightarrow \bR$ be bounded above and semi-convex with convexity constant $\kappa \geq 1$. Suppose that $(x_0,y_0)$ is an optimizer of
    \begin{equation*}
        \phi(x_0,y_0) = \ssup{\phi}.
    \end{equation*}
    Let $R > 0$, $\{\zeta_{z,p}\}_{z \in E, p \in \bR^q} \subset C(E)$ and $\{\xi_{z}\}_{z \in E} \subset C^1(E)$ and semi-concavity constant $\kappa_\xi$ be as in Definition \ref{definition:perturbation_first_second_order}.

    Fix  $\varepsilon_1,\varepsilon_2 > 0$ such that $1 - (\varepsilon_1+\varepsilon_2) \kappa_\xi > 0$. Furthermore, define for $p= (p_1,p_2) \in \bR^{q} \times \bR^q$ the perturbed functions
    \begin{equation} \label{eqn:def:phi_p_eta}
        \phi_{p}(x,y) := \phi(x,y) - \eps_1 \left(\xi_{x_0} (x) + \zeta_{x_0, p_1} (x)\right) - \eps_2 \left( \xi_{y_0} (y) + \zeta_{y_0, p_2} (y)\right).
    \end{equation}
    Then there exist $p_1,p_2 \in B_\eta(0)$, and a pair $(x_1,y_1) \in B_\eta(x_0) \times B_\eta(y_0)$ globally maximizing $\phi_{p}$ at which $\phi_{p}$ is twice differentiable. 
\end{proposition}

\begin{corollary} \label{corollary:Jensen_optimizerbound}
    For $\eta >0$, $p$ and $(x_1,y_1)$ as in Proposition \ref{proposition:Jensen_Alexandrov_cutoff}, we have
\begin{equation}\label{eqn:Jensen_control_optimizationproblem_onlyPerturb}
        0 \leq - \eps_1 \left(\xi_{x_0} (x_1) + \zeta_{x_0, p_1} (x_1)\right) - \eps_2 \left( \xi_{y_0} (y_1) + \zeta_{y_0, p_2} (y_1)\right) \leq \eps_1 \eta + \eps \eta_2,
    \end{equation}
and
    \begin{equation} \label{eqn:Jensen_control_optimizationproblem}
        \ssup{\phi} \leq \phi_{p,\eps}(x_1,y_1) \leq \ssup{\phi} + \eps_1 \eta + \eps \eta_2.
    \end{equation}
\end{corollary}

The proof of the perturbation proposition is based partly on results from set-valued analysis. To facilitate the proof, we first introduce the necessary auxiliary definitions and results.

\begin{definition}
    A set-valued function $\Gamma : A \rightrightarrows B$ is called \emph{upper hemi-continuous at $a \in A$}, if, for all open neighbourhoods $V \subseteq B$ of $\Gamma(a)$ (meaning that $\Gamma(a) \subseteq V$), there exists a neighbourhood $U$ of $a$ such that, for all $x \in U$, we have $\Gamma(x) \subseteq V$.

    If $A,B$ are metric, this can equivalently formulated in terms of sequences: A set-valued map $\Gamma : A \rightrightarrows B$, which takes closed values, is upper hemi-continuous at $a$, if, for any sequence $a_n \rightarrow a$ and $b_n \in \Gamma(a_n)$ satisfying $b_n \rightarrow b$, we have $b \in \Gamma(a)$.

    We say that $\Gamma$ is upper hemi-continuous, if it is upper hemi-continuous at all points.
\end{definition}

\begin{lemma} \label{lemma:upper_hemi_continuity_maximum_map}
    Let $K$ be a compact metric space and let $\Xi$ be a metric space.

    For any $\xi \in \Xi$, let $\phi_\xi \in C(K)$ and suppose that the map $\xi \mapsto \phi_\xi$ is continuous from $\Xi$ to $C(K)$, endowed with the supremum norm on $K$. Then the set-valued map $\text{Opt} : \Xi \rightrightarrows K$ defined by
    \begin{equation*}
        \text{Opt}(\xi) := \left\{x \in K \, \middle| \, \phi_\xi \text{ has a maximum at } x \right\}
    \end{equation*}
    is upper hemi-continuous.
\end{lemma}

\begin{proof}
    The result follows immediately from Berge's Maximum Theorem \cite[Theorem 17.31]{AlBo06} with $\xi \mapsto \text{image}_{\phi_\xi}(K)$ being the relevant set-valued map.
\end{proof}

\begin{remark} \label{remark:limsup_sets}
    In the proof below, we will make use of the notion of a $\limsup$ of sets. For a sequence of sets $(A_n)_{n \in \bN}$ denote
    \begin{equation*}
        \limsup_{n \rightarrow \infty} A_n = \bigcap_{n \in \bN} \bigcup_{m \geq n} A_m
    \end{equation*}
    to be interpreted as $x \in \limsup_{n \rightarrow \infty} A_n$ if and only if there are infinitely many $n \in \bN$ such that $x \in A_n$.
\end{remark}

The following proof is a variant of the proof of \cite[Lemma A.3]{CIL92} and \cite[Theorem 2.3.3]{CaSi04}.

\begin{proof}[Proof of Proposition \ref{proposition:Jensen_Alexandrov_cutoff}]

    For notational convenience, we will write $w = (x,y)$ and $w_0 = (x_0,y_0)$. Let $R > 0$ and $\{\zeta_{z,p}\}_{z \in E, p \in \bR^q} \subset C(E)$ and $\{\xi_{z}\}_{z \in E} \subset C(E)$ be two collections of functions as in Definition \ref{definition:perturbation_first_second_order}. Without loss of generality, we can assume that $R \geq \eta$.

    We start out by making $z_0$ the unique optimizer by replacing $\phi$ by
    \begin{equation}
        \widehat{\phi}(w) = \phi(w) - \varepsilon_1 \xi_{x_0} (x) - \varepsilon_2 \xi_{y_0} (y).
    \end{equation}
    Note that as $1 - (\varepsilon_1+\varepsilon_2) \kappa_\xi > 0$ the map $\widehat{\phi}$ is semi-convex and bounded from above with a unique optimizer $w_0$.

    \smallskip
    
    Our next step is to locally, linearly perturb $\widehat{\phi}$ to obtain $\phi_p$ as in equation \eqref{eqn:def:phi_p_eta}. This procedure produces a new optimizer close to $w_0$ in which the perturbed function $\phi_p$ is twice differentiable.

    To further facilitate the analysis of optimizers, we smoothen out $\phi$. To that end, let $C_\delta\colon C_b(E)\to C_b^2(E)$ be a mollifier with $\sup_{\delta>0} \vn{C_\delta f}<\infty$ and $C_\delta f\to f$ uniformly on compacts as $\delta \downarrow 0$. Define
    \begin{equation} \label{eqn:def:phi_eta_p_kappa}
        \phi_{p,\delta}(w) \coloneqq (C_\delta \phi) (w) - \eps_1 \left(\xi_{x_0} (x) + \zeta_{x_0, p_1} (x) \right) - \eps_2 \left( \xi_{y_0} (y) + \zeta_{y_0, p_2} (y) \right),
    \end{equation}
    where we will read $C_0 = \bONE$ such that $\phi_{p,0} = \phi_{p}$ and $\phi_{0,0} = \widehat{\phi}$.

    \smallskip
    
    We next study the optimizers for the map $(p,\delta) \mapsto \phi_{p,\delta}$ on $\Xi = \left(B_{1}(0) \times B_{1}(0)\right) \times [0,1]$ using Berge's Maximum Theorem with $K = \overline{B_R(w_0)}$. Set
    \begin{equation} \label{eqn:Berge1}
        \text{Opt}(p,\delta) \coloneqq \left\{ w \in \overline{B_{R}(w_0)} \, \middle| \, \phi_{p,\delta} \text{ has a local maximum at } w \in \overline{B_{R}(w_0)} \right\}.
    \end{equation}
    First note that the local nature of the problem can be removed due to the fact that the perturbations all vanish in $w_0$, whereas they add up to something negative outside the ball $\overline{B_R(w_0)}$ by Definition \ref{definition:perturbation_first_second_order} \ref{item:definition:penalization:domination}, implying that
    \begin{equation} \label{eqn:Berge2}
        \text{Opt}(p,\delta) = \left\{ w \in \overline{B_{R}(w_0)} \, \middle| \, \phi_{p,\delta} \text{ has a global maximum at } w  \right\}.
    \end{equation}
    
    Applying Lemma \ref{lemma:upper_hemi_continuity_maximum_map} to $(p,\delta) \mapsto \phi_{p,\delta}$ on $\Xi = \left(B_{\eta}(0) \times B_{\eta}(0)\right) \times [0,1]$ with $K  = \overline{B_R(w_0)}$, we find that the set-valued map $\text{Opt} \colon \Xi \rightrightarrows K \subseteq \bR^{q} \times \bR^q$, as defined above, is upper hemi-continuous in the variables $(p,\delta)$. We can thus find a closed set $U$ with $0$ in its interior satisfying
    \begin{equation} \label{eqn:definition_Jensen_U}
        U \subseteq B_{\eta}(0)\times B_{\eta}(0)
    \end{equation}
    and $\delta_0 > 0$ such that, if $p = (p_1,p_2) \in U$ and $\delta < \delta_0$, then
    \begin{equation} \label{eqn:Opt_p_kappa_insideball}
        \text{Opt}(p,\delta) \subseteq \text{Opt}(0,0) \oplus B_{\eta}(0) = B_{\eta}(w_0),
    \end{equation}    
    as the unique optimizer of $\widehat{\phi}$ is $w_0$.

    \smallskip

    We next aim to show that the set of such optimizers has positive Lebesgue measure $m$. Recall that $\kappa \coloneqq 1 - (\varepsilon_1 + \varepsilon_2)\kappa_\xi > 0$ is the semi-convexity constant of $\widehat{\phi}$. In particular, we will proceed to show the following steps. 
    \begin{itemize}
        \item[\textit{Step 1:}] For any $\delta \in (0,\delta_0)$ we have $m(\text{Opt}(U,\delta)) \geq |\kappa|^{-2d} m(U) > 0$.
        \item[\hskip1em\textit{Step 2:}] We take the limit $\delta \downarrow 0$ to obtain $m(\text{Opt}(U,0)) \geq |\kappa|^{-2d} m(U) > 0.$
    \end{itemize}
    \textit{Step 1}. By definition, all perturbations are at least once continuously differentiable on $\overline{B_R(w_0)}$.  It follows that for $p\in U$, $\delta \in (0,\delta_0)$ and $w \in \text{Opt}(p,\delta)$ we have that $D(C_\delta \phi)(z) = p$.
    This, in turn, implies that, for fixed $\delta \in (0,\delta_0)$,
    \begin{equation} \label{eqn:Jensen_lowerbound}
        U \subseteq \left(\text{Opt}(\cdot,\delta)\right)^{-1}(\text{Opt}(U,\delta)) \subseteq D(C_\delta \phi)(\text{Opt}(U,\delta)).
    \end{equation}

    We next argue towards a lower bound on the measure of $\text{Opt}(U,\delta)$ for $\delta \in (0,\delta_0)$. We exclude $\delta = 0$ here, due to the possible non-smoothness of $\phi$. As the convolution operator is taking averages, the semi-convexity of $\phi$ carries over to $C_\delta \phi$, which yields
    \begin{equation}\label{eqn:Jensen_upperbound}
        -\kappa I_{2d} \leq D^2 (C_\delta \phi)(w)
    \end{equation}
    for all $w \in E^2$. On the other hand, if $w \in \text{Opt}(U,\delta)$, we know that there is some $p \in U$ such that $w$ maximizes $\phi_{p,\delta}$, implying that $D^2(C_\delta \phi)(w) \leq 0$. Applying \eqref{eqn:Jensen_lowerbound}, the chain rule, and \eqref{eqn:Jensen_upperbound}, we thus obtain that, for any $\delta \in (0,\delta_0)$,
    \begin{multline*}
        m(U) \leq m(D(C_\delta\phi)(\text{Opt}(U,\delta))) \\
        = \int_{\text{Opt}(U,\delta)} \left|\det D^2 (C_\delta \phi)(w) \right| \dd w \leq m(\text{Opt}(U,\delta)) |\kappa|^{2q} 
    \end{multline*}
    leading to the lower bound
    \begin{equation} \label{eqn:first_lower_bound_vol_Jensen}
         0 < |\kappa|^{-2q} m(U) \leq m(\text{Opt}(U,\delta)),
    \end{equation}
    as $U$ has non-empty interior, establishing the claim of Step 1.
    
    \textit{Step 2}.  Next, we transfer our bound to $m(\text{Opt}(U,0))$. We first establish that
    \begin{equation} \label{eqn:Jensen_limsup_inclusion}
        \limsup_{\delta \downarrow 0} \text{Opt}(U,\delta) \subseteq \text{Opt}(U,0),
    \end{equation}
    see Remark \ref{remark:limsup_sets} for the definition of the $\limsup$ of sets. To that end, we pick an element $w \in \limsup_{\delta \downarrow 0} \text{Opt}(U,\delta)$. By definition we can find a sequence $\delta_n \downarrow 0$ such that $w \in \text{Opt}(U,\delta_n)$ for all $n \in \bN$. Then, there are $p_n \in U$ such that $w$ is an optimizer for $\left(\phi_{p_n,\delta_n}\right)_{n \in \bN}$. By the closedness of $U$ and \eqref{eqn:definition_Jensen_U}, $U$ is compact, and we can therefore extract a subsequence from $(p_n)_{n\in \N}$ that converges to some $p_0 \in U$. By upper semi-continuity of the map $(p,\delta) \mapsto \phi_{p,\delta}$, see Lemma \ref{lemma:upper_hemi_continuity_maximum_map}, we find that $w$ maximizes $\phi_{p_0}$, or in other words, $w \in \text{Opt}(U,0)$.

    Thus, by \eqref{eqn:Jensen_limsup_inclusion}, it suffices to lower bound the volume of $\limsup_{\delta \downarrow 0} \text{Opt}(U,\delta)$. As a first step, note that  \eqref{eqn:first_lower_bound_vol_Jensen} leads to
    \begin{equation*}
        m\left(\bigcup_{\delta' \leq \delta} \text{Opt}(U,\delta') \right) \geq |\kappa|^{-2q} m(U)
    \end{equation*}
    for any $\delta \in (0,\delta_0)$. Consequently, as
    \begin{equation*}
        \limsup_{\delta \downarrow 0} \text{Opt}(U,\delta) = \bigcap_{\delta \in (0,\delta_0)} \bigcup_{\delta' \leq \delta} \text{Opt}(U,\delta'),
    \end{equation*}
    by continuity from above of the Lebesgue measure $m$,
    %is thus the decreasing limit as $\delta \downarrow 0$,
    we find that
    \begin{equation*}
        m\left(\limsup_{\delta \downarrow 0}  \text{Opt}(U,\delta)\right) = \lim_{\delta \downarrow 0} m\left(\bigcup_{\delta' \leq \delta} \text{Opt}(U,\delta') \right) \geq |\kappa|^{-2q} m(U).
    \end{equation*}
    By \eqref{eqn:Jensen_limsup_inclusion}, we conclude that
    \begin{equation*}
        m(\text{Opt}(U,0)) \geq |\kappa|^{-2q} m(U) > 0,
    \end{equation*}
    establishing the claim of Step 2.

    \smallskip

    We proceed by verifying that we can now find $p \in U$ with an optimizer $z_1$ in $B_\eta(z_0)$ in which $\phi_{p}$ is twice differentiable.

    First of all, recall that, by \eqref{eqn:Opt_p_kappa_insideball}, we have
    \begin{equation*}
        \text{Opt}(U,0) \subseteq B_{\eta}(z_0).
    \end{equation*}
    Furthermore, by Alexandrov's theorem \cite[Theorem 2.3.1]{CaSi04}, the set of points in $B_{\eta}(z_0)$ where $\phi_p$ is twice differentiable has full measure. As the measure of $\text{Opt}(U,0)$ is positive, it follows that there exist $z_1 \in B_{\eta}(z_0)$ and $p \in U$ such that $\phi_p$ is twice differentiable in $z_1$ and has a local maximum at $z_1$ in $\overline{B_R(z_0)}$. Finally, recall from \eqref{eqn:Berge2} that the local optimizer is in fact a global optimizer. This establishes the claim.
\end{proof}

\begin{proof}[Proof of Corollary \ref{corollary:Jensen_optimizerbound}]
        
    We stay in the context of Proposition \ref{proposition:Jensen_Alexandrov_cutoff} and proceed with the proof of \eqref{eqn:Jensen_control_optimizationproblem}. By construction, we have
    \begin{equation} \label{eqn:Jensen_controloptimization1}
        \ssup{\phi} = \phi(x_0,y_0) = \phi_{p,\varepsilon}(x_0,y_0) \leq \ssup{\phi_{p,\varepsilon}} = \phi_{p,\varepsilon}(x_1,y_1).
    \end{equation}
    This implies the lower bound of \eqref{eqn:Jensen_control_optimizationproblem_onlyPerturb}. Note that by the properties of $\xi_{x_0},\xi_{y_0}, \zeta_{x_0,p_1}$ and $\zeta_{y_0,p_2}$ we have
    \begin{align*}
        -\eps_1 \left( \xi_{x_0}(x_1) + \zeta_{x_0,p_1}(x_1) \right) - \eps_2 \left( \xi_{y_0}(y_1) + \zeta_{y_0,p_2}(y_1) \right) &\leq \eps_1 |p_1| d(x_0,x_1) + \eps_2 |p_2| d(y_0,y_1)\\
        &\leq \eps_1 \eta + \eps_2 \eta,
    \end{align*}
    leading to the upper bound of \eqref{eqn:Jensen_control_optimizationproblem_onlyPerturb}. Consequently,
    \begin{equation}\label{eqn:Jensen_controloptimization2}
        \begin{aligned}
             \phi_{p,\varepsilon}(x_1,y_1) & \leq  \phi(x_1,y_1) + \eps_1 \eta + \eps_2 \eta \\
             &\leq \ssup{\phi} + \eps_1 \eta + \eps_2 \eta.
        \end{aligned}
    \end{equation}
    Combining \eqref{eqn:Jensen_controloptimization1} and \eqref{eqn:Jensen_controloptimization2}, finally yields \eqref{eqn:Jensen_control_optimizationproblem}.
\end{proof}

\section{Smooth test function construction}

The main result of this section is Lemma \ref{lemma:smooth_test_function_construction}, in which we construct a smooth test function that lies between a function that is twice differentiable in one point and a perturbed version of that function.

\begin{lemma}\label{lemma:smooth_test_function_construction}
    Let $\Pi_1$, $\Pi_1^0$, $\Pi_2$, and $\Pi_2^0$ be as in the proof of Proposition \ref{proposition:test_function_construction}.

    Then, there exist $f_1, f_2 \in C^\infty(E)$ such that, for all $y \in E$,
    \begin{align*}
        \Pi_1 (y) \leq &f_1(y) \leq \Pi_1^0 (y), \\
        \Pi_2 (y) \geq &f_2(y) \geq \Pi_2^0 (y)
    \end{align*}
    with equality only in $y_\alpha$ and $y'_\alpha$, respectively.
\end{lemma}

\begin{proof}
    As in the proof of Proposition \ref{proposition:test_function_construction}, we only consider the case
    \begin{equation}
        \Pi_1 (y) \leq f_1(y) \leq \Pi_1^0 (y),
    \end{equation}
    for $y \in E$ with equality only in $y_\alpha$, since the other statement follows analogously.

    Our goal is to find $f_1$, by first constructing a function that is squeezed between $\Pi_1$ and $\Pi_1^0$, using the Whitney Extension Theorem \cite[Theorem 2.3.6]{MR1996773}, and then modifying it to obtain $f_1$.

    Recall that, by construction, we have that
    \begin{equation*}
        \Pi_1(y) < \Pi_1^0(y) \qquad \text{for } y \in E \setminus \{ y_\alpha \} 
    \end{equation*}
    and 
    \begin{equation}
        \Pi_1(y_\alpha) = \Pi_1^0(y_\alpha), \quad 
        D\Pi_1(y_\alpha) = D\Pi_1^0(y_\alpha), \quad
        D^2\Pi_1(y_\alpha) < D^2\Pi_1^0(y_\alpha).
    \end{equation}
    We apply the Whitney Extension Theorem to $\frac{1}{2}(\Pi_1 + \Pi_1^0)$ on the closed set $A = \{y_\alpha\}$, yielding a function $\psi_1 \in C^2(E)$ such that $\Pi_1 \leq \psi_1 \leq \Pi_1^0$ on $B_{2\delta}(y_\alpha)$ for some $\delta>0$
    %an open neighborhood $\cN$ of $y_\alpha$ with 
    with equality only in $y_\alpha$. Inspecting the construction of $\psi_1$ in the proof of \cite[Theorem II]{Wh34}, we find that $\psi_1 \in C^\infty(E)$.

    Next, we modify $\psi_1$ such that the resulting function stays between $\Pi_1$ and $\Pi_1^0$ on all of $E$.
    As smooth functions are dense in the set of continuous functions, we can find a function $\psi_2 \in C^\infty(E)$ such that $\Pi_1 < \psi_1 < \Pi_1^0$ on $E \setminus B_\delta(y_\alpha)$.

    Then, defining
    \begin{equation*}
        f_1 (y) = \ell (y)\psi_1 (y) + (1-\ell(y))\psi_2 (y),
    \end{equation*}
    where $\ell$ is a smooth function that is $1$ on $B_\delta(y_\alpha)$ and $0$ outside of $B_{2\delta}(y_\alpha)$, 
    %i.e. on $E \setminus (B_\delta(y_\alpha))$ for $\delta >0$, 
    for example $\ell$ as defined as point $(3)$ on \cite[p. 33]{MR0267467}. This concludes the proof.
\end{proof}

\section{Convergence of integrals}

\begin{lemma}\label{lemma:converging_integrals}
 Let $\cX$ be a Polish space, $W\colon \cX\to (0,\infty)$ be a continuous function, and $\nu_n,\nu_\infty$ be non-negative Borel measures with $\int_{\cX} W \,\dd \nu_n<\infty$ for all $n\in \N$ and
 \begin{equation}\label{eq:AppendixC_Lemma}
 \lim_{n\to \infty}\int \phi\,\dd\nu_n= \int \phi\,\dd\nu_\infty\in \R 
 \end{equation}
 for every function $\phi\in C(\cX)$ with $|\phi(x)|\leq W(x)$ for all $x\in \cX$. Moreover, let $\phi_n,\phi_\infty\in C(\cX)$ with $\phi_n\to \phi_\infty$ uniformly on compacts and $\sup_{n\in \N}\sup_{x\in \cX}\frac{|\phi_n(x)|}{W(x)}<\infty$.
  Then,
    \begin{equation*}
        \lim_{n \rightarrow \infty} \int \phi_n \, \dd \nu_n = \int \phi_\infty \, \dd \nu_\infty.
    \end{equation*}
\end{lemma}

\begin{proof}
 By assumption, the family $\mu_n:=W\dd \nu_n$ satisfies $C_\mu:=\sup_{n\in \N}\mu_n(\cX)<\infty$ and $$C_\phi:=\sup_{n\in \N}\sup_{x\in \cX}\frac{|\phi_n(x)-\phi_\infty(x)|}{W(x)}<\infty.$$ Using the fact that a function $\phi\in C(\cX)$ satisfies $|\phi(x)|\leq W(x)$ for all $x\in \cX$ if and only if $\phi= W \psi$ for some $\psi\in C_b(\cX)$, it follows that $\mu_n\to \mu_\infty := W\dd\nu_\infty$ weakly. In particular, the family $(\mu_n)_{n\in \N}$ is tight. Hence, for all $\eps >0$, there exists a compact set $K_\eps\subseteq \cX$ such that
 \begin{equation*}
C_\phi\mu_n\big(\cX\setminus K_\eps\big)<\frac{\epsilon}{3} \quad\text{for all }n\in \N.
 \end{equation*}
 Now, let $\eps>0$. By \eqref{eq:AppendixC_Lemma} and since $\phi_n\to \phi_\infty$ uniformly on compacts and $W$ is continuous, there exists some $n_0\in \N$ such that
 \begin{equation*}
  C_\mu\sup_{x\in K_\eps}\frac{|\phi_n(x)-\phi_\infty(x)|}{W(x)}<\frac{\eps}{3}\quad\text{and}\quad \bigg|\int \phi_\infty \,\dd\nu_n-\int \phi_\infty \,\dd\nu_\infty\bigg|<\frac{\eps}{3}.
 \end{equation*}
 We thus obtain that
 \begin{align*}
  \bigg|\int\phi_n\,\dd \nu_n-\int \phi_\infty\,\dd \nu_\infty\bigg|&\leq \int \big|\phi_n-\phi_\infty\big|\,\dd \nu_n+\bigg|\int\phi_\infty\,\dd \nu_n-\int \phi_\infty\,\dd \nu_\infty\bigg| \\
  & \leq \int_{K_\eps}\big|\phi_n-\phi_\infty\big|\,\dd \nu_n+ \int_{\cX\setminus K_\eps}\big|\phi_n-\phi_\infty\big|\,\dd \nu_n+\frac{\eps}{3}\\
  &\leq C_\mu \frac{|\phi_n(x)-\phi_\infty(x)|}{W(x)}+C_\phi \mu_n\big(\cX\setminus K_\epsilon\big)+\frac{\eps}{3}<\eps
 \end{align*}
 for all $n\in \N$ with $n\geq n_0$. The proof is complete.
\end{proof}

\section{Proofs of auxiliary results}

\subsection{Equivalent characterization of the definition of viscosity solutions}\label{section:def_equiv}

\begin{lemma}\label{lemma:def_equiv}
    Let $H_1 \subseteq C_l(E) \times C(E)$ and $H_2 \subseteq C_u(E) \times C(E)$ be two operators with domains $\cD(H_1)$, $\cD(H_2)$. Moreover, let $\lambda>0$ and $h_1\in C_l(E)$ and $h_2\in C_u(E)$.
    \begin{enumerate}[(a)]
        \item\label{item:reduce_subsol} Let $u\colon E\to \R$ be $u$ a viscosity subsolution to \eqref{eqn:HJ_subsolution}. Suppose $\delta > 0$ and $(f,g) \in H_1$ are such that $\{x \in E \, | \, u(x)-f(x) \geq \ssup{u-f} - \delta\}$ is compact. Then there exists some $x_0 \in E$ with
        \begin{equation} \label{eqn:strong_subsol}
            \begin{split}
            u(x_0) - f(x_0)  = \sup_{x \in E} u(x) - f(x), \\
            u(x_0) - \lambda g(x_0) \leq h_1(x_0).
            \end{split}
        \end{equation} 

        \item\label{item:reduce_supersol} Let $v\colon E\to \R$ be a viscosity supersolution to \eqref{eqn:HJ_supersolution}. Suppose $\delta > 0$ and $(f,g) \in H_2$ are such that $\{x \in E \, | \, v(x) - f(x) \leq \iinf{v-f} + \delta\}$ is compact. Then there exists some $x_0 \in E$ with
        \begin{equation} \label{eqn:strong_supersol}
            \begin{split}
                v(x_0) - f(x_0)  = \inf_{x \in E} v(x) - f(x), \\
                v(x_0) - \lambda g(x_0) \geq h_2(x_0).
            \end{split}
        \end{equation}
    \end{enumerate}
    In particular, the outcomes of \ref{item:reduce_subsol} and \ref{item:reduce_supersol} hold if $H_1 \subseteq C_+(E) \times C(E)$ and $H_2 \subseteq C_-(E) \times C(E)$.
\end{lemma}

\begin{proof}
    We only show Part \ref{item:reduce_subsol}. Part \ref{item:reduce_supersol} follows analogously. Assume that $u$ is a viscosity subsolution to \eqref{eqn:HJ_subsolution} and let $(f,g)\in H_1$. We aim to establish the existence of $x_0$ such that \eqref{eqn:strong_subsol} is satisfied.

    \smallskip

    Due to the subsolution property of $u$, there exists a sequence $(x_n)_{n\in \N}\subset E$ such that
    \begin{gather*}
       \lim_{n \rightarrow \infty} u(x_n) - f(x_n) = \ssup{u - f}, \\
       \limsup_{n \rightarrow \infty} u(x_n) - \lambda g(x_n) - h_1(x_n) \leq 0.
    \end{gather*}
    For $n$ large, we have
    \begin{equation*}
        x_n \in \{x \in E \, | \, u(x) - f(x) \geq \ssup{u-f} - \delta\},
    \end{equation*}
    which is a compact set by assumption. Thus, there exists a subsequence $(x_{n_k})_{k\in \N} \rightarrow x_0 \in E$. Since $u(x_n) - f(x_n) \rightarrow \ssup{u-f}$, it follows that
    \begin{equation*}
        \ssup{u-f} =\lim_{k\to \infty} u(x_{n_k}) - f(x_{n_k})\leq u(x_0)-f(x_0) \leq \ssup{u-f},
    \end{equation*}
    where the inequality follows by the upper semi-continuity of $u-f$. The inequality is thus an equality, establishing the first statement of \eqref{eqn:strong_subsol}. Due to the continuity of $f$, we additionally find that 
    \begin{equation*}
        u(x_0)=\lim_{k\to \infty} u(x_{n_k}).
    \end{equation*}
    Since $g$ and $h_1$ are continuous, we conclude
    \begin{align*}
        u(x_0)-\lambda g(x_0)-h_1(x_0)&= \lim_{k\to \infty} u(x_{n_k})-\lambda g(x_{n_k})-h_1(x_{n_k})\\
        &\leq \limsup_{n\rightarrow\infty} u(x_{n})-\lambda g(x_{n})-h_1(x_{n})\leq 0,
    \end{align*} 
    establishing the second statement of \eqref{eqn:strong_subsol}.
\end{proof}

\subsection{Proof of Lemma \ref{lemma:properties_supinf_convolution}} \label{appendix:proof_supconvolution_properties}

\begin{proof}
    For the proof of \ref{item:properties_supinf_convolution:bounded}, note that, for any $x, y \in E$, we have
    \begin{equation*}
        u(x) - \frac{\alpha}{2}d^2(x,y) \leq u(x).
    \end{equation*}
    This implies that
    \begin{equation} \label{eqn:convolution_bounded1}
        \ssup{P^\alpha[u]} = \ssup{ u - \frac{\alpha}{2}d^2 } \leq \ssup{u}.
    \end{equation}
    On the other hand, we have
    \begin{equation*}
        u(y) \leq \ssup{u - \frac{\alpha}{2}d^2(\cdot,y)} = P^\alpha[u](y).
    \end{equation*}
    It follows that
    \begin{equation} \label{eqn:convolution_bounded2}
        \iinf{u} \leq \iinf{P^\alpha[u]}.
    \end{equation}
    Now, \ref{item:properties_supinf_convolution:bounded} follows by \eqref{eqn:convolution_bounded1} and \eqref{eqn:convolution_bounded2}.
    Part \ref{item:properties_supinf_convolution:convergence_optimizers} is equivalent to
    \begin{equation*}
        P_\alpha [u] \leq u \leq P^\alpha [u] \quad \text{on } E,
    \end{equation*}
    which is immediately clear from the definitions of sup- and inf-convolutions.
    Part \ref{item:properties_supinf_convolution:decreasing} follows similarly from the definitions.
    For the proof of \ref{item:properties_supinf_convolution:semi_convex}, let $y_0 \in E$. Then, since $d$ is the Euclidean metric, we find
    \begin{equation*}
        P^\alpha[u](y) + \frac{\alpha}{2} d^2(y,y_0) = \ssup{ u + \alpha \ip{y-y_0}{\cdot-y_0} - \frac{\alpha}{2}d^2(\cdot,y_0) },
    \end{equation*}
    where the right-hand side is convex as it is a supremum over affine functions.
    By Proposition 2.1.5 and Theorem 2.1.7 of \cite{CaSi04} the claim follows.
    Lastly, \ref{item:properties_supinf_convolution:optimizers_differentiability} follows from Theorem 3.4.4 of \cite{CaSi04} by noting that the sets over which can be optimized are compact due to the boundedness of $u$ and $v$.
\end{proof}

\printbibliography
 
\end{document}